\documentclass[final,leqno,letterpaper]{etna}
\usepackage[linesnumbered,lined,boxed,commentsnumbered]{algorithm2e}

\setbibdata{1}{xx}{46}{2017} 
\usepackage[english]{babel}

\usepackage{mathrsfs} 

\usepackage{amsmath}   
\usepackage{amsfonts}  

\newcommand{\daJKO}{daJKO\ }

\DeclareMathOperator*{\argmin}{arg ~min}         %

\newcommand{\B}{\mathcal{B}}

\newcommand{\N}{\mathbb{N}}

\newcommand{\R}{\mathbb{R}}

\newcommand{\dx}{\,\textup{d}x}
\newcommand{\du}{\,\textup{d}u}
\newcommand{\dt}{\,\textup{d}t}
\newcommand{\dd}[1]{\,\textup{d}{#1}}

\usepackage{amsmath}
\usepackage{graphicx}
\hypersetup{%
    pdftitle={Data driven gradient flows},
    pdfauthor={Jan-F. Pietschmann, M. Schlottbom},
    }

\title{Data driven gradient flows\thanks{%
Received... Accepted... Published online on... Recommended by....
}}

\author{Jan-Frederik Pietschmann\footnotemark[2]
        \and Matthias Schlottbom\footnotemark[3]}

\shorttitle{Data driven gradient flows} 
\shortauthor{J.-F.~PIETSCHMANN AND M.~SCHLOTTBOM}

\begin{document}
\maketitle

\renewcommand{\thefootnote}{\fnsymbol{footnote}}

\footnotetext[2]{Fakult\"at f\"ur Mathematik, Technische Universit\"at Chemnitz Reichenhainer Stra{\ss}e 41, 09126 Chemnitz, Germany. jfpietschmann@math.tu-chemnitz.de}
\footnotetext[3]{Department of Applied Mathematics, University of Twente, Postbus 217, 7500 AE Enschede, Netherlands. m.schlottbom@utwente.nl}

\begin{abstract}
We present a framework enabling variational data assimilation for gradient flows in general metric spaces, based on the minimizing movement (or Jordan-Kinderlehrer-Otto) approximation scheme. After discussing stability properties in the most general case, we specialise to the space of probability measures endowed with the Wasserstein distance. This setting covers many non-linear partial differential equations (PDEs), such as the porous medium equation or general drift-diffusion-aggregation equations, which can be treated by our methods independent of their respective properties (such as finite speed of propagation or blow-up).
We then focus on the numerical implementation of our approach using an primal-dual algorithm. The strength of our approach lies in the fact that by simply changing the driving functional, a wide range of PDEs can be treated without the need to adopt the numerical scheme. We conclude by presenting detailed numerical examples.
\end{abstract}

\begin{keywords}
gradient flows, variational data assimilation, primal-dual algorithms, Wasserstein distance
\end{keywords}

\begin{AMS}
49N45 (Inverse problems in optimal control), 35R30 (Inverse problems for PDEs), 49N45 (Inverse problems in optimal control), 49Q22 (Optimal transportation)
\end{AMS}




\section{Introduction}
The aim of this work is to develop an data assimilation approach for gradient flows (GFs) in metric spaces. In  particular we are interested in time-dependent nonlinear partial differential equations that are GFs in the space of probability measures endowed with the 2-Wasserstein metric. Data assimilation is useful in this context if the model is not completely specified, for instance if the initial datum is unknown or certain model parameters are not exact. 
%

%
Our approach combines a time-discrete variational scheme, the famous Jordan-Kinderlehrer-Otto (JKO) or minimizing movement scheme, with additional data fidelity terms in order to  overcome model uncertainties and to improve the solution in each time step.

Let us briefly describe our idea in the finite dimensional case. Given $u_0 \in \R^d$, $F \in C^{1,1}(\R^d; \R)$, we consider the corresponding gradient flow
\begin{equation}\label{eq:grad_flow_classic_Rd}
\begin{aligned}
    u'(t) &= -\nabla F(u(t)), \quad t> 0,\\
    u(0) &= u_0.
    \end{aligned}
\end{equation}
For a fixed time step $\tau > 0$ and initial condition $u_0$, we may recursively define a sequence $(u^\tau_k)_{k\in\N}$ using the implicit Euler scheme, which amounts to solving the following non-linear system of equations
\begin{align}\label{eq:implicitEuler}
    u_{k+1}^\tau = u_k^\tau - \tau \nabla F(u^\tau_{k+1}), \quad k=0,1,\ldots
\end{align}
A well-known, but crucial, observation is that \eqref{eq:implicitEuler} is the first order optimality condition of the following minimization problem
\begin{align}\label{eq:JKO}
u_{k+1}^\tau = \argmin_{u\in\R^d} ( F(u) + \frac{1}{2\tau}\|u-u^\tau_k\|_2^2 ),
\end{align}
which is a variational interpretation of the implicit Euler time-stepping scheme.
%
%

In practically relevant situations, the model might be specified incompletely, that is the functional $F$ or the initial condition $u_0$ are not known or known with uncertainties.
In such situations, we may, however, have access to data points $y_k$ that are obtained from the (true) state $u$ by applying a linear measurement operator $\B$, i.e.,
$$
y_{k} = \B u(t_k) +\eta_k, \quad t_k=k\tau, \quad k=1,2,\ldots,
$$
where $\eta_k$ models (possibly random) measurement errors.
Motivated by the methodology of variational data assimilation, we use this information to modify the objective functional in \eqref{eq:JKO} by adding an additional \emph{data fidelity term} as follows
\begin{align}\label{eq:JKO_data}
u_{k+1}^\tau = \argmin_{u}  F(u) + \frac{1}{2\tau}\left(\|u-u^\tau_k\|_2^2+ \frac{\|\mathcal{B}(u)-v_{k+1}\|^2}{\theta} \right) .
\end{align}
Here, $\theta \in (0,\infty)$ is a weighting parameter. 
In the limit $\theta \to 0^+$, we formally obtain $u^\tau_{k+1}$ s.t. $\B(u^\tau_{k+1}) = v_{k+1}$, i.e., the flow is completely data driven; while for $\theta \to + \infty$, the data has no more influence on $u^\tau_{k+1}$. 
We note that \eqref{eq:JKO_data} is local in time in the sense that only estimations or observations for the next time step are required. Such an approach is usually referred to as  \emph{filtering problem} in the data assimilation literature, \cite{Law2015_data_assimilation_book}.
Equation \eqref{eq:JKO_data} allows for different interpretations: We may think of performing the JKO scheme for the functional $F$ in the modified, data induced distance measure
$$
d_{\mathrm{dat}}(u,u_k) := \|u-u^\tau_k\|_2^2+ \frac{\|\mathcal{B}(u)-v_{k+1}\|^2}{\theta}.
$$
Clearly, $d_{\mathrm{dat}}(u,u_k)$ is no longer a metric as neither $d_{\mathrm{dat}}(u,u) =  0$ nor does the triangle inequality hold for the first argument. 
The second interpretation, which we will exploit in the actual computations later, is a JKO scheme in the original metric but for the modified (and now $\tau$- and data-dependent) functional
$$
E_{\tau,\mathrm{dat}} = F(u) + \frac{1}{2\tau}\frac{\|\mathcal{B}(u)-v_{k+1}\|^2}{\theta}.
$$
Finally, we can also interpret \eqref{eq:JKO_data} as Tikhonov regularization with data term $\|\mathcal{B}(u)-v_{k+1}\|^2$,   regularizer  $(\tau F+\|u-u_k^\tau\|_2^2/2)$, and $\theta$ takes the role of a regularization parameter. The latter interpretation might be useful for the analysis of the minimization problem \eqref{eq:JKO_data}.
Another interesting remark is that if one chooses $\theta$ such that $\theta\tau=\mathrm{const}$, one recovers a nudging method \cite{Anthes1974_NudgingFirst}, which is an example of a four-dimensional data assimilation method, see Remark~\ref{rem:control_interpretation} for further details.

The connection to non-linear evolution partial differential equations is due to the seminal works of Otto \cite{Otto1998FirstGradFlow,Otto2001GeometryPME} and Otto, Jordan and Kinderlehrer \cite{JKO1998}. They showed that replacing the $L^2$-Norm in \eqref{eq:JKO} by the $2$-Wasserstein metric $W_2$ over the space of probability measures $P(\Omega)$ over some set $\Omega \subset \R^d$, the resulting iterates converge to solutions to equations of the form 
\begin{equation}\label{eq:pde_grad_flow}
\begin{aligned}
    \partial_t u &= \nabla \cdot ( u \nabla \frac{\delta F}{\delta u}),\quad \Omega \times (0,T),\\
    u(0) &= u_0\text{ in } \Omega,
\end{aligned}
\end{equation}
for some final time $T$ and supplemented with no flux boundary conditions
\begin{align*}
     u \nabla \frac{\delta F}{\delta u}\cdot n = 0, \text{ on } \partial\Omega \times (0,T),
\end{align*}
with $n$ denoting the outward normal vector.
Typical functionals $F$ of interest are of the form
\begin{align}\label{eq:functional_F}
F(u) = \int_\Omega U \left(\frac{du}{dx}(x)\right) \dx + \int_\Omega V(x) \du(x) + \frac{1}{2}\int_\Omega\int_\Omega W(x,y)\du(x)\du(y),
\end{align}
where $\frac{du}{dx}$ denotes the Radon-Nikodym derivative with respect to the Lebesgue measure and the convention is to set first integral to $+\infty$ if $u$ is not absolutely continuous. In \eqref{eq:functional_F}, $U:\R\to \R$ denotes the internal energy, $V:\Omega \to \R$ a potential and $W:\Omega\times\Omega \to \R $ an interaction potential.
In this setting, our approach becomes
\begin{align}\label{eq:daJKO_intro}
u_{k+1}^\tau = \argmin_{u} ( F(u) + \frac{1}{2\tau}\left(W_2^2(u,u^\tau_k) + \frac{d_D^2(\mathcal{B}(u),v_{k+1}^\delta)}{\theta} \right),
\end{align}
with $d_D$ a data fidelity metric. We call \eqref{eq:daJKO_intro} \daJKO - scheme in the following. 
Before stating the contributions of this work, let us briefly review related work.

%
\subsection*{Related work}\mbox{}

%
\noindent \emph{$W_2$ gradient flows}: The interpretation of \eqref{eq:pde_grad_flow} as Wasserstein gradient can be used to show existence of solutions via the limit $\tau \to 0$ in \eqref{eq:JKO} but also to obtain results on the long time behaviour of solutions, see for example \cite{Carrillo2003,Markowich99onthe}. More recently there are efforts of understanding a larger class of evolution equations as gradient flows, e.g. equations on discrete state spaces, \cite{MaasGradFlowEntropyFiniteMarkov2011,ChowFokkerPlanckMarkovGraph2012}, graph structures \cite{EPSS2021}, equations featuring non-local terms such as the Boltzmann \cite{Erbar2016AGF} or aggregation equation \cite{Esposito2021}. Another popular research direction are equations with reaction terms or non-homogeneous boundary conditions where the total mass is not conserved in time and the Wasserstein distance has to be replaced by Fisher-Rao-type distances that are finite even for two measures having different mass \cite{Mielke_2011,Burger2020}.
A detailed overview of the general theory in metric spaces can be found in \cite{AmbrosioGigliSavare2008}.
\emph{Computational optimal transport}: To implement our \daJKO scheme, we need to be able to efficiently compute the $2$-Wasserstein distance. In recent years, many algorithms for solving this optimal transport problem have been developed. Different approaches were taken, e.g. based on the Monge-Amp\`ere equation \cite{Benamou2014_MongeAmpere}, for semi-discrete problems where the target space is finite \cite{Kitagawa2019_semidiscrete}, or scaling algorithms for regularised problems obtained by adding a entropic term to the static Wasserstein functional \cite{Cuturi2013_lightspeed,Peyre2019_book}.
We focus on a class of algorithms that is based on the dynamic formulation of optimal transport due to Benamou and Brenier \cite{BenamouBrenier2000}, see also \eqref{eq:Benamou_Brenier} below. For the approximation of the Wasserstein distance, the original paper \cite{BenamouBrenier2000} used an augmented Lagrangian (ALG) approach to enforce the divergence constraints, which has been extended in \cite{Benamou2016_ALG2-JKO} to a convex formulation for the corresponding JKO scheme. 
The ALG approach also been adopted to finite volumes in \cite{Cances2020_FVSchemeJKO}, see also \cite{Cances2019_MultiphaseJKO} for an application to multiphase problems. 
In \cite{Papadakis2014} proximal splitting algorithms are combined with staggered grid discretizations to enforce the divergence constraints; proximal methods have also been combined with finite element discretizations \cite{Albi_2017} for the discretization of $L^2$-gradient flows.
Starting also from the dynamic formulation of the Wasserstein distance, \cite{Carrillo2022_PrimalDual} employ piecewise constant approximations and proximal splitting algorithms to discretize the JKO scheme. Notably, in \cite{Carrillo2022_PrimalDual}, the divergence constraints are enforced only in a relaxed sense. The resulting formulation then involves the minimization of a sum of three convex functionals, which is done by the algorithm developed in \cite{Yan2018}.

\emph{Data assimilation}: Our approach is closest to $3$D variational data assimilation methods (3DVar) in that is relies on minimizing a functional containing a data term in each time step. Being purely deterministic it is, however, different from methods which aim to obtain posterior probability distributions, see \cite{Law2015_data_assimilation_book} for an overview. As mentioned above, we focus on the filtering problem to improve the quality of the numerical solutions based on measurements at the current time step. The Wasserstein metric recently started to become an important tool in the analysis of particle filters, \cite{Reich2013_Wasserstein}. There are also several works which use it in a variational setting which is closer to our situation \cite{Tamang2020_Wasserstein}. Performing data assimilation in general metric space does not seem to have been considered so far, but see \cite{Tamang2021_Metric} for a related problem. 
Let us also mention the connection to minimisation based formulations of inverse problems, where both model and data fidelity term a combined in a variational approach \cite{kaltenbacher2018minimization}. Finally, let us remark that \eqref{eq:JKO_data} is, at least in the Euclidean setting, connected to an closed loop control technique called \emph{nudging}, \cite{Anthes1974_NudgingFirst}, see Remark~\ref{rem:control_interpretation} for further details. 

%
\subsection*{Contributions}
In this work we 
\begin{itemize}
    \item Introduce a methodology for variational data assimilation that can be applied to gradient flows in metric spaces.
    \item Consider the particular case of the space of probability measures endowed with the $2$-Wasserstein distance, which includes many time-dependent nonlinear partial differential equations with linear mobility. We provide a uniform framework mostly independent of the properties (such as regularity, finite time blow-up, etcc.) of the specific PDE.
    \item Introduce a numerical algorithm  based on the recent work by \cite{Carrillo2022_PrimalDual}, but, instead of a piecewise constant approximation, uses a finite difference approximation to discretize the divergence constraints. Similar to \cite{Carrillo2022_PrimalDual}, we obtain a positivity preserving scheme that avoids the solution of linear systems. This again yields a uniform method that can be applied to a large class of non-linear PDEs without changing the algorithm. 
    \item Test our method numerically on different examples of non-linear PDEs.
\end{itemize}

The paper is structured as follows: In Section \ref{sec:prelim}, we present some background on optimal transport, the $p$-Wasserstein distances and their dynamic formulation as well as details on the variational scheme in general metric spaces. Section \ref{sec:daJKO} presents the data driven variational approach and some of its properties while Section \ref{sec:numerics} is devoted to a detailed description of the numerical scheme. Finally, Section \ref{sec:examples} contains several numerical studies.

\section{Preliminaries}\label{sec:prelim}

We call $\Omega \subset \R^d$ an open, bounded and connected set. By $P(\Omega)$ we denote the space of probability measures. 

\subsection{Optimal transport} 
Given $u_0,u_1 \in P(\Omega)$, optimal transport aim at finding a transport plan $\pi \in P(\Omega \times \Omega)$ that realizes the transport of $u_0$ onto $u_1$ as minimal cost with respect to a cost functional $c:\Omega \times \Omega \to \R_+$. Formally speaking, $c(x,y)\dd{\pi}(x,y)$ is the cost for transporting an infinitesimal amount of mass from $x$ to $y$. Evaluating the total cost at the optimal transport plan gives rise to a distance between probability measures, the $p$-Wasserstein distance is defined as 
\begin{align}\label{eq:Wasserstein_static}
    W_p^p(u_0, u_1) = \inf_{\pi \in \Pi(u_0,u_1)}\int_{\Omega\times \Omega}\frac{1}{p} |x-y|^p \;\dd{\pi}(x,y),
\end{align}
where the infimum is taken over all $\pi\in P(\Omega \times \Omega)$ having marginals $u_0$ and $u_1$, respectively, i.e.,
\begin{align*}
    \Pi(u_0,u_1) = \{ \pi \in P(\Omega \times \Omega) \; : \; &\pi(A,\Omega) = u_0(A)\text{ and } \pi(\Omega, B) = u_1(B)\\ &\text{for all measurable } A,\,B \subset \Omega \}.
\end{align*}
With this distance, $(P(\Omega),W_p)$ becomes a metric space and convergence with respect to $W_p$ metrizes weak-$*$ convergence of measures. 
Let us remark that in the space $(P(\Omega), W_p)$, constant speed geodesics are curves $\rho : [0,1] \to P(\Omega)$ satisfying
\begin{align*}
    W_p(\rho(s), \rho(t)) = (t-s)W_p(\rho(0),\rho(1))\quad \forall\, 0 \le s \le t \le 1.
\end{align*}
Given two measures $u_0,\,u_1 \in P(\Omega)$, a constant speed geodesic connecting them can be constructed as follows: Denoting $\pi \in \Pi(u_0,u_1)$ an optimal transport plan from $u_0$ to $u_1$, the constant speed geodesic is defined by
\begin{align}\label{eq:geodesic_interpolation}
\rho(t) = ((1-t)x + ty)_{\#}\pi.
\end{align}

In the following, we will restrict ourselves to the case $p=2$, and introduce a dynamic version of the Wasserstein distance due to Benamou and Brenier, \cite{BenamouBrenier2000}, which will later serve as a basis for our numerical scheme. The idea to select, among all absolutely continuous (with respect to $W_2$) curves $\rho:[0,1] \to P(\Omega)$ connecting two given measures $u_0,u_1 \in P(\Omega)$, the one which has the least kinetic energy.
%
For a rigorous formulation, we introduce the convex lower-semicontinuous function \cite{BenamouBrenier2000}
\begin{align*}
    \Phi(\rho,m):=\left\{\begin{array}{ll}
\frac{|m|^{2}}{2 \rho} & \text { if } \rho>0, \\
0 & \text { if }(\rho, m)=(0,0), \\
+\infty & \text { otherwise. }
\end{array}\right.
\end{align*}
Given $\rho \in P(\Omega)$, $m \in P(\Omega^d)$ we then define the \emph{action density}
\begin{align*}
    \mathcal{A}(\rho,m) = \int_{ \Omega} \Phi(\frac{\dd{\rho}}{\dd{\sigma}}, \frac{\dd{m}}{\dd{\sigma}})\;\dd{\sigma}(x),
\end{align*}
where $\sigma$ is a reference measure such that $\rho,\,m$ are absolutely continuous with respect to $\sigma$. Due to the  one-homogeneity of $\Phi$, the value of $\mathcal{A}$ is independent of the choice of $\sigma$.
We have 
\begin{align}\label{eq:Benamou_Brenier}
    W_2^2(u_0, u_1) = \inf_{ (\rho, m) \in \mathrm{CE}(u_0,u_1)} \int_0^1  \mathcal{A}(\rho, m)\;\dt.
\end{align}
%
Here, $(\rho,m)\in P(\Omega)\times P(\Omega)^d$ belong to the constrained set $\mathrm{CE}(u_0,u_1)$ if
\begin{align}
    \partial_t \rho + {\rm div}(m) = 0 &\quad\text{on } \Omega\times (0,1), \label{eq:continuity}\\
    m\cdot n=0 &\quad\text{on } \partial\Omega\times (0,1), \label{eq:continuity_bc}\\
    \rho(0)=u_0 &\quad\text{on }\Omega,\label{eq:continuity_initial}
    \\
    \rho(1)=u_1 &\quad\text{on }\Omega,\label{eq:continuity_end}
\end{align}
where $\nu$ denotes the unit normal vector field pointing outwards of $\Omega$. This formulation constitutes a convex optimisation problem with linear constraints \cite{BenamouBrenier2000}.

\subsection{Minimizing movement / JKO scheme}
Let us consider the minimzing movement scheme \eqref{eq:JKO} in a general metric space $(X,d)$ where it reads as  
\begin{align}\label{eq:JKO_general}
u_{k+1}^\tau \in \argmin_{u \in X} ( F(u) + \frac{1}{2\tau}{d^2(u^\tau_k,u)} ).
\end{align}

Existence of iterates in this setting is usually shown via the direct method of calculus of variations, i.e., using the fact that sublevel sets are compact, to show that every minimizing sequence admits a converging subsequence and showing that, as a consequence of lower semicontinuity, its limit is a minimizer. Here, one can either assume compactness already in the topology induced by $d$ or in a weaker topology in which $d$ is still lower semicontinuous (e.g. weak convergence in a Sobolev space).
Before we specialize to the Wasserstein case, we briefly discus the stability properties of iterates. This is relevant for us since we are interested in a setting where both the initial datum as well as the functional $F$ may not be known exactly. The following lemma shows that iterates are stable under perturbations of the functional $F$ or the previous time step.
\begin{lemma}\label{lem:JKO_stability} Let $\delta_F>0$, and let $F_1,\, F_2:  X \to \R_+ \cup \{+\infty\}$ be two  functionals satisfying 
$$
\sup_{u} |F_1(u) - F_2(u)| \le \delta_F.
$$
Furthermore, fix two elements $u_k^{\tau,1},\, u_k^{\tau,2} \in X$.
Denoting by $u_{k+1}^{\tau,1}$ and $u_{k+1}^{\tau,2}$ the solutions of \eqref{eq:JKO_general} with $F=F_1$, $F=F_2$ and $u_k^\tau=u_{k+1}^{\tau,1}$ and $u_k^\tau=u_{k+1}^{\tau,2}$, respectively. Then, we have the estimate
\begin{align}\label{eq:JKO_stability}
d^2(u_{k+1}^{\tau,1},u_{k+1}^{\tau,2}) \le 9 d^2(u_k^{\tau,1},\, u_k^{\tau,2}) + 8\tau\delta_F + 4\tau(\|F_1\|_{C^0} + \|F_2\|_{C^0}).
\end{align}
\end{lemma}
\begin{proof}
Using the optimality of $u^{\tau,i}_{k+1}$, $i=1,2$ in \eqref{eq:JKO_general} yields
\begin{align*}
F_1(u_{k+1}^{\tau,1}) + \frac{1}{2\tau} d^2(u_{k+1}^{\tau,1},u^{\tau,1}_k) &\le F_1(u_{k}^{\tau,2}) + \frac{1}{2\tau}   d^2(u_{k}^{\tau,2},u^{\tau,1}_k),\\
F_2(u_{k+1}^{\tau,2}) + \frac{1}{2\tau} d^2(u_{k+1}^{\tau,2},u^{\tau,2}_k) &\le F_2(u_{k}^{\tau,1}) + \frac{1}{2\tau}   d^2(u_{k}^{\tau,1},u^{\tau,2}_k),    
\end{align*}
which implies
\begin{align*}
    d^2(u_{k+1}^{\tau,j}, u_{k}^{\tau,j}) \le 2\tau(F_j(u_k^{\tau,i}) - F_j(u_{k+1}^{\tau,j})) + d^2(u_{k}^{\tau,i}, u_{k}^{\tau,j}),\quad i,\,j=1,2,\; i\neq j.
\end{align*}
Using the triangle inequality, this allows to estimate
\begin{align*}
    d^2(u_{k+1}^{\tau,1},u_{k+1}^{\tau,2}) & \le 3(d^2(u_{k+1}^{\tau,1},u_{k}^{\tau,1}) + d^2(u_{k+1}^{\tau,2},u_{k}^{\tau,2}) + d^2(u_{k}^{\tau,1},u_{k}^{\tau,2}))\\ 
    &\le 9 d^2(u_{k}^{\tau,1},u_{k}^{\tau,2}) + 2\tau(F_1(u_k^{\tau,2}) - F_1(u_{k+1}^{\tau,1}) + F_2(u_k^{\tau,1}) - F_2(u_{k+1}^{\tau,2}))
\end{align*}
We further estimate the second term on the right hand side as 
\begin{align*}
    &F_1(u_k^{\tau,2}) - F_1(u_{k+1}^{\tau,1}) + F_2(u_k^{\tau,1}) - F_2(u_{k+1}^{\tau,2})\\
    & = F_1(u_k^{\tau,2})      - F_2(u_k^{\tau,2}) 
      + F_2(u_{k+1}^{\tau,1})  - F_1(u_{k+1}^{\tau,1})
      + F_2(u_k^{\tau,1})      - F_1(u_k^{\tau,1}) \\
    &+ F_1(u_{k+1}^{\tau,2}) - F_2(u_{k+1}^{\tau,2})
     + F_2(u_k^{\tau,2})        - F_2(u_{k+1}^{\tau,1}) 
     + F_1(u_k^{\tau,1})        - F_1(u_{k+1}^{\tau,2})\\
&\le 4\delta_F+ 2(\|F_1\|_{C^0} + \|F_2\|_{C^0}),
\end{align*}
which completes the proof.
\end{proof}
\begin{remark} Note that in general, assuming $F$ to be convex but not strictly convex, minimizers of \eqref{eq:JKO_general} are not unique. Thus additional term $4\tau(\|F_1\|_{C^0} + \|F_2\|_{C^0})$ on the right hand side is natural in this setting. However, because this term remains even if we assume $F$ to be strictly convex (implying uniqueness of minimizers) shows that our estimate is not optimal. A way to improve this estimate would be to work with geodesic interpolations, see e.g. \cite{Santambrogio2017_overiewGF}. We leave this question for future work.
\end{remark}


A natural questions is the behaviour of iterates as $\tau$ tends to zero. The starting point for the convergence analysis is to derive a priori estimates that, using compactness, guarantee the existence of a limiting curve. The key idea is to exploit the minimization property of the iterates, i.e.,
\begin{align}\label{eq:JKO_minimality_estimate}
    F(u_{k+1}^\tau) + \frac{1}{2\tau}d^2(u_k^\tau, u_{k+1}^\tau) \le F(u_k^\tau).
\end{align}

This implies the \emph{energy estimate} $\sup_k F(u_k^\tau) \le F(u_0)$.
Summing over $k$ further yields the \emph{total square estimate}
\begin{align}\label{eq:total_square}
    \sum_{k\in\N_0} d^2(u_k^\tau, u_{k+1}^\tau) \le 2\tau(F(u_0) - \inf_k F(u_k^\tau)).
\end{align}
Assuming $F$ to be bounded from below and $u_0$ such that $F(u_0) < +\infty$, this provides a uniform bound.
Applying the triangle inequality
for any $0 \leq l \leq k$, and using the Cauchy-Schwarz inequality we obtain
$$
d(u_l^\tau, u_k^\tau) \leq \sqrt{2 \tau} \sum_{j=l}^{k-1} \frac{1}{\sqrt{2 \tau}} d(u^\tau_{j}, u^\tau_{j+1}) \leq \sqrt{2 \tau} \sqrt{(k-l)} \sqrt{F(u^\tau_{0})-\inf _{n \in \mathbb{N}} F(u_k^\tau)}.
$$
As a consequence for any $0 \leq s \leq t$, denoting by $u^\tau(t)$ the constant speed geodesic interpolation introduced in \eqref{eq:geodesic_interpolation}, we have
\begin{equation}\label{eq:Hoelder_est}
\begin{aligned}
d(u^\tau(s), u^\tau(t)) \leq & d(u^\tau(s), u^\tau_{[\frac{s}{\tau}+1] \tau})+d(u^\tau_{[\frac{s}{\tau}+1] \tau}, u^\tau_{[\frac{t}{\tau}] \tau})+d(u^\tau_{[\frac{t}{\tau}] \tau}, u^\tau(t)) \\
\leq &([\frac{s}{\tau}+1]-\frac{s}{\tau}) d(u^\tau_{[\frac{s}{\tau}] \tau}, u^\tau_{[\frac{s}{\tau}+1] \tau}) \\
&+\sqrt{2 \tau([\frac{t}{\tau}]-[\frac{s}{\tau}+1])(F(u^\tau_{0})-\inf _{n \in \mathbb{N}} F(u_k^\tau))} \\
&+(\frac{t}{\tau}-[\frac{t}{\tau}]) d(u^\tau_{[\frac{t}{\tau}] \tau}, u^\tau_{[\frac{t}{\tau}+1] \tau}) \\
\leq & \sqrt{6(F(u_{0})-\inf _{n \in \mathbb{N}} F(u_{n}))}(t-s)^{\frac{1}{2}}
\end{aligned}
\end{equation}
Ultimately, this $\frac{1}{2}$-H\"older estimate allows the application of the Arzel\'a-Ascoli theorem to conclude relative compactness of the family $(u^\tau)_{\tau > 0}$, i.e. the existence of a limit curve. It then remains to show that this limit indeed satisfies the chosen gradient flow formulation, which depends on the problem at hand.
\begin{remark}
While the above estimate is only of order $\sqrt{\tau}$, it can be shown, \cite[Ch. 4]{AmbrosioGigliSavare2008}, that under mild regularity assumptions on the initial point, convergence with order $\tau$ holds. In order to obtain second order convergence, one has to modify the variational scheme. In the Euclidean case, replacing \eqref{eq:JKO} by 
$$
u_{k+1}^\tau = \argmin_{u} ( 2F\left(\frac{u+u_k^\tau}{2}\right) + \frac{1}{2\tau}\|u-u^\tau_k\|_2^2 ),
$$
is sufficient. Using the notion of geodesics, this idea carries over to the Wasserstein setting, \cite{Legende2017_2ndOrder}; see also \cite{Matthes_2019,Carrillo2022_PrimalDual} for BDF-type and Crank-Nicholson-type higher-order schemes.
\end{remark}
\section{Data driven minimizing movement scheme}\label{sec:daJKO}
We first introduce our approach in the most general setting in metric spaces, generalizing the Euclidean setting described in the introduction. To this end we denote by $(X,d_S)$ the metric space containing the state while the data lives in the space $(Y,d_D)$. We first introduce a measurement operator $\B: X \to Y$. In this work, we will always assume $\B$ to be linear and $d_D$-$d_S$-Lipschitz continuous with constant $L_\B$, i.e. for all $x_1,\, x_2 \in X$ there holds
\begin{align}\label{eq:B_Lipschitz}
d_D(\B(x_1),\B(x_2)) \le L_{\B} d_S(x_1, x_2).
\end{align}
In this set-up, we extend \eqref{eq:JKO_general} by introducing the data driven minimizing movement scheme as follows
\begin{align}\label{eq:JKO_metric_data}
u_{k+1}^{\tau,\delta} \in \argmin_{u} ( F(u) + \frac{1}{2\tau}\left(d_S^2(u,u^{\tau,\delta}_k)+ \frac{d_D^2(\mathcal{B}(u),v_{k+1}^\delta)}{\theta} \right),
\end{align}
where the $v_k \in Y$ are given measurements and $\theta \in (0,\infty)$ is a weighting parameter. 
For the particular choice $d_S = W_2$ we call \eqref{eq:JKO_metric_data} the data driven JKO (\daJKO) scheme.
In the limit $\theta \to 0^+$, we obtain $u_{k+1}$ s.t. $\mathcal{B}(u_{k+1}) = v_{k+1}$, i.e. the flow is completely data driven while for $\theta \to + \infty$, the data has no more influence on $u_{k+1}$. 
\begin{remark}[Minimizers without noise]
Let us briefly discuss the case of exact measurements. If the data is generated by solutions $u_k^\tau$ to \eqref{eq:JKO_general}, i.e. $v_k = \B(u_k^\tau)$, then solving \eqref{eq:JKO_metric_data} yield the same minimizers as \eqref{eq:JKO_general} (up to the possible non-uniqueness of minimizers). If, on the other hand, the data is generated from the time-continuous limit curve $u(t)$, evaluated at discrete times $t_k = t\tau$, we will always have an approximation error of order $\tau$.
\end{remark}
\begin{remark}[Connection to nudging]\label{rem:control_interpretation}
Returning to the Euclidean setting for a moment, we see that if we choose $\theta$ as a function of $\tau$, we obtain, formally in the limit $\tau \to 0$, an interpretation as a closed loop control problem / data assimilation problem called nudging introduced in \cite{Anthes1974_NudgingFirst}. We start with 
\begin{align}\label{eq:JKO_data_Rd}
u_{k+1}^\tau = \argmin_{u\in\R^d} ( F(u) + \frac{1}{2\tau}\|u-u^\tau_k\|_2^2  + \frac{1}{2\tau\theta}\|\B u-v_{k+1}\|_2^2) 
\end{align}
Calculating the Euler-Lagrange equation results in
\begin{align}\label{eq:EulerLagrange_Data}
u_{k+1}^\tau = u_k^\tau - \tau \nabla F(u^\tau_{k+1}) + \frac{1}{\theta}\B^*(\B u_{k+1}^\tau - v_{k+1}).
\end{align}
If we choose, for a  given constant $c>0$, $\theta = c / \tau$ and (formally) pass to the limit $\tau \to 0$, we recover the equation
\begin{equation*}
\begin{aligned}
    u'(t) &= -\nabla F(u(t)) + \frac{1}{c}\B^*(\B u(t) - v(t)), \quad t> 0,\\
    u(0) &= u_0,
    \end{aligned}
\end{equation*}
where $v(t)$ is a suitable interpolation of the data points in time (e.g. piecewise linear).
\end{remark}

We do not argue on the question of existence of minimizers (usually proven by means of the direct method of calculus of variations) which depends on specific choices of $(X,d_S)$ and $(Y,d_D)$.
Instead, we comment on convergence as $\tau \to 0$.
\begin{remark}[Convergence as $\tau \to 0$]\label{eq:JKO_data_convergence}
For the JKO scheme without data, the estimates in \eqref{eq:total_square} and \eqref{eq:Hoelder_est} provide the necessary compactness to obtain a limiting curve as $\tau \to 0$. They are shown by comparing the functional values of $u_{k+1}^\tau$ and $u_k^\tau$, using optimally of the former, see \eqref{eq:JKO_minimality_estimate}. In the data driven case, this inequality becomes 
\begin{equation}
\begin{aligned}
    F(u_{k+1}^{\tau,\delta}) + \frac{1}{2\tau}d^2(u_k^{\tau,\delta}, u_{k+1}^{\tau,\delta}) &\le F(u_k^{\tau,\delta}) + \frac{1}{2\theta\tau}\left(d^2(\B(u^{\tau,\delta}_k),v_{k+1}) - d^2(\B(u^{\tau,\delta}_{k+1}),v_{k+1})\right)\\
    \end{aligned}
\end{equation}
In order to proceed, we assume the error bound
\begin{align}\label{eq:bnd_data_error}
d_D(\B(u^\tau_{k+1}),v_{k+1}) \le \delta_k,
\end{align}
as well as $d_S$-$d_D$-Lipschitz-continuity of $\B$ with constant $L_\B$ satisfying $L_\B^2 < \theta / 2$. This allows us to further estimate, using the reverse triangle inequality $|d(x, z)-d(y, z)| \leq d(x, y)$, that
\begin{equation}
    \begin{aligned}
    &\quad d^2(\B(u^\tau_k),v_{k+1}) - d^2(\B(u^\tau_{k+1}),v_{k+1})\\
    &=\big(d(\B(u^\tau_k),v_{k+1}) - d(\B(u^\tau_{k+1}),v_{k+1})\big) \big(d(\B(u^\tau_k),v_{k+1}) + d(\B(u^\tau_{k+1}),v_{k+1}) \big)\\
    &\le d(\B(u^\tau_k),\B(u^\tau_{k+1}))\, (d(\B(u^\tau_k),v_{k+1}) + d(\B(u^\tau_{k+1}),v_{k+1})).
\end{aligned}
\end{equation}
Using the triangle inequality once more,
$$
d(\B(u^\tau_k),v_{k+1}) \le d(\B(u^\tau_{k+1}),v_{k+1}) + d(\B(u^\tau_k),\B(u^\tau_{k+1})),
$$
we obtains
\begin{align*}
     d^2(\B(u^\tau_k),v_{k+1}) - d^2(\B(u^\tau_{k+1}),v_{k+1})
    \le  2 d^2(\B(u^\tau_k),\B(u^\tau_{k+1}))  + d^2(\B(u^\tau_{k+1}),v_{k+1}).
\end{align*}
Employing the Lipschitz continuity of $\B$ and $L_\B^2 < \theta / 2$ as well as the bound on the measurement error \eqref{eq:bnd_data_error}, results in
\begin{align*}
    F(u_{k+1}^\tau) + \frac{1}{2\tau}\left(1- \frac{2L_\B^2}{\theta}\right)d^2(u_k^\tau, u_{k+1}^\tau) &\le F(u_k^\tau)  + \frac{1}{2\theta\tau}\delta_k^2.
\end{align*}
 Thus in order to have convergence, both $\delta_k$ as well as $\theta$ must depend on $\tau$ in a suitable way to ensure that the last term converges to zero in the final estimate. A similar situation occurs when studying \eqref{eq:JKO_general} with a perturbed functional, see \cite{braides2014local,Braides2016_MMSequenceFunctionals}. 
\end{remark}
\begin{remark}[Relation to Tikhonov regularization]
As mentioned in the introduction, \eqref{eq:JKO_data} can also be seen as a Tikhonov regularization with a non-linear forward operator. Thus, one might expect  stability and convergence results as in \cite[Chapter 10]{engl2000regularization}. However, in the most general metric setting, neither minimizers nor limiting curves are unique (not even locally). In the $2$-Wasserstein case, combined with additional convexity assumptions on $F$, similar results might be possible which is, however, beyond the scope of this work.
\end{remark}

\section{Numerical realization}\label{sec:numerics}
We now introduce our numerical scheme which is based on the recent work by \cite{Carrillo2022_PrimalDual}, yet using a finite difference approximation to discretize the divergence constraints instead of a piecewise constant approximation. This approach yields a clear interpretation of the initial and boundary conditions appearing in \eqref{eq:Benamou_Brenier}. We provide a detailed description of the numerical scheme to highlight the influence of weighted norms and inner products that arise from the finite difference discretization. Let us remark that higher order schemes have been developed in \cite{Carrillo2022_PrimalDual,Matthes_2019}, and that our data assimilation approach can be combined with these schemes in a similar fashion.

\subsection{Finite difference approximation}
To convey the main ideas, we consider the one-dimensional case, i.e., $\Omega=(L,R)$ with $L<R$ being real numbers. The case of multiple dimensions can be handled similarly if the domain is a Cartesian product of intervals by employing a usual tensor product construction for the discretization.
We partition $\Omega$ into $N_x$ intervals with mesh size $\delta_x=(R-L)/N_x$. Similarly, we partition the time interval $(0,1)$ into $N_t$ intervals with mesh size $\delta_t$, 
and introduce the grid points
\begin{align*}
    x_j=(j-1)\delta_x + L,\quad t_k = (k-1)\delta_k, \quad j\in\{1,\ldots, N_x+1\},\, k\in\{1,\ldots,N_t+1\}.
\end{align*}
To any continuous function $v:[L,R]\times[0,1] \to \mathbb{R}$, we associate a grid function $v_h=(v_{j,k})_{j,k}$ via
\begin{align*}
    v_{j,k}=v(x_j,t_k) \quad j\in\{1,\ldots, N_x+1\},\, k\in\{1,\ldots,N_t+1\}.
\end{align*}
Similar notation is used for grid functions associated with the spatial and temporal partitions, respectively.

The discrete energy approximating \eqref{eq:functional_F} is obtained by using a composite trapezoidal rule
\begin{align}\label{eq:functional_Fh}
    F_h(u_h) 
    = \sum_{j=1}^{N_x+1} w^x_j \Big( U(u_j) + V_j u_j + \frac{1}{2} \sum_{i=1}^{N_x+1} w^x_i W_{i,j} u_{i}u_{j}\Big),
\end{align}
where $V_j=V(x_j)$, $W_{i,j}=W(x_i,x_j)$, and $u_h=(u_j)$ is a grid function associated with the spatial partition, and the weight functions are given by
\begin{align*}
    w_j^x =\begin{cases} \frac{\delta_x}{2} &j\in\{1,N_x+1\},\\
                        \delta_x   &\text{otherwise}.
                        \end{cases}
\end{align*}
We denote by $w_k^t$ a similar weight function for the composite trapezoidal rule associated with the temporal partition, and note that $\sum_{k=1}^{N_t+1}w_k^t=1$.

Using these rules, we define the following inner product and norm for grid functions $v_h=(\rho_{j,k}^v,m_{j,k}^v)$, $w_h=(\rho_{j,k}^w,m_{j,k}^w)$
\begin{align}\label{eq:def_norm}
    \langle v_h,w_h\rangle = \sum_{k=1}^{N_k+1}\sum_{j=1}^{N_x+1} w_k^tw_j^x\left( \rho_{j,k}^v\rho_{j,k}^w + m_{j,k}^vm_{j,k}^w\right),\quad \|v_h\|=\sqrt{\langle v_h,v_h\rangle},
\end{align}
which makes the linear space of grid functions a Hilbert space.

Before we state the discrete optimization problem corresponding to \eqref{eq:daJKO_intro}, we discretize the constraints \eqref{eq:continuity}--\eqref{eq:continuity_initial}.
For the divergence constraint \eqref{eq:continuity} we use centred differencing in space and a backward differencing in time for the interior grid points $j\in\{2,\ldots N_x\},\ k\in\{2,\ldots,N_t+1\}$, i.e.,
\begin{align}\label{eq:FD_int}
    \frac{\rho_{j,k}-\rho_{j,k-1}}{\delta_t}+\frac{ m_{j+1,k}-m_{j-1,k}}{2\delta_x}=0.
\end{align}
For the boundary we use a one-sided finite difference approximation to approximate $m_x$, i.e., for $j\in\{1,N_x+1\}$ we use
\begin{align}
    \frac{\rho_{1,k}-\rho_{1,k-1}}{\delta_t}+\frac{ m_{2,k}-m_{1,k}}{\delta_x}&=0,\quad\text{and}\label{eq:FD_bdry_left}\\
    \frac{\rho_{N_x+1,k}-\rho_{N_x+1,k-1}}{\delta_t}+\frac{ m_{N_x+1,k}-m_{N_x,k}}{\delta_x}&=0.\label{eq:FD_bdry_right}
\end{align}
The boundary condition \eqref{eq:continuity_bc} becomes
\begin{align}\label{eq:FD_bdry}
    m_{1,k}=0=m_{N_x+1,k}\quad\text{for all } k\in\{1,\ldots N_t+1\}.
\end{align}
Since the discretization does not depend on $m_{j,1}$, we set $m_{j,1}=1$ for $j=1,\ldots N_x+1$.
Furthermore, the initial condition \eqref{eq:continuity_initial} and the mass constraint are discretized as follows
\begin{align}
    \rho_{j,1}&=\rho_0(x_j), \quad j\in\{1,\ldots N_x+1\},\label{eq:FD_initial}\\
    \sum_{j=1}^{N_x+1} w_j^x (\rho_{j,k}-\rho_0(x_j))&=0,\quad k\in\{2,\ldots N_t+1\}. \label{eq:FD_mass}
\end{align}

\subsection{Discrete daJKO scheme}
Using the dynamic formulation of the Wasserstein distance \eqref{eq:Benamou_Brenier} and the notation $\rho_h^1=\rho_h(\cdot,1)$ for a grid function associated to the spatial partition, the discrete daJKO scheme, i.e., the minimization problem \eqref{eq:daJKO_intro} - but without data term - is then
\begin{align*}
    \inf_{(\rho_h,m_h)} \sum_{k=1}^{N_t+1}\sum_{j=1}^{N_x+1} w^x_j w_k^t \left(\Phi(\rho_{j,k},m_{j,k})+ {\tau}\big( U(\rho^1_j) + V_j \rho^1_j + \frac{1}{2} \sum_{i=1}^{N_x+1} w_i^x W_{i,j} \rho^1_{i}\rho^1_{j}\big)\right).
\end{align*}
Here, $(\rho_h,m_h)$ are grid function associated to the space time grid such that they satisfy the constraints \eqref{eq:FD_int}--\eqref{eq:FD_mass} in the following relaxed form, that is
\begin{align}
     \sum_{k=2}^{N_t+1}w_k^t \Bigg( w_1^x \left(\frac{\rho_{1,k}-\rho_{1,k-1}}{\delta_t}+\frac{ m_{2,k}-m_{1,k}}{\delta_x}\right)^2 \notag\\
     + w_{N_x+1}^x \left(\frac{\rho_{N_x+1,k}-\rho_{N_x+1,k-1}}{\delta_t}+\frac{ m_{N_x+1,k}-m_{N_x,k}}{\delta_x}\right)^2 \notag \\
     + \sum_{j=2}^{N_x}w_j^x \left(\frac{\rho_{j,k}-\rho_{j,k-1}}{\delta_t}+\frac{ m_{j+1,k}-m_{j-1,k}}{2\delta_x}\right)^2\Bigg)&\leq \delta_1^2, \label{eq:div_const_LS}\\
     \sum_{k=1}^{N_t+1}w_k^t \big(m_{0,k}^2+m_{N_x+1,k}^2\big) &\leq \delta_2^2 \label{eq:bdry_const_LS}\\
     \sum_{k=1}^{N_t+1} w_k^t \left(\sum_{j=1}^{N_x+1} w_j^x (\rho_{j,k}-\rho_0(x_j))\right)^2&\leq \delta_3^2 \label{eq:mass_const_LS}\\
     \sum_{j=1}^{N_x+1}w_j^x (\rho_{j,1}-\rho_0(x_j))^2 &\leq \delta_4^2\label{eq:initial_const_LS}
\end{align}
for some tolerances $\delta_i$, $i=1,2,3,4$.
Similar to \cite{Carrillo2022_PrimalDual}, we observe that the weakened constraints \eqref{eq:div_const_LS}--\eqref{eq:initial_const_LS} are quadratic and can be written in the form
\begin{align}\label{eq:def_C}
   Au \in \mathcal{C}_\delta=\{x \colon \| x_i -b_i\|_2\leq \delta_i, i=1,2,3,4\}, 
\end{align}
where the vector $u$ contains the coefficient of the grid functions $(\rho_h,m_h)$. Note that the weights $w_j^x$ and $w_k^t$ are included in the definition of $A_i$ and $b_i$, respectively, and the vectors $x_i$ are slices of the vector $x$ corresponding to the number of rows in $A_i$.
We define the matrix $A$ by vertically concatenating the matrices $A_i$, $i=1,\ldots, 4$. We note that $A$ is the matrix of a linear map from the Hilbert space of grid functions with inner product defined in \eqref{eq:def_norm} to a Euclidean space.
In order to enforce the constraints, we introduce the indicator function of the set $\mathcal{C}_\delta$ as
\begin{align*}
    \mathfrak{i}_{\delta}(\phi) =\begin{cases} 0 & \text{if } \phi \in \mathcal{C}_\delta,\\
                                        \infty & \text{otherwise}.
                                        \end{cases}
\end{align*}


Summarizing, given $u_h^{(n)}$ as an approximation of $u_n^\tau$ and suitable data $v^{(n+1)}$, one step of the discrete daJKO scheme is to compute $u_h^{(n+1)}=\rho_h^*(\cdot,1)$ where, for $\rho_h^0=u_h^{(n)}$, the grid function
$(\rho_h^*,m_h^*)$ is the minimizer of 
\begin{align*}
    \inf_{(\rho_h,m_h)} \left(\sum_{k=1}^{N_t+1}\sum_{j=1}^{N_x+1} w^x_j w_k^t \Phi(\rho_{j,k},m_{j,k})\right)+\tau E^{(n+1)}_h(\rho_h^1) + \mathfrak{i}_\delta(A u),
\end{align*}
where the combined energy is defined by
\begin{align*}
E_h^{(n+1)}(\rho_h^1)=F_h(\rho_h^1) + \frac{1}{2\theta} \|B_h(\rho_h^1)-v^{(n+1)}\|_d^2. 
\end{align*}
Here, $B_h$ denotes a suitable discretization of the measurement operator $\mathcal{B}$ and $\|\cdot\|_d$ a suitable $\ell_2$-like norm. One may choose different metrics for the data term, but we do not go in this direction here. Similarly, the initialization $\rho^0_h(x_j)=u^{(n)}_h(x_j)$ can be modified, e.g., to incorporate different data, and we employ such a modification in one of the examples in Section~\ref{sec:perturbed_m} below.

\subsection{Implementation of the daJKO scheme}
In order to perform one step of the daJKO algorithm computationally, we employ the algorithm developed in \cite{Yan2018} for the minimization of a sum of three convex functionals, see Algorithm~\ref{alg:JKOStep}. Iterating the daJKO steps yields our overall data assimilation scheme, as described in Algorithm~\ref{alg:JKO_scheme}. We note that, if no data terms are present, Algorithm~\ref{alg:JKO_scheme} reduces to the algorithm used in \cite{Carrillo2022_PrimalDual}. Here, we use, however, a different differencing scheme to discretize the constraints.
Next, let us discuss the building blocks used in Algorithm~\ref{alg:JKOStep}.
\IncMargin{1em}
\begin{algorithm}
\SetKwInOut{Input}{Input}\SetKwInOut{Output}{Output}
\Input{$u^{(0)},\, \phi^{(0)},\, it_{\max},\, E,\, \nabla E,\, \lambda,\, \sigma,\, A,\, b,\, \delta$}
\Output{$u^*$, $\phi^*$}
\BlankLine
\emph{Initialize $\bar u^{(0)}=u^{(0)}$ and $l=0$}\;
\For{$i = 0$ \KwTo $it_{\max}$}{
   \emph{$\phi^{(i+1)}={\rm prox}_{\sigma \mathfrak{i}^*_\delta}(\phi^{(i)} +\sigma A \bar u^{(i)})$}\;
  \emph{$u^{(i+1)} = {\rm prox}_{\lambda\Phi}( u^{(i)}-\lambda\nabla E(u^{(i)})-\lambda A^* \phi^{(i+1)})$}\;
  \emph{$\bar u^{(i+1)}= 2 u^{(i+1)}-u^{(i)}+\lambda \nabla E(u^{(i)}) -\lambda\nabla E(u^{(i+1)})$}\;
  \If{convergence}{
  \emph{$u^*=u^{(i+1)}$}\;
  \emph{$\phi^*=\phi^{(i+1)}$}\;
  break}{}
}
\caption{Primal dual algorithm for one daJKO step, cf. \cite{Yan2018}.\label{alg:JKOStep}}
\end{algorithm}\DecMargin{1em}

The adjoint $A^*$ of $A$ is defined via the relation
\begin{align*}
    \langle \psi, A u\rangle_2 = \langle A^*\psi, u\rangle,
\end{align*}
where the weighted inner product \eqref{eq:def_norm} is used in the right-hand side.
Denoting $W$ the Gramian of that weighted inner product, we obtain for the matrix representation of the adjoint the identity $A^*\psi =W^{-1}A^T\psi$, where $A^T$ is the transpose matrix of the matrix $A$.
\IncMargin{1em}
\begin{algorithm}
\SetKwFunction{DAJKOStep}{DAJKOStep}
\SetKwInOut{Input}{Input}\SetKwInOut{Output}{Output}
\Input{$\rho_0$}
\Output{$(\rho^{(k)})_{k=1}^{N_{JKO}}$}
\BlankLine
\emph{Initialize $\rho^{(1)}=\rho_0, \, \phi^{(1)}=0,\, u^{(1)}=0,\,  u^{(1)}_{1:N_x+1,1}=\rho_0$}\;
\For{$n = 1$ \KwTo $N_{JKO}-1$}{
   \emph{$u^{(n+1)},\phi^{(n+1)}=\DAJKOStep(u^{(n)},\phi^{(n)},it_{\max},\lambda,\sigma,A,b,\delta)$}\;
   \emph{$\rho^{(n+1)}= u^{(n+1)}_{1:N_x+1,N_t+1}$}\;
   \emph{Update $b$}\;
}
\caption{Primal dual algorithm for daJKO scheme, where DAJKOStep refers to Algorithm~\ref{alg:JKOStep}. Line 5 ensures that the initial condition for the next iteration is set correctly.\label{alg:JKO_scheme}}
\end{algorithm}\DecMargin{1em}

For a convex, lower semi-continuous and proper functional $\Psi$ defined on a Hilbert space, the proximal operator is defined by
\begin{align*}
    {\rm prox}_{\Psi}(u) = \argmin_{v} \Psi(v) + \frac{1}{2}\|v-u\|^2.
\end{align*}
Note that the norm is the one from the corresponding Hilbert space.
Next, we show that the proximal operators used in Algorithm~\ref{alg:JKOStep} can be evaluated efficiently. By definition of the proximity operator and the weighted norm in \eqref{eq:def_norm}, we need to consider the following minimization problem for a given $u_h=(\rho_{j,k},m_{j,})$:
\begin{align*}
    &\inf_{v_h=(\rho^v_{j,k},m^v_{j,k})} \left(\sum_{k=1}^{N_t+1}\sum_{j=1}^{N_x+1} w^t_k w_j^x \Phi(\rho_{j,k}^v,m^v_{j,k})\right) + \frac{1}{2\lambda} \|v_h-u_h\|^2 \\
    &= \inf_{(\rho^v_{j,k},m^v_{j,k})}  \sum_{k=1}^{N_t+1}\sum_{j=1}^{N_x+1}w_k^t w_j^x \left( \Phi(\rho_{j,k}^v,m^v_{j,k}) + \frac{1}{2\lambda}\big( |m^v_{j,k}-m_{j,k}|^2+(\rho^v_{j,k}-\rho_{j,k})^2\big)\right).
\end{align*}
Therefore, the minimizer is given per grid point by, cf. \cite[Proposition~1]{Papadakis2014},
\begin{align}\label{eq:prox_action}
    {\rm prox}_{\lambda\Phi}(\rho_{j,k},m_{j,k})=\begin{cases} (\rho^*_{j,k},m^*_{j,k}) &\text{if }\rho_{j,k}^*>0,\\ (0,0) &\text{otherwise,}\end{cases}
\end{align}
with $m^*_{j,k}=\rho_{j,k}^* m_{j,k}/(\rho_{j,k}^*+\lambda)$, and $\rho_{j,k}^*$ the largest positive real root of the cubic polynomial
\begin{align}\label{eq:polynom}
    P_{j,k}(x)=(x-\rho_{j,k})(x+\lambda)^2-\frac{\lambda}{2}|m_{j,k}|^2,
\end{align}
which can be computed using Cardano's formula; we refer to the appendix for details on the calculation of the largest real root. Therefore, every iterate $u^{(i)}=(\rho^{(i)}_{j,k},m^{(i)}_{j,k})$ in Algorithm~\ref{alg:JKOStep} satisfies $\rho^{(i)}_{j,k}\geq 0$.
The computation of ${\rm prox}_{\lambda\Phi}(\cdot)$ per grid point allows for an efficient implementation and a straight-forward parallelization.

The functional $\mathfrak{i}_\delta^*$ showing up in Algorithm~\ref{alg:JKOStep} is the Legendre-Fenchel transform of $\mathfrak{i}_\delta$, which is defined as
\begin{align*}
    \mathfrak{i}^*_\delta (\phi) = \max_{\psi} \langle \psi,\phi\rangle_2 - \mathfrak{i}_\delta(\psi),
\end{align*}
where the maximum is taken over all vectors $\psi$ with dimension corresponding to the number of rows of $A$, and the inner product $\langle\cdot,\cdot\rangle_2$ is the standard Euclidean inner product on that space.
Moreau's identity \cite{Moreau65} implies that 
\begin{align*}
    {\rm prox}_{\sigma \mathfrak{i}^*_\delta}(\phi)=\phi-\sigma\, {\rm prox}_{\mathfrak{i}_\delta}(\phi/\sigma).
\end{align*}
We note that
\begin{align*}
    {\rm prox}_{\mathfrak{i}_\delta}(\phi) = \argmin_\psi \mathfrak{i}_\delta(\psi) + \frac{1}{2}\|\psi-\phi\|_2^2 = {\rm proj}_{B_\delta}(\phi),
\end{align*}
where the projection onto $B_\delta$ is given by
\begin{align*}
    {\rm proj}_{B_\delta}(x)&=\begin{cases} x_i &\|x_i-b_i\|_2\leq \delta_i,\\
    \delta_i \frac{x_i-b_i}{\|x_i-b_i\|_2}+b_i &\text{otherwise},
    \end{cases}\,\, i\in\{1,2,3,4\}.
\end{align*}
 Summarizing, both proximal operators used in Algorithm~\ref{alg:JKOStep} can be  applied efficiently.

Next, let us compute the gradient of the discrete energy functional $F_h$, defined in \eqref{eq:functional_Fh}, which is defined via
\begin{align*}
    dF_h(u_h)[v_h] = \langle \nabla_uF_h(u_h),v_h\rangle,
\end{align*}
where $\langle\cdot,\cdot\rangle$ denotes the inner product of grid function induced by the composite trapezoidal rule defined in \eqref{eq:def_norm}.
For $v_h=(\rho^v_{j,k},m^v_{j,k})_{j,k}$, we have that
\begin{align*}
    dF_h(u_h)[v_h] &= \sum_{j=1}^{N_x+1} w_j^x \left(U'(\rho^1_j) \rho^v_{j,N_t+1}+V_j \rho^v_{j,N_t+1}+ \sum_{i=1}^{N_x+1} w_i^x W_{i,j}\rho^1_i \rho^v_{j,N_t+1}\right)\\
    &= \sum_{k=1}^{N_t+1}\sum_{j=1}^{N_x+1} w_j^x w_k^t \left( \Big(U'(\rho^1_j) +V_j + \sum_{i=1}^{N_x+1} w_i^x W_{i,j}\rho^1_i\Big)/w_{N_t+1}^t \right) \rho^v_{j,k}\delta_{N_t+1,k}
\end{align*}
Hence, the gradient is given by
\begin{align}\label{eq:grad_Fh}
    (\nabla_u F_h(u_h))_{j,k} &= \left(\frac{\delta_{k,N_t+1}}{w_k^t}\big(U'(\rho^1_j) + V_j + \sum_{i=1}^{N_x} w_i^x W_{i,j}\rho^1_i\big),0\right),
\end{align}
where the zero entry corresponds to the variation in $m_{j,k}$.
As a stopping criterion in line 6 of Algorithm~\ref{alg:JKOStep}, we require that 
\begin{align*}
    \max\left\{\frac{\|u^{(n+1)}_h-u^{(n)}_h\|}{\|u^{(n)}_h\|},\frac{\| \phi^{(n+1)}-\phi^{(n)}\|_2}{\|\phi^{(n)}\|_2},\frac{|E_h(u^{(n+1)})-E_h(u^{(n)})|}{E_h(u^{(n)})} \right\}<tol.
\end{align*}

\section{Examples}\label{sec:examples}
We will illustrate our method on the following two non-linear PDEs: The porous medium equation \cite{Vazquez2007_PME} and a variant of the Patlak-Keller-Segel model for the motion of bacteria under the influence of a chemical signal introduced in \cite{Blanchet2008_KellerSegelJKO}. 
%

\subsection{Porous medium equation}
The porous medium equation (PME) has a number of physical applications, exampling being the description of the flow of an isentropic gas through a porous medium or the study of groundwater infiltration. Given $m \ge 1$, it reads as 
\begin{align}\label{eq:PME}
    \partial_t u = \Delta u^m, \quad \text{ in } \Omega \times (0,T).
\end{align}
As shown in the seminal work of Otto \cite{Otto2001GeometryPME}, this is indeed a gradient flow with respect to the Wasserstein distance for the energy functional 
\begin{align}
    F(u) = \int_\Omega U(u)\;dx
\end{align}
where the  internal energy is defined as 
\begin{align}\label{eq:internal}
    U(u) = \begin{cases}
     \frac{1}{m-1}u^m, &  m > 1,\\
     u\log(u), &  m = 1.
    \end{cases}
\end{align}
We supplement \eqref{eq:PME} with the following initial and boundary conditions
\begin{align*}
    u(x,0) = u_0(x) \text{ in } \Omega, \quad \nabla u \cdot n  = 0\text{ on }\partial\Omega \times (0,T).
\end{align*}
For $m > 1$, the PME has finite speed of propagation so that for compactly supported initial data the solution remains compactly supported also for later times, \cite{Vazquez2007_PME}. The solutions show a self-similar behaviour and in spatial dimension one, an explicit form, the Barenblatt solution, is known. It is given as
\begin{align}\label{eq:Barenblatt}
    u_b(x, t)=\left(t+t_{0}\right)^{-\frac{1}{m+1}}\left(C-\frac{m-1}{2 m(m+1)} x^{2}\left(t+t_{0}\right)^{-\frac{2}{m+1}}\right)_{+}^{\frac{1}{m-1}}, 
\end{align}
for $C, t_{0}>0$.
In Figure~\ref{fig:PME_baren} we show some Barenblatt profiles with $m=2$, $t_0=10^{-3}$, and $C=(3/16)^{1/3}$, together with their numerical approximations generated by the daJKO scheme using no data terms. For the numerical algorithm we used $\tau=5\times 10^{-4}$, $N_t=10$, $N_x=100$, $\delta_i=10^{-5}=tol$, $\lambda=0.2$, $\sigma=4\times 10^{-4}$ such that $\lambda\sigma\|A A^*\|_2=0.9<1$, which is required for the convergence theory \cite{Yan2018}. We observe a good match between the exact solution and its numerical approximation. Furthermore, the energy is decaying monotonically.
Here, we do not aim for a full convergence study of the JKO scheme, but refer to \cite{Carrillo2022_PrimalDual} for results in this direction, which might be carried over to our discretization. Instead, let us discuss the use of data to correct the gradient flow in the presence of uncertainties.

\begin{figure}
    \centering
    \includegraphics[width=0.49\textwidth]{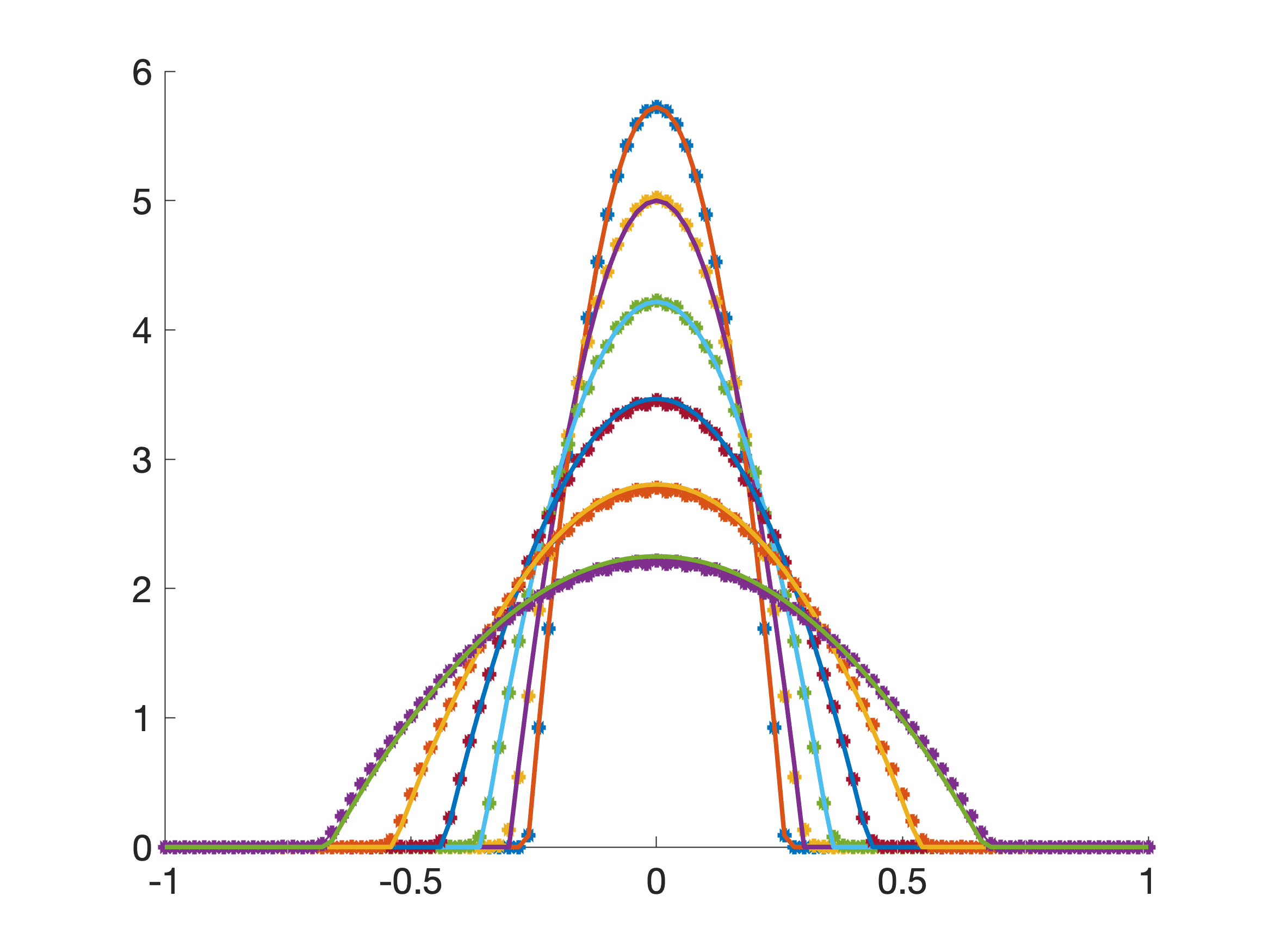}
    \includegraphics[width=0.49\textwidth]{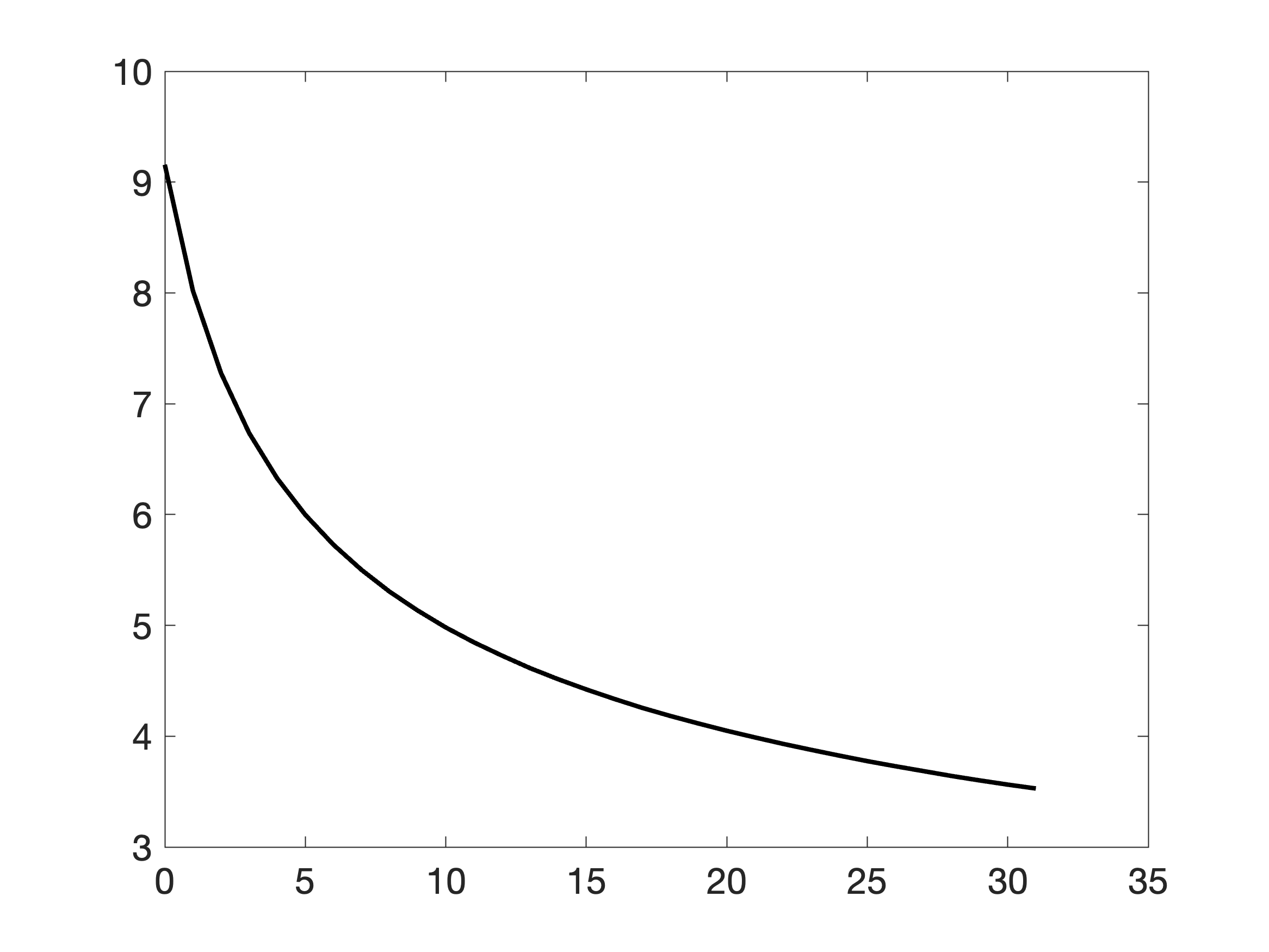}
    \caption{
    Left: Barenblatt profiles (solid) and their approximations (dotted) generated by the JKO scheme for different times $t=(k-1)\tau$ with $k=2^i$, $i\in\{0,\ldots,5\}$.
    Right: Energy decay during the JKO scheme for simulating a Barenblatt profile, i.e., without data terms. On average, Algorithm~\ref{alg:JKOStep} required $6000$ steps to converge.
    \label{fig:PME_baren}}
\end{figure}

\subsubsection{Unknown initial condition}
As a first test of our data assimilation framework, we assume that the true but unknown initial condition is given by \eqref{eq:Barenblatt} with $t=0$, $m=2$, $t_0=10^{-3}$, and $C=(3/16)^{1/3}$.
As a perturbed initial condition we employ a shifted Barenblatt solution
\begin{align}\label{eq:initial_shifted_space}
    u(x,0) = u_b(x-\bar x,\bar t), \text{ for some } \bar x, \bar t>0.
\end{align}
In order to correct for these perturbations, we assume a measurement operator $\B$, that employs the expected value and and the variance, respectively,
\begin{align*}
    \B^1(u) &= \int_\Omega x u(x)\dx,\\
    \B^2(u) &= \int_\Omega (x-\B_1(u))^2 u(x)\dx,
\end{align*}
which after employing the trapezoidal rule are discretized as follows
\begin{align}
    B^1_h(u_h) &= \sum_{j=1}^{N_x+1} w_j^x x_{j} \rho^1_j, \label{eq:observation_disc_expectation}\\
    B^2_h(u_h) &= \sum_{j=1}^{N_x+1} w_j^x (x_{j}-B_1(u_h))^2 \rho^1_j. \label{eq:observation_disc_variance}
\end{align}
Here, as above, we write $\rho^1_j=\rho_{j,N_t+1}$ for $u_h=(\rho_{j,k},m_{j,k})$.
The measurement operator is then given by $B_h(u_h)=(B_h^1(u_h),\vartheta B_h^2(u_h))$ with either $\vartheta=0$ or $\vartheta=1$, depending on whether or not variance data is used.
Similar to \eqref{eq:grad_Fh}, we compute the gradients of the measurement operator as follows
\begin{align*}
    (\nabla_u B_h^1(u_h))_{j,k} &= \frac{\delta_{k,N_t+1}}{w_k^t}\left( x_{j},0\right),\\
    (\nabla_u B_h^2(u_h))_{j,k} &= \frac{\delta_{k,N_t+1}}{w_k^t}\left( (x_{j}-B_h^1[u_h])^2 - 2 x_j\sum_{j'=0}^{N_x+1} w_{j'}^x(x_{j'}-B_h^1[u_h]) \rho^1_{j'},0\right).
\end{align*}
The penalty parameter is chosen $\theta=1/200$ in the following experiments.
%
%
\begin{figure}
    \centering
    \includegraphics[width=0.32\textwidth]{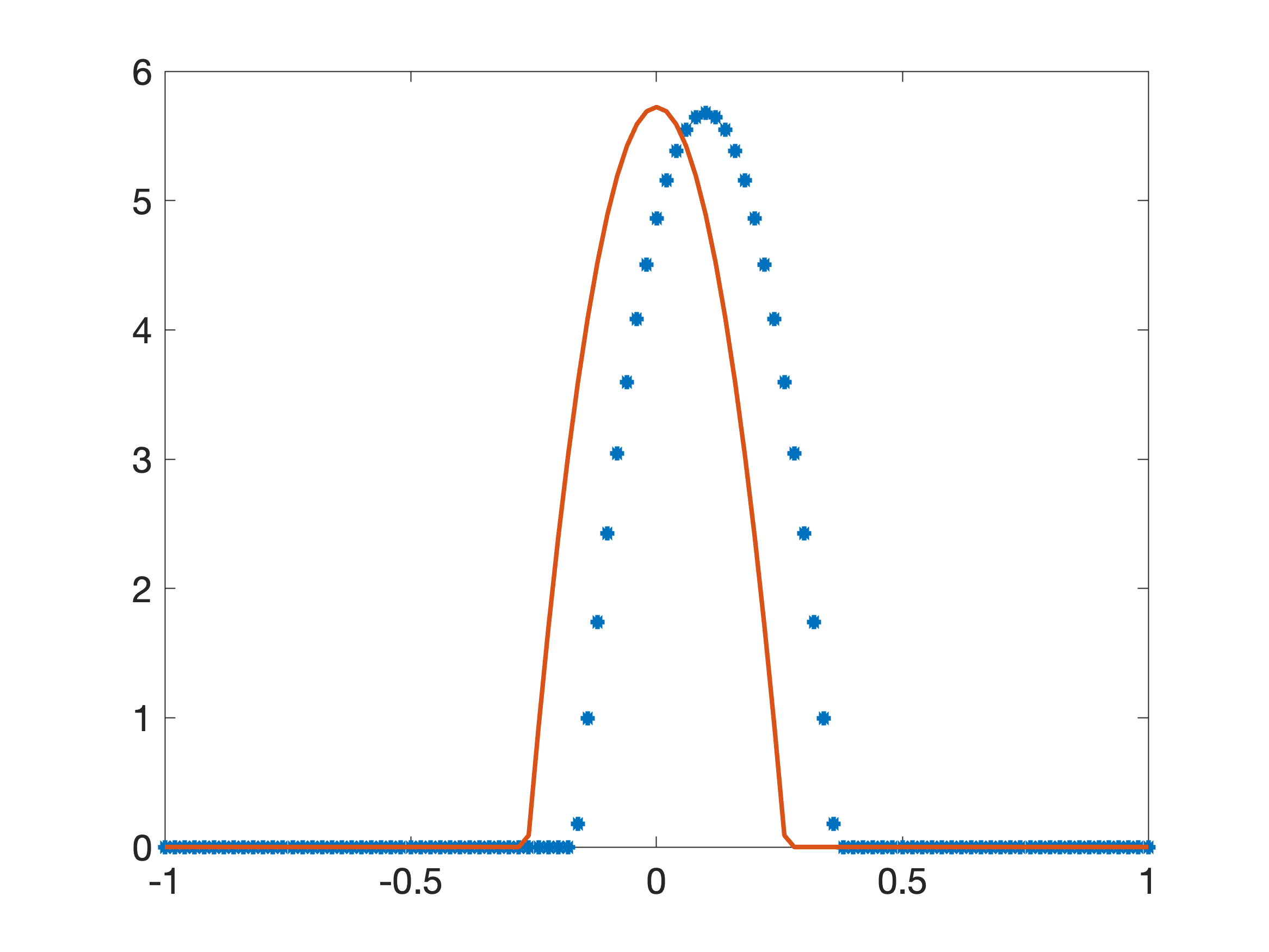}
    \includegraphics[width=0.32\textwidth]{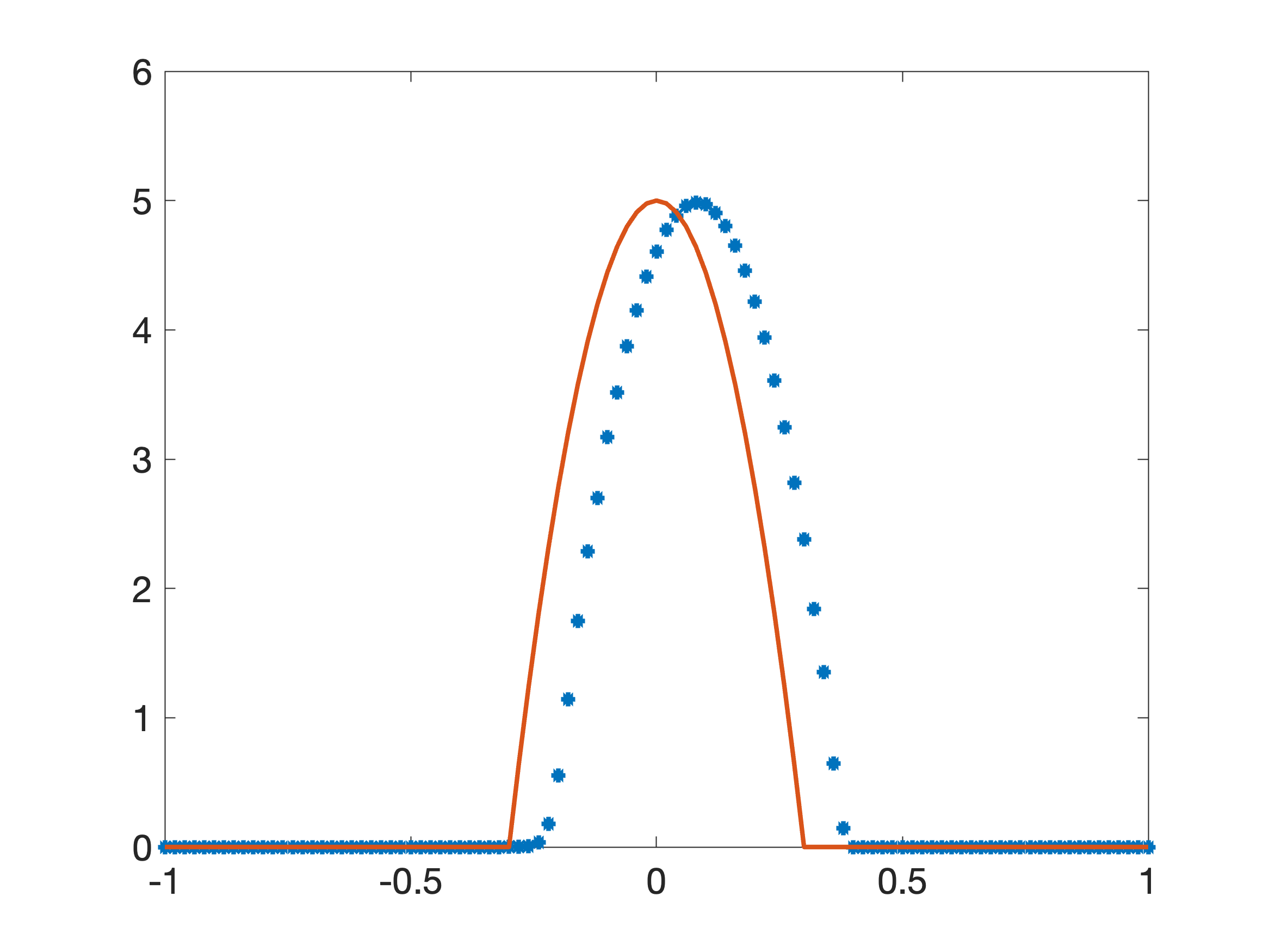}
    \includegraphics[width=0.32\textwidth]{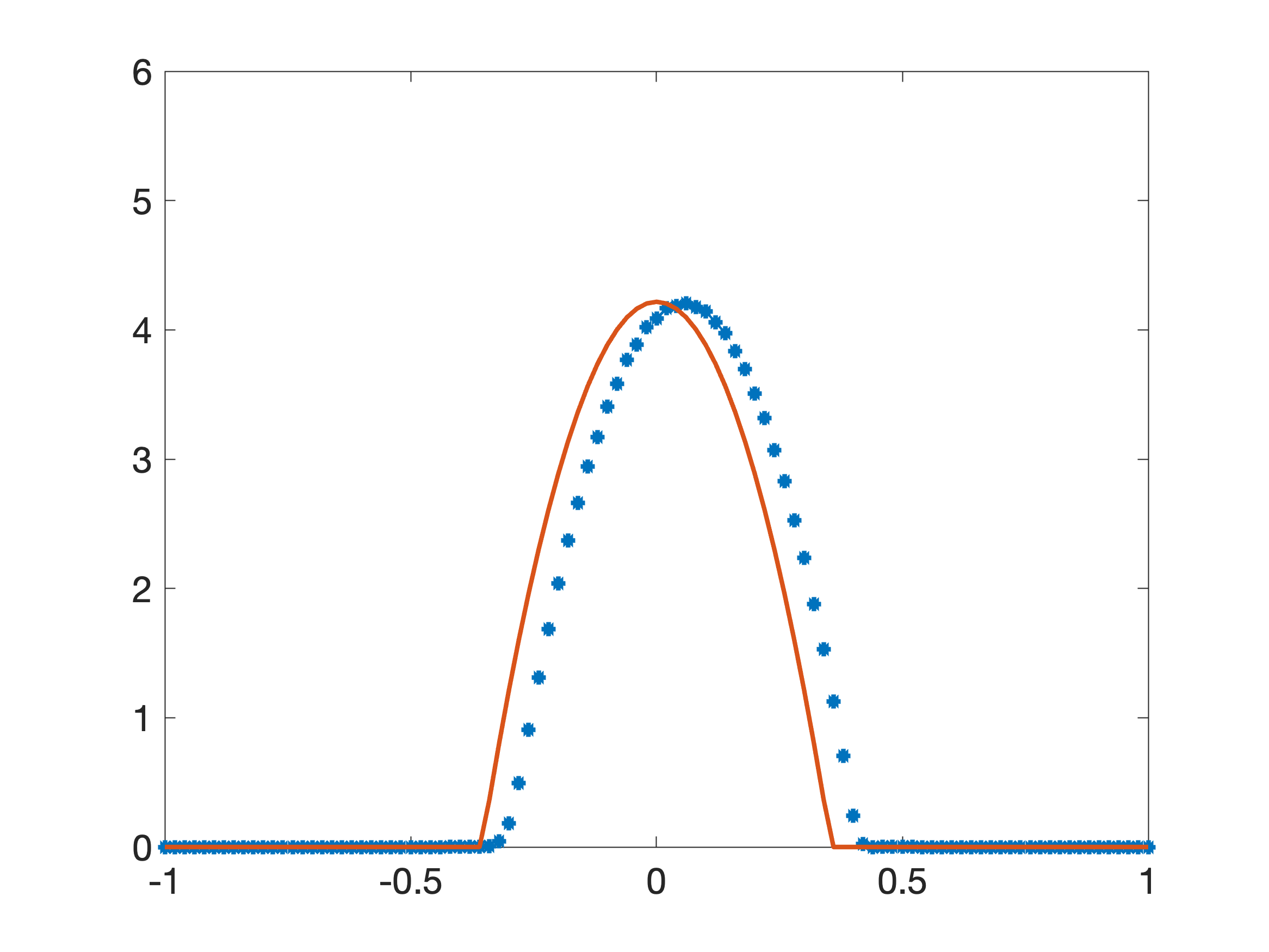}\\
    \includegraphics[width=0.32\textwidth]{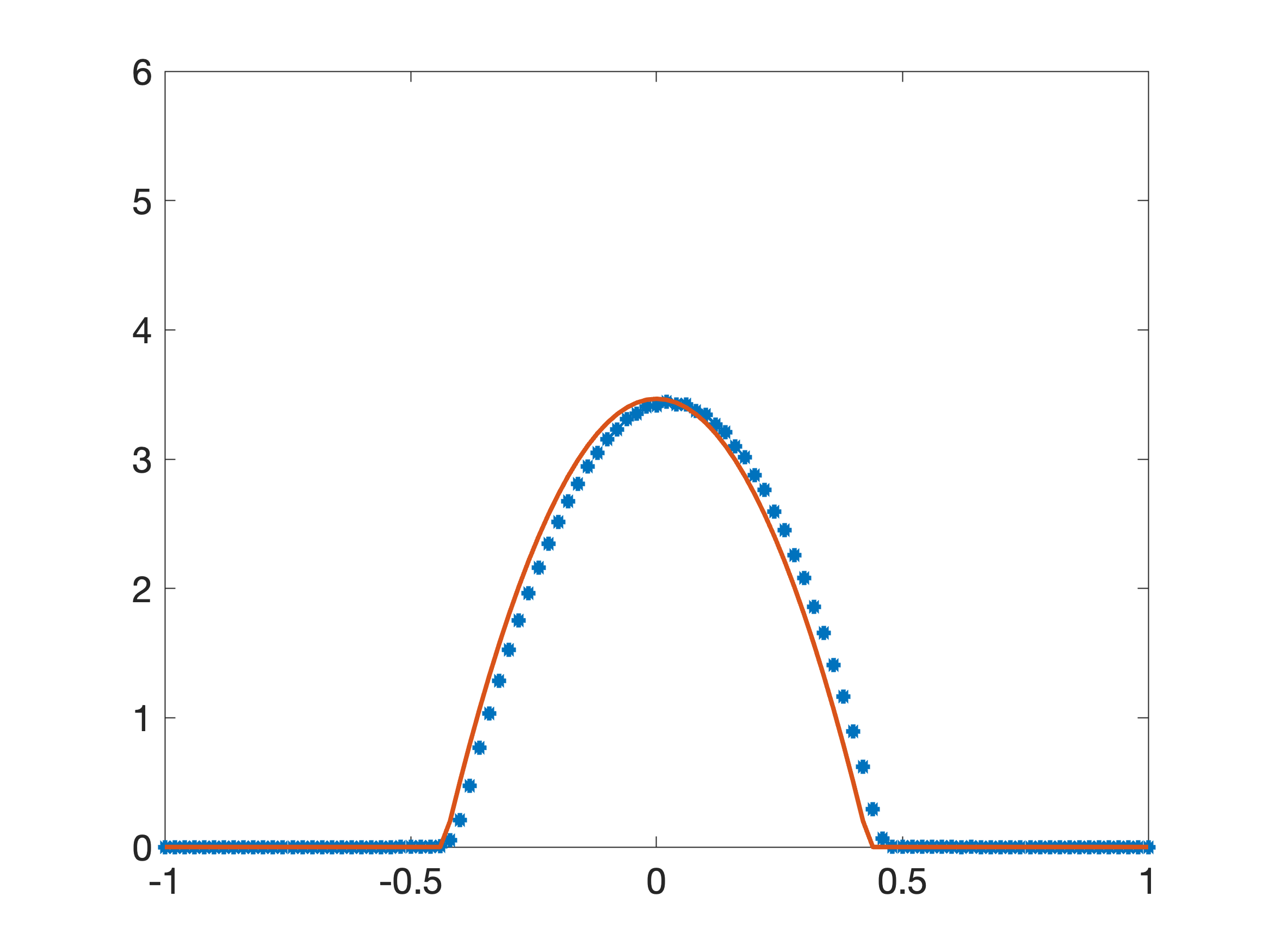}
    \includegraphics[width=0.32\textwidth]{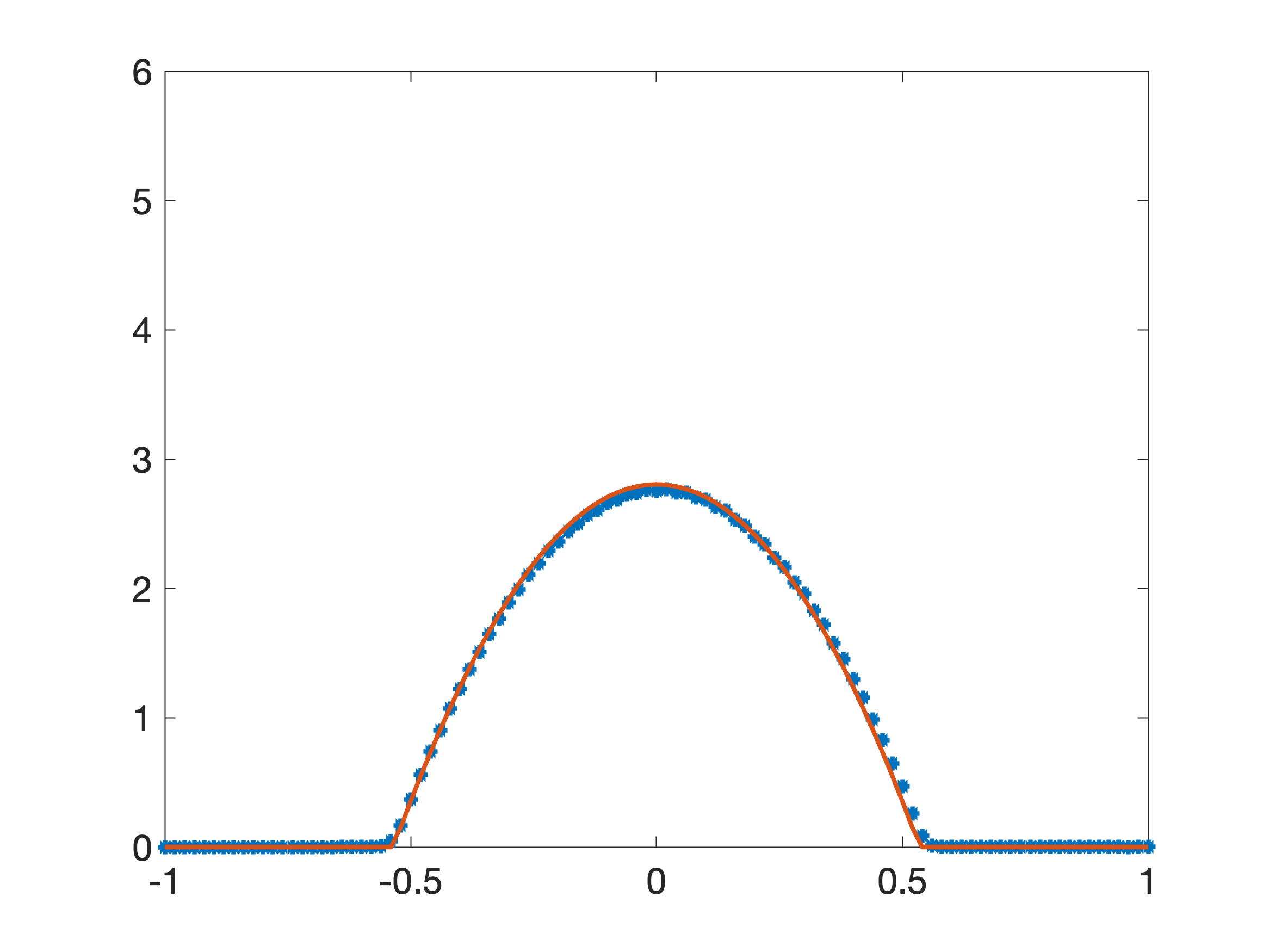}
    \includegraphics[width=0.32\textwidth]{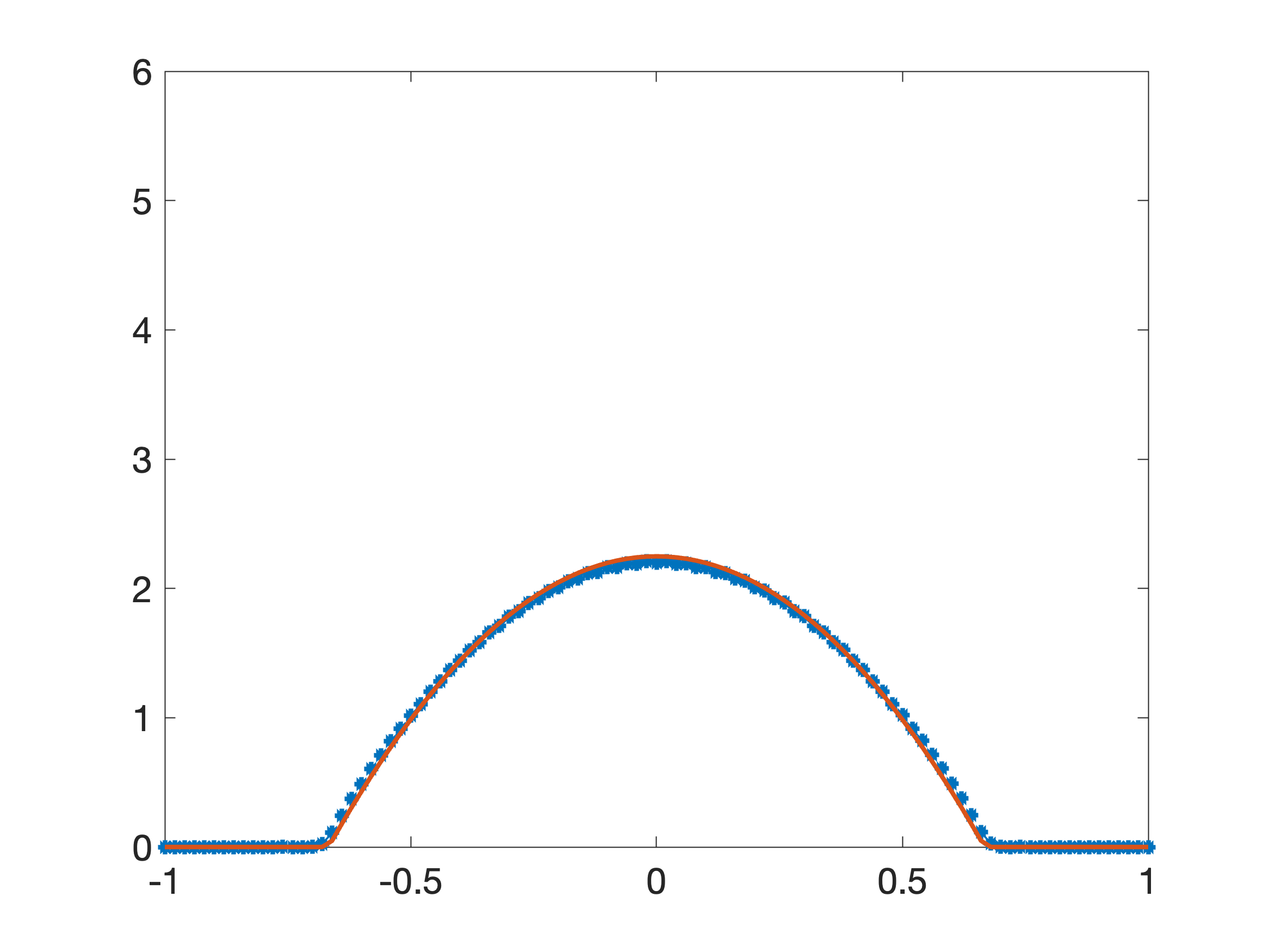}
    \caption{Barenblatt solution (solid red) and its approximation (dotted blue) using a shifted initial condition \eqref{eq:initial_shifted_space} with $\bar x=0.1$ and $\bar t=0$ and expectation $B_h^1(u_h)$ for the data term ($\vartheta=0$). The results are shown for $k=2^i$ for $i=0,\ldots,5$ from top left to bottom right. \label{fig:PME_data_no_noise_expected_value}}
\end{figure}

In our first experiment, we use a perturbed initial condition for the \daJKO scheme, given by a spatially shifted profile as defined in \eqref{eq:initial_shifted_space} with $\bar x= 0.1$ and $\bar t=0$. In order to correct for this shift, we employ the expected value of the ground truth, i.e., $\vartheta=0$.
Figure~\ref{fig:PME_data_no_noise_expected_value} shows that the \daJKO scheme is able to steer the numerical solution to the exact solution, thereby correcting for the perturbation in the initial condition. This might be expected, because the expected values effectively determines the maximum of the Barenblatt profile.

\begin{figure}
    \centering
    \includegraphics[width=0.32\textwidth]{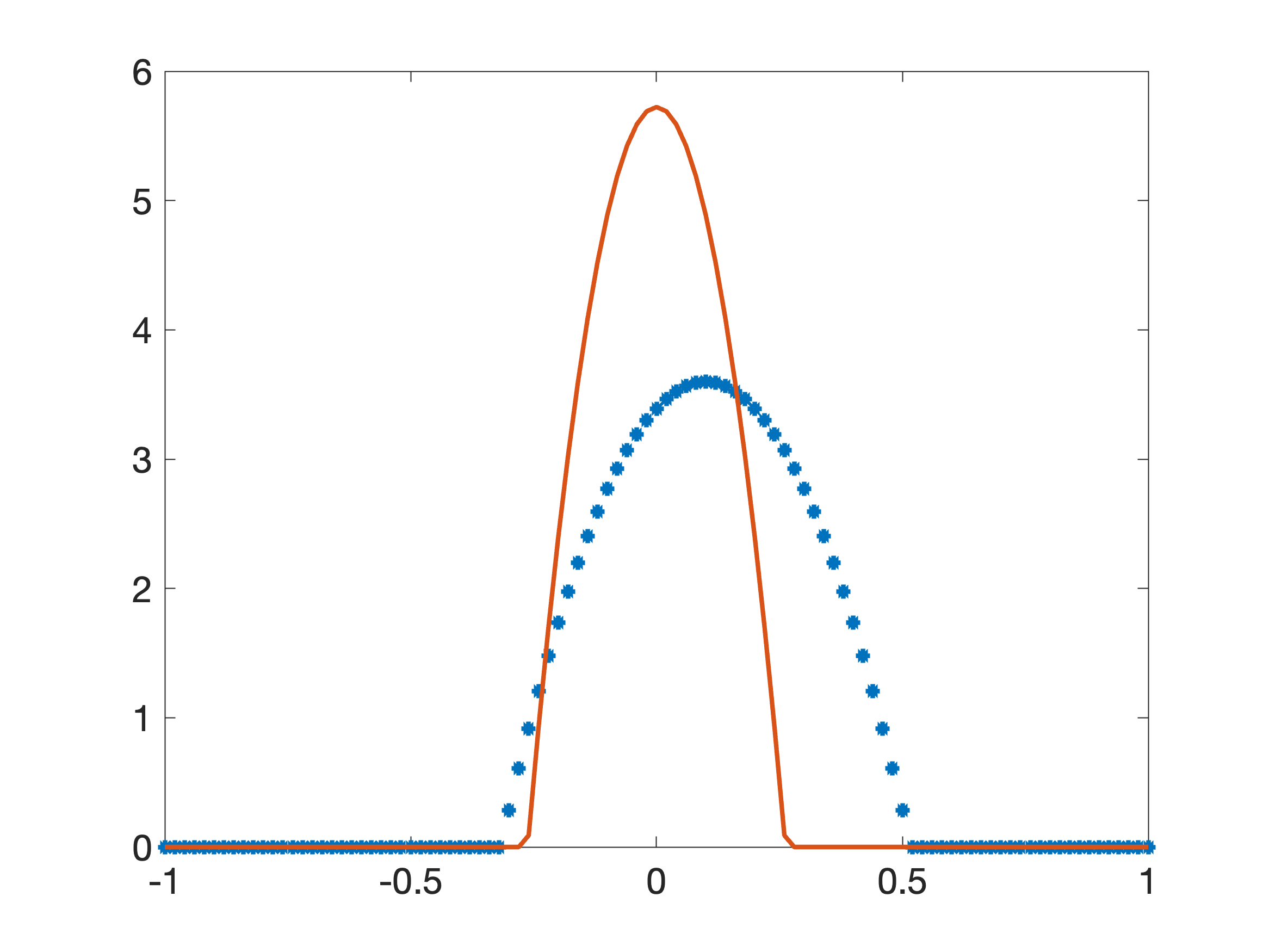}
    \includegraphics[width=0.32\textwidth]{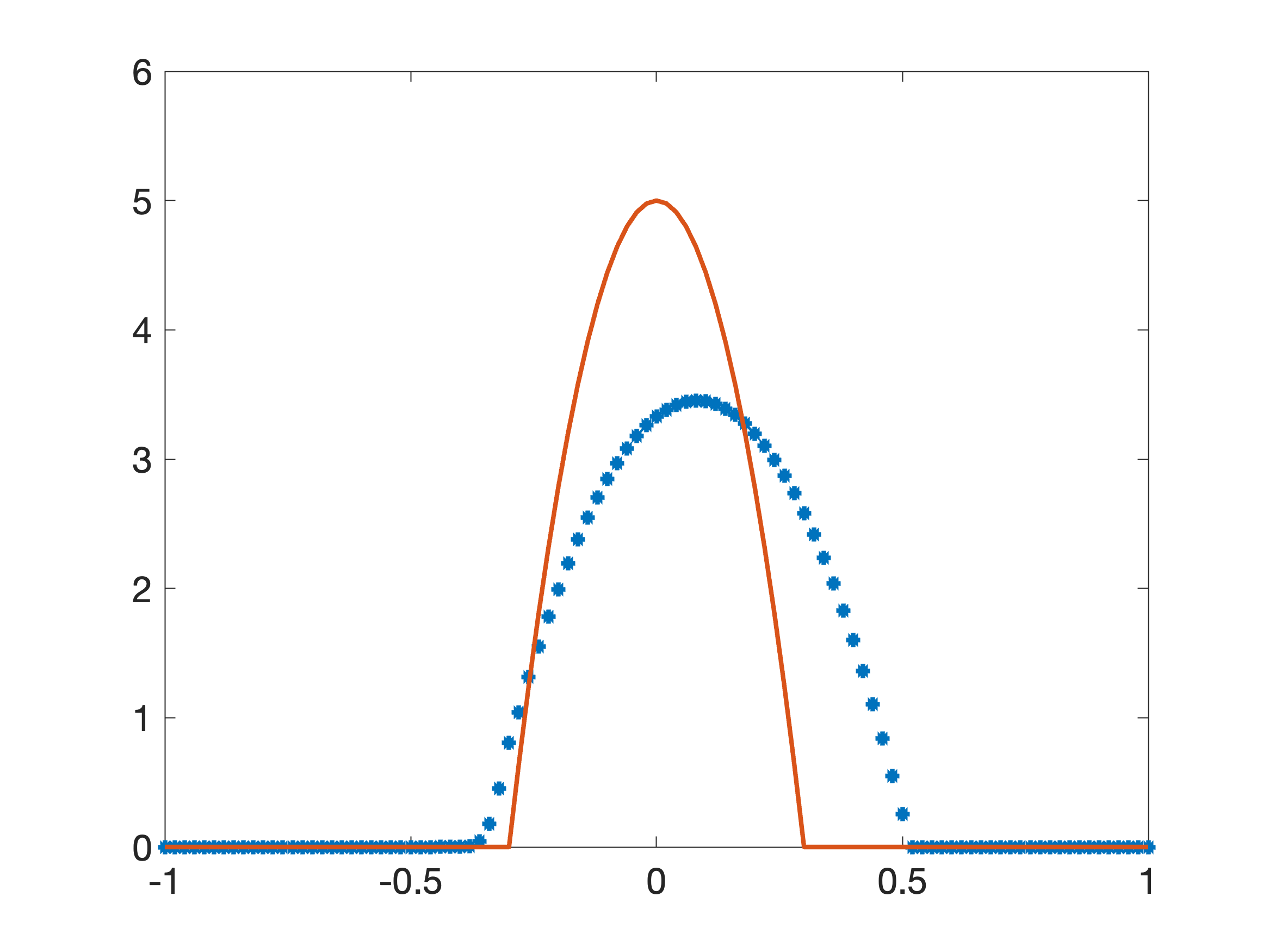}
    \includegraphics[width=0.32\textwidth]{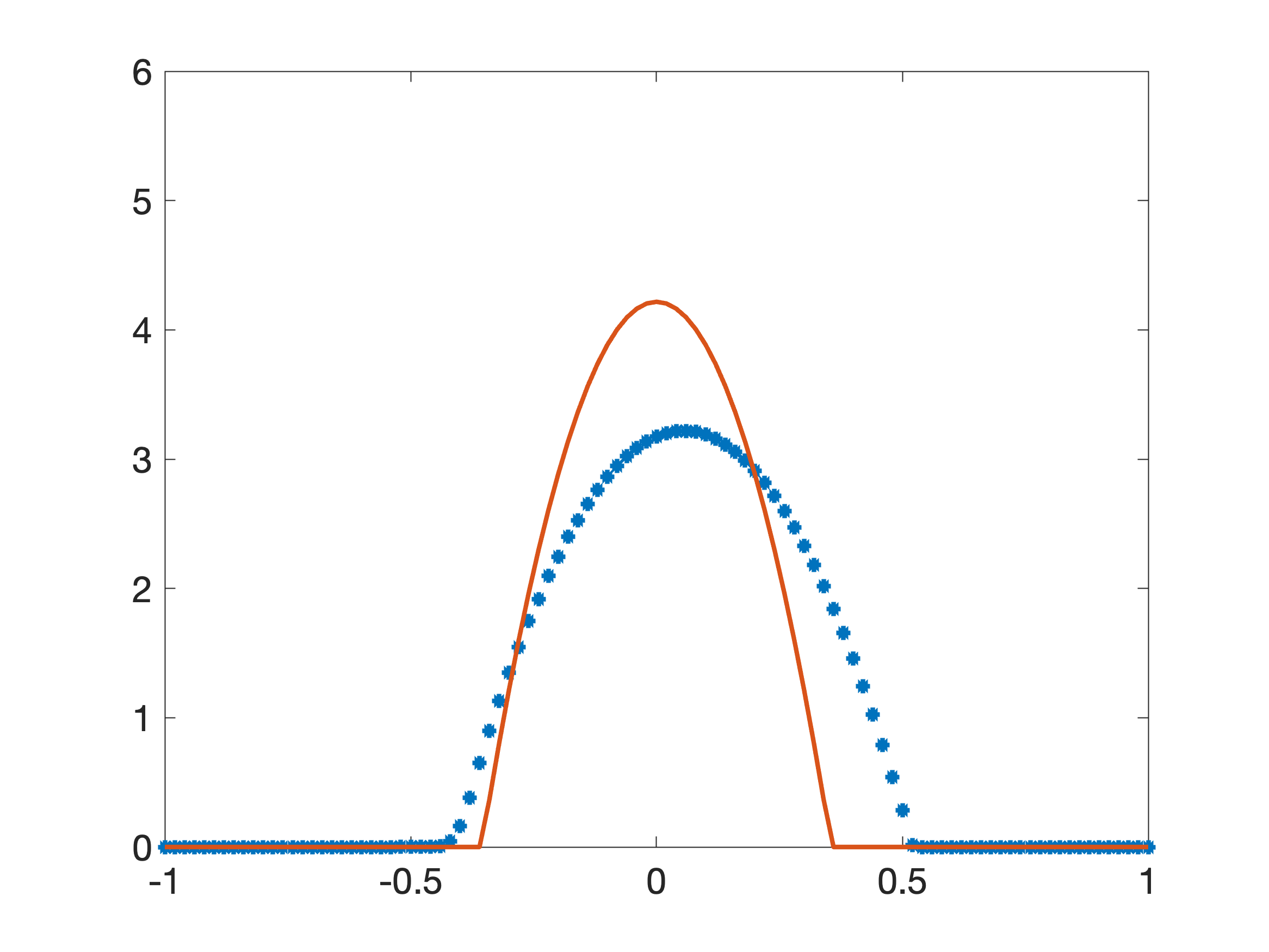}\\
    \includegraphics[width=0.32\textwidth]{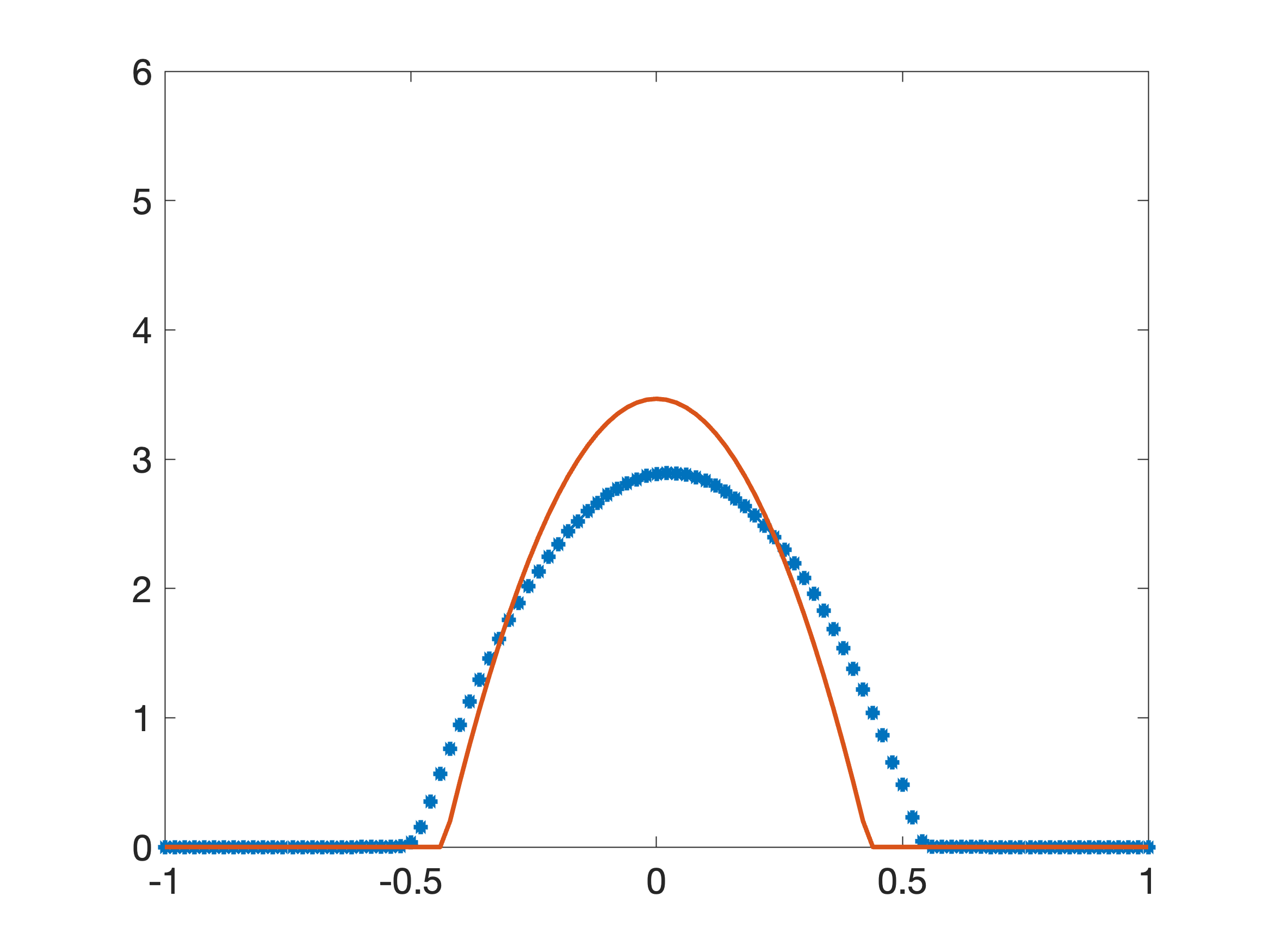}
    \includegraphics[width=0.32\textwidth]{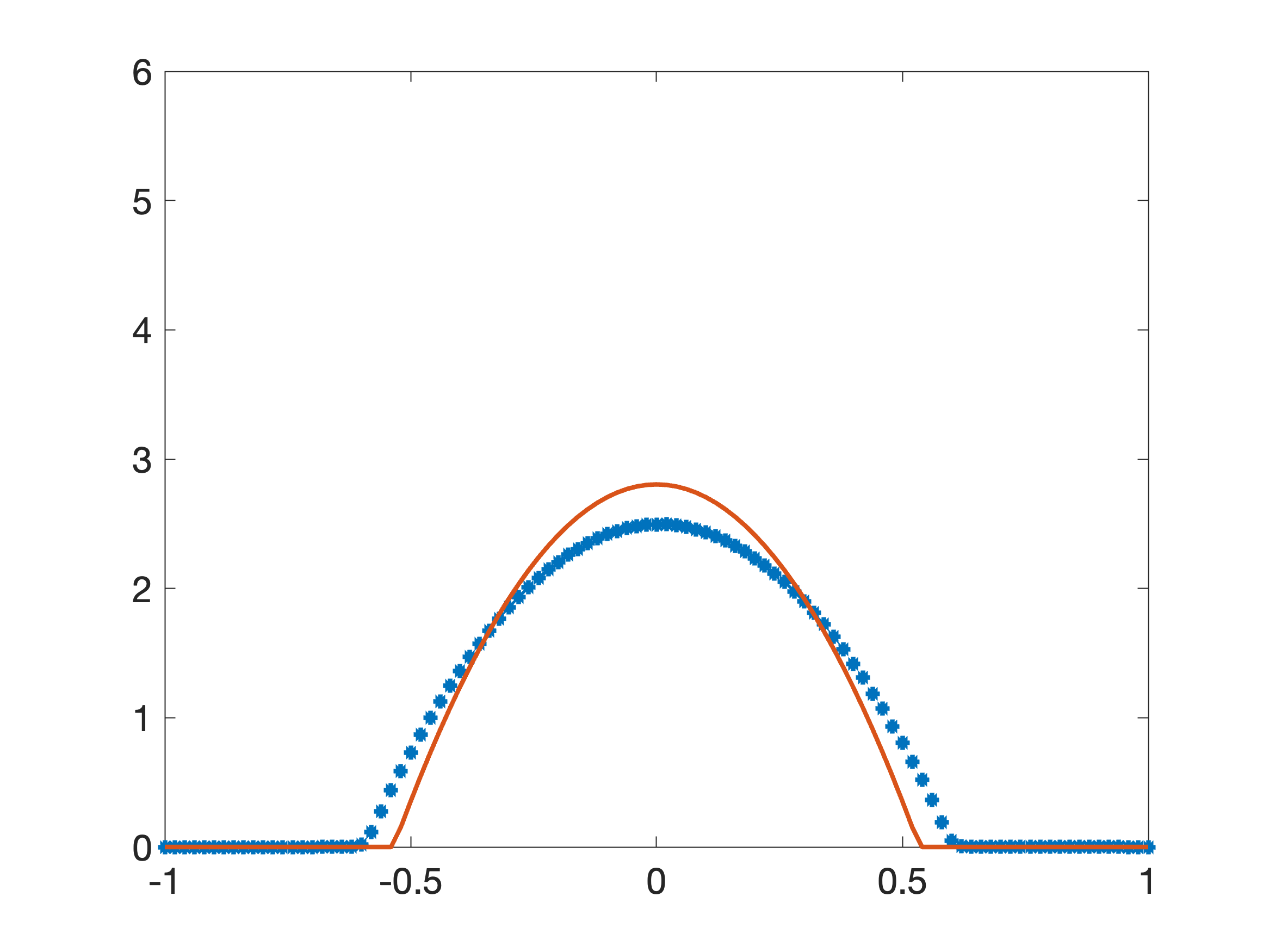}
    \includegraphics[width=0.32\textwidth]{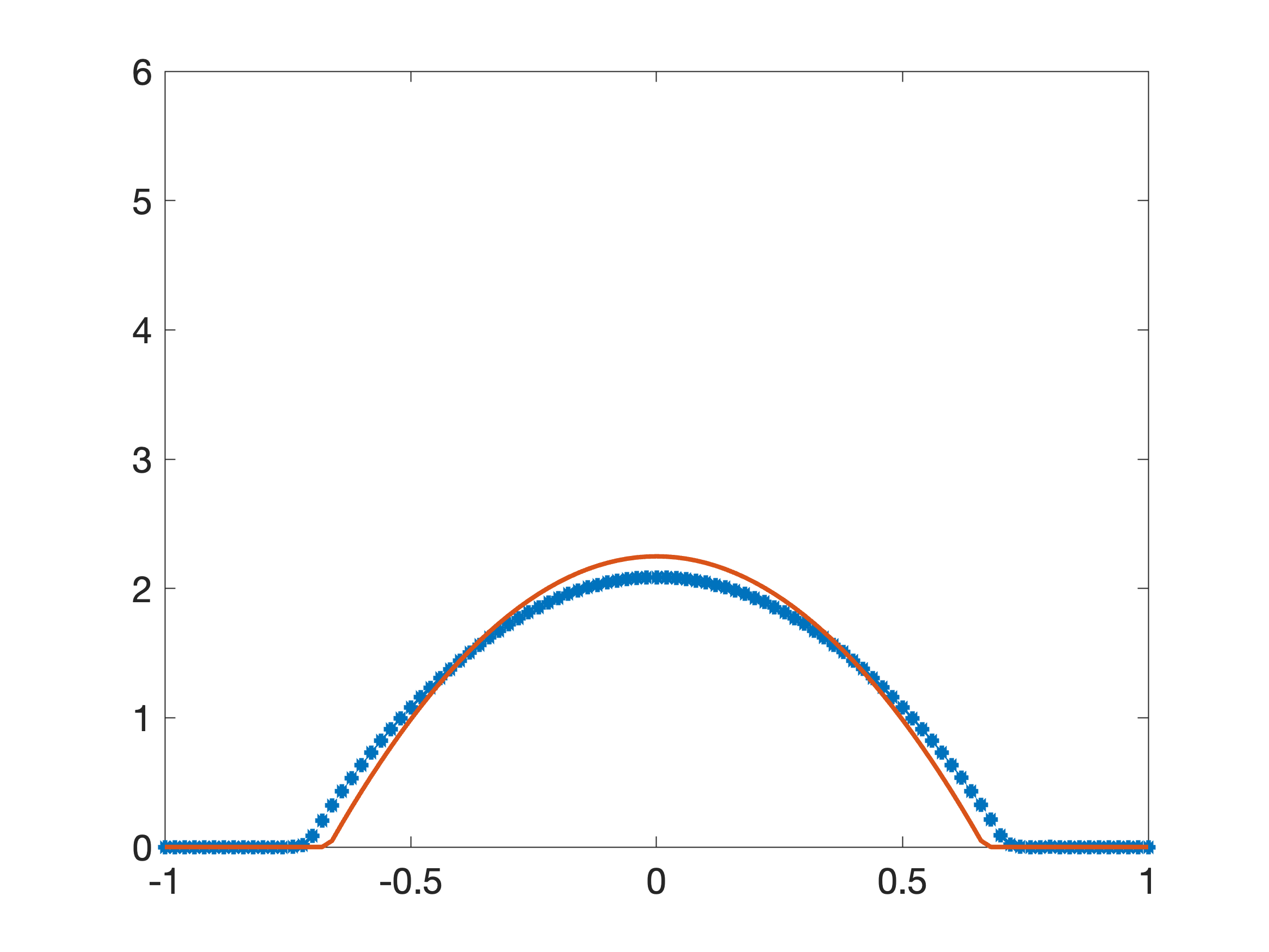}
    \caption{The same results as shown in Figure~\ref{fig:PME_data_no_noise_expected_value} but with a shifted initial condition \eqref{eq:initial_shifted_space} with $\bar x=0.1$ and $\bar t=6\tau$.
    \label{fig:PME_data_no_noise_expected_value_double_shift}}
\end{figure}

In the next experiment, we add to the spatial shift $\bar x=0.1$ also a temporal perturbation $\bar t=6\tau$ for the initial condition \eqref{eq:initial_shifted_space} of the \daJKO scheme.
In this situation, using only the expected value as data, we expect that the \daJKO scheme can again compensate the shift. This is in agreement with the results shown in Figure~\ref{fig:PME_data_no_noise_expected_value_double_shift}. It seems that the numerical approximation generated by the \daJKO scheme 'waits' for the true solution to arrive while aligning its peak with the one of the ground truth.
The matching is, however, not as good as in the previous experiment, in which no time shift has been applied.

\begin{figure}
    \centering
    \includegraphics[width=0.32\textwidth]{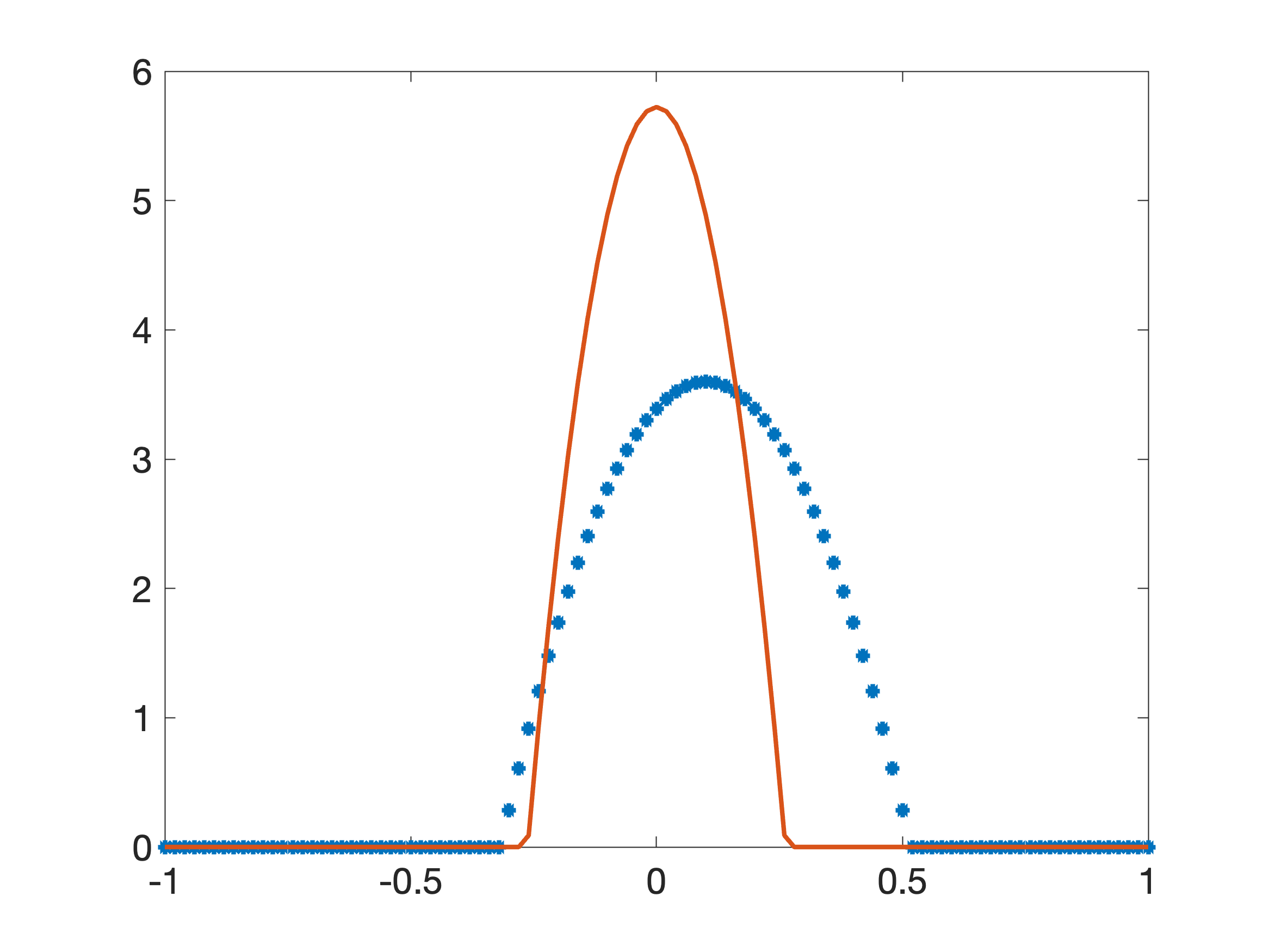}
    \includegraphics[width=0.32\textwidth]{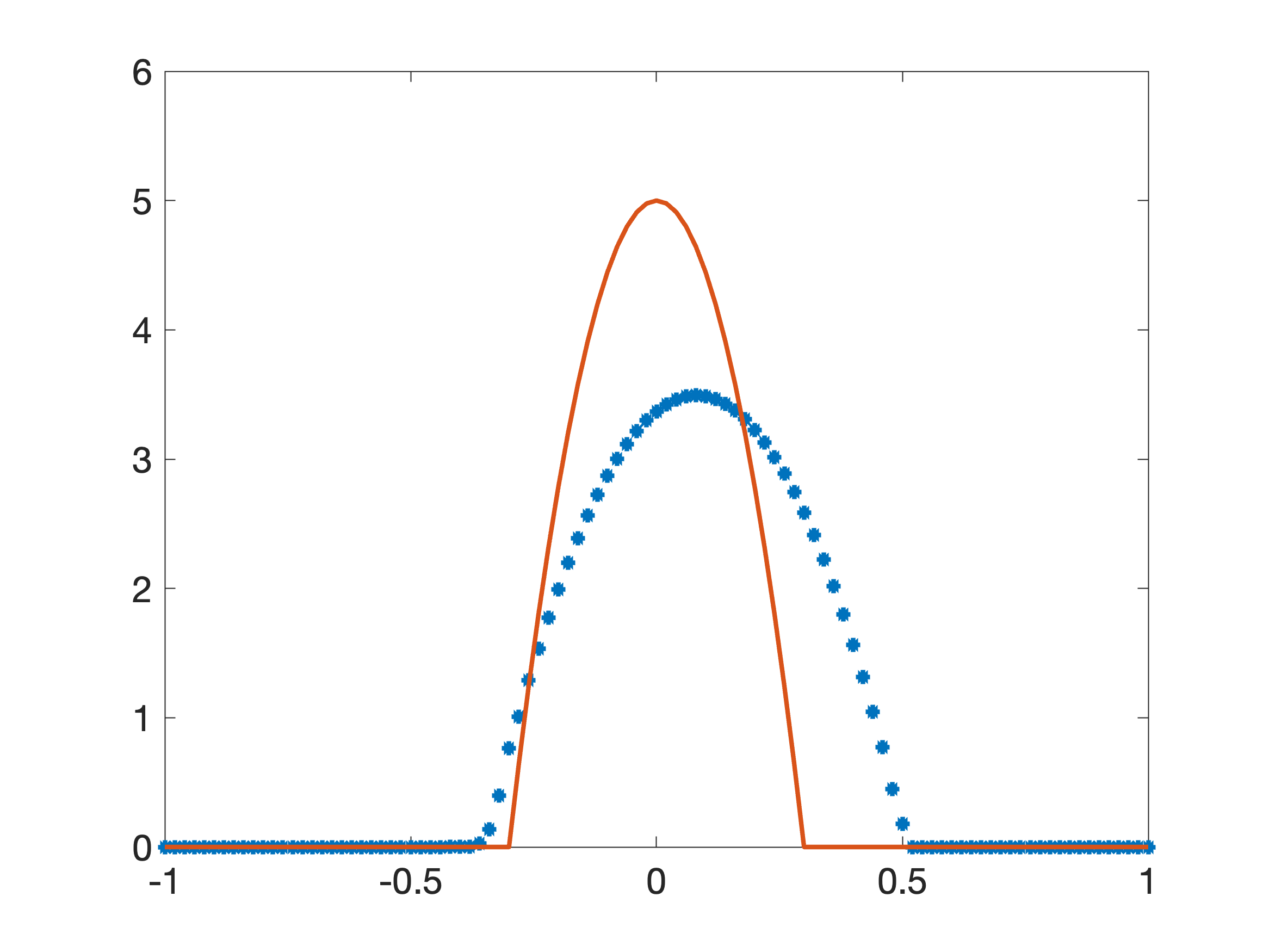}
    \includegraphics[width=0.32\textwidth]{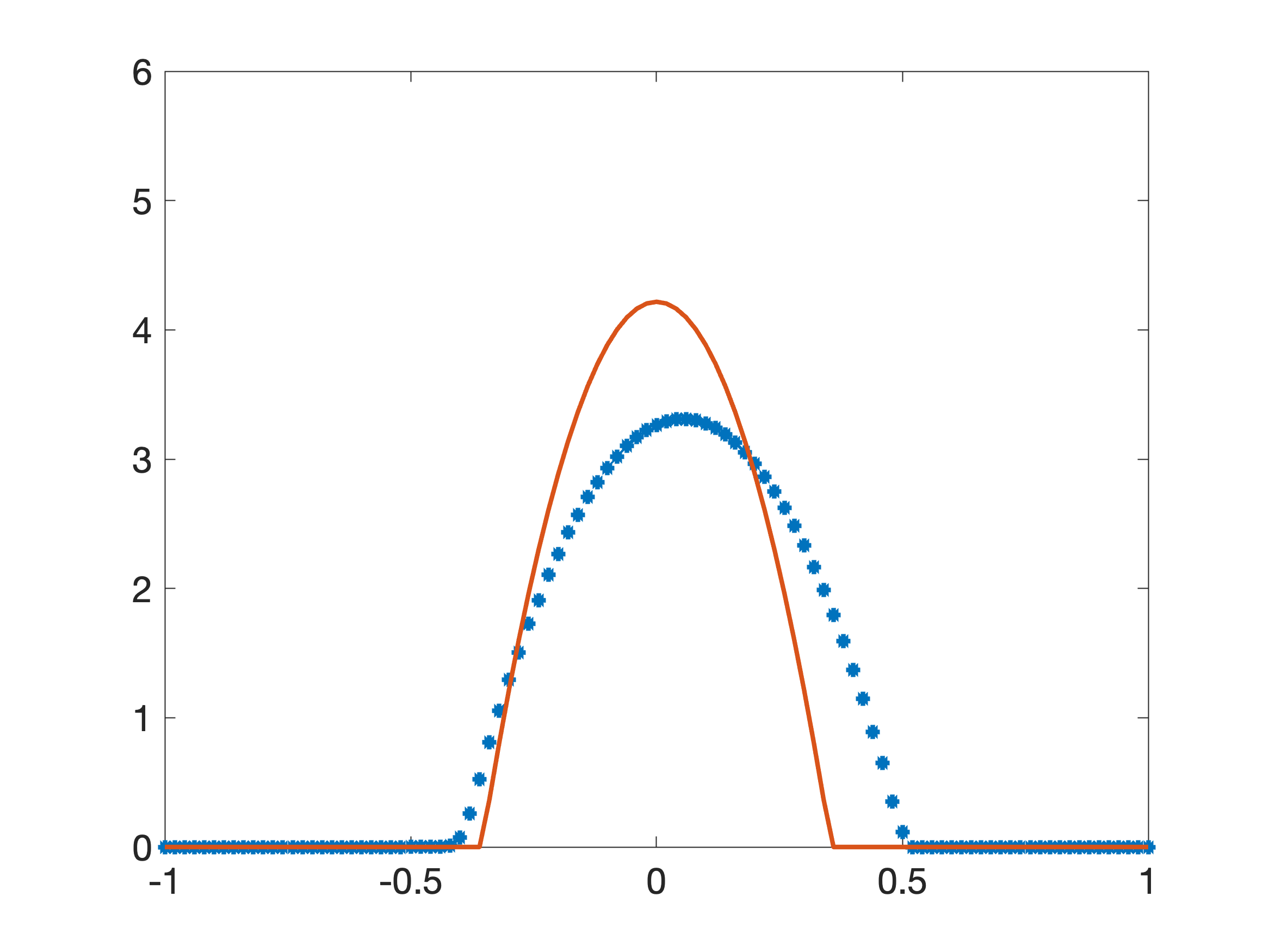}\\
    \includegraphics[width=0.32\textwidth]{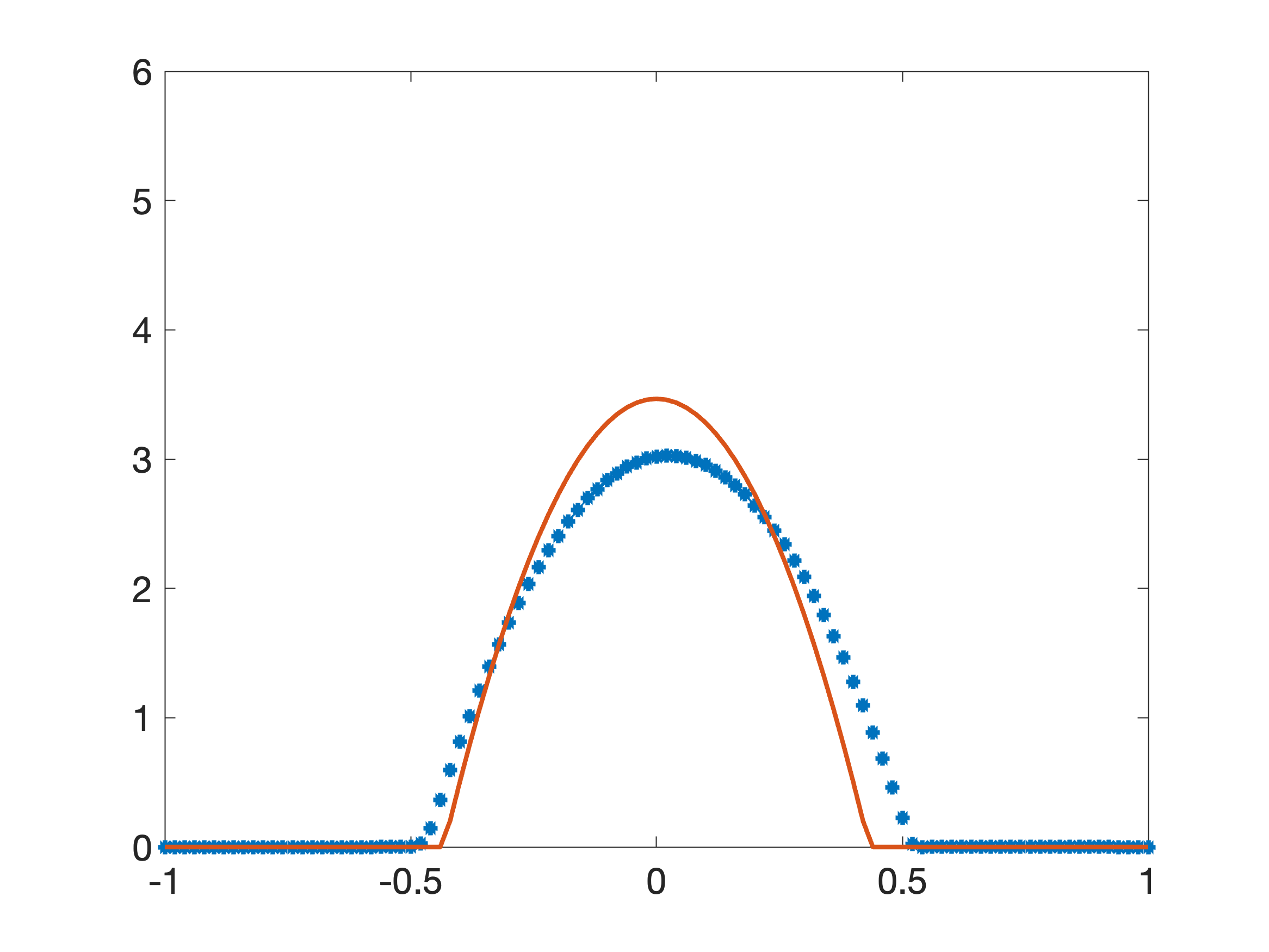}
    \includegraphics[width=0.32\textwidth]{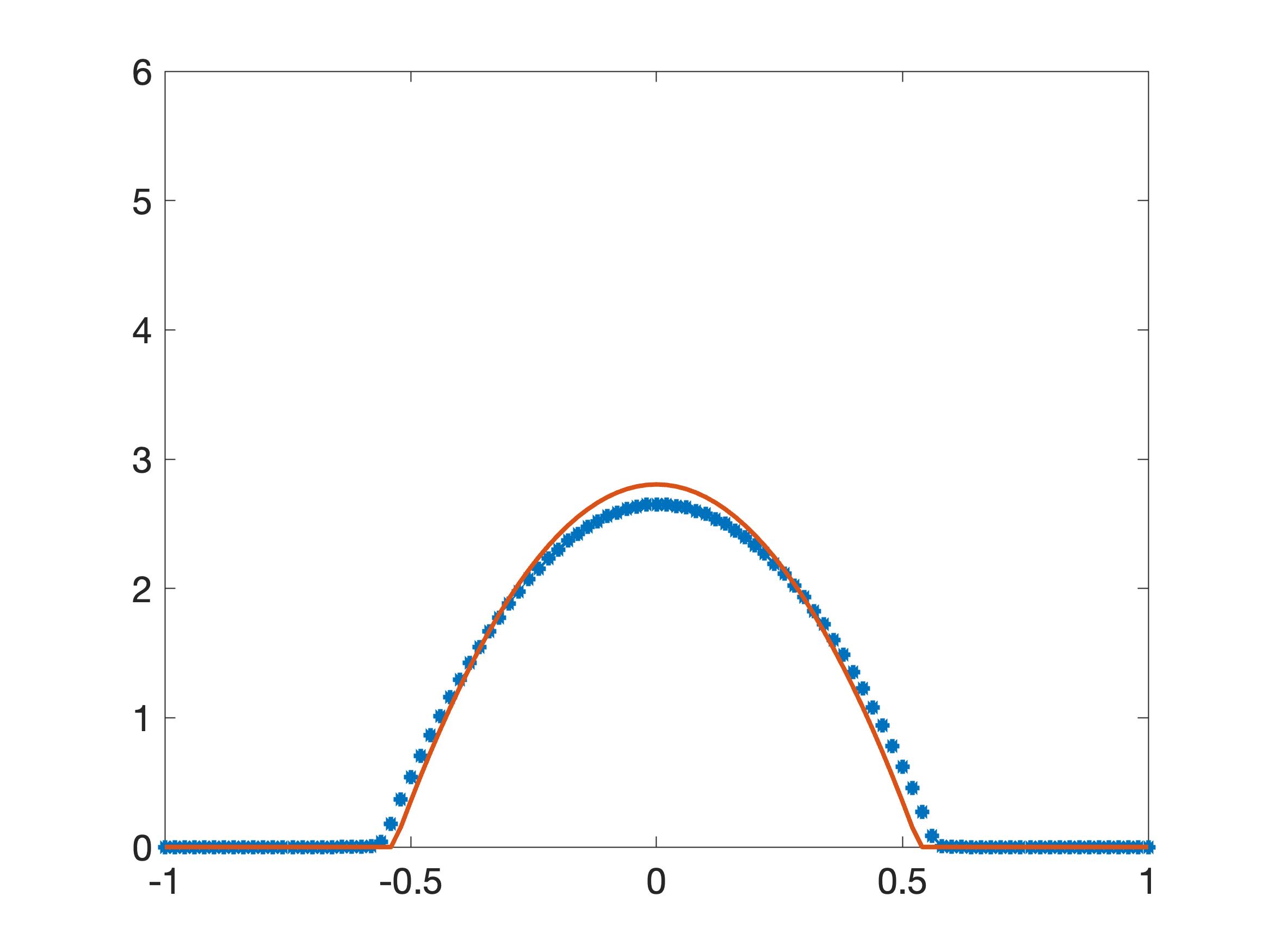}
    \includegraphics[width=0.32\textwidth]{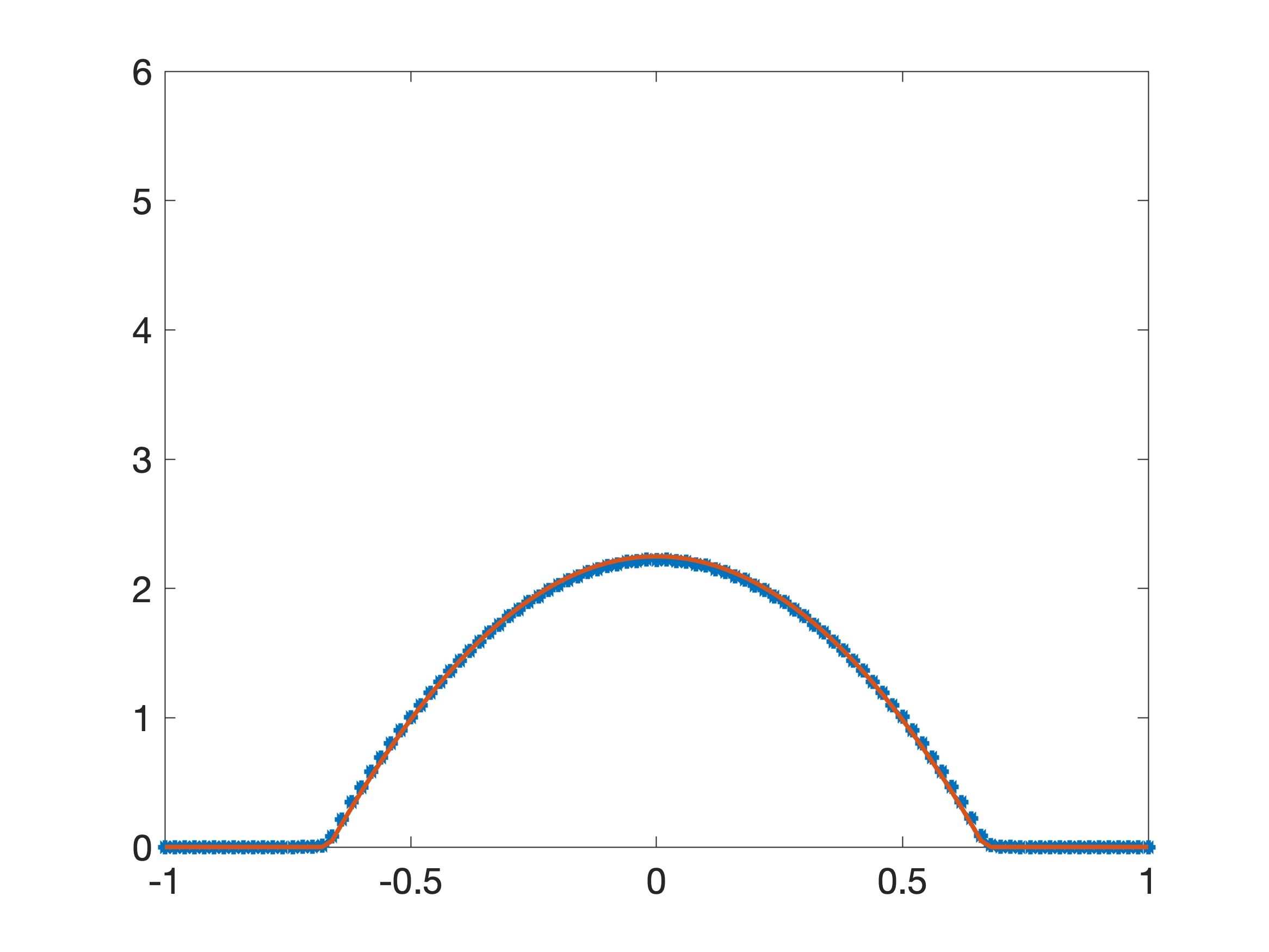}
    \caption{
    The same results as shown in Figure~\ref{fig:PME_data_no_noise_expected_value_double_shift} but using also variance data $B_h^2(u_h)$ as data ($\vartheta=1$).
    \label{fig:PME_data_no_noise_expected_value_and_variance_double_shift}}
\end{figure}

In order to improve the data fitting, we repeat the experiment for $\bar x=0.1$ and $\bar t=6\tau$, using additional variance data of the ground truth, i.e., $\vartheta=1$ in the definition of $B_h$ after \eqref{eq:observation_disc_variance}. Given that the value of $m$ is known to us, we can determine the Barenblatt profile by its expected value and variance. Therefore, we expect that the \daJKO scheme can match the true profile better, when the expected value and the variance are used as measurements. This expectation is confirmed by the results shown in Figure~\ref{fig:PME_data_no_noise_expected_value_and_variance_double_shift}.

\subsubsection{Unknown $m$}\label{sec:perturbed_m}
In addition to perturbed initial data, we consider model perturbations as a next experiment. We assume that $m^\dagger=2$ describes the true dynamics, while the computational model employs the perturbed value $m=2.5$.
Except for $\theta=1/400$, we keep the parameters as in the previous section. In particular, we use a shift $\bar x=0.1$ and $\bar t=6\tau$ for the initial condition.

In our first experiment we use again the expected value and the variance of the Barenblatt profile \eqref{eq:Barenblatt} with the ground truth values $m^\dagger$, $\bar x=0$, $\bar t=0$.
The results of the \daJKO scheme are displayed in Figure~\ref{fig:PME_data_no_noise_expected_var_m}, together with the groundtruth Barenblatt profile and a Barenblatt profile for the perturbed parameter $m=2.5$. We observe that the \daJKO scheme is able to correctly centre the numerical solution. Furthermore, by inspecting the supports of the three functions, we observe that using available data can correct for the speed of propagation. Comparing to the results of the previous section, the numerical approximation does, however, not match the ground truth solution very well.
This can be explained by the fact that the Barenblatt profiles for $m=2.5$ and $m^\dagger=2$ have different masses.

\begin{figure}
    \centering
    \includegraphics[width=0.32\textwidth]{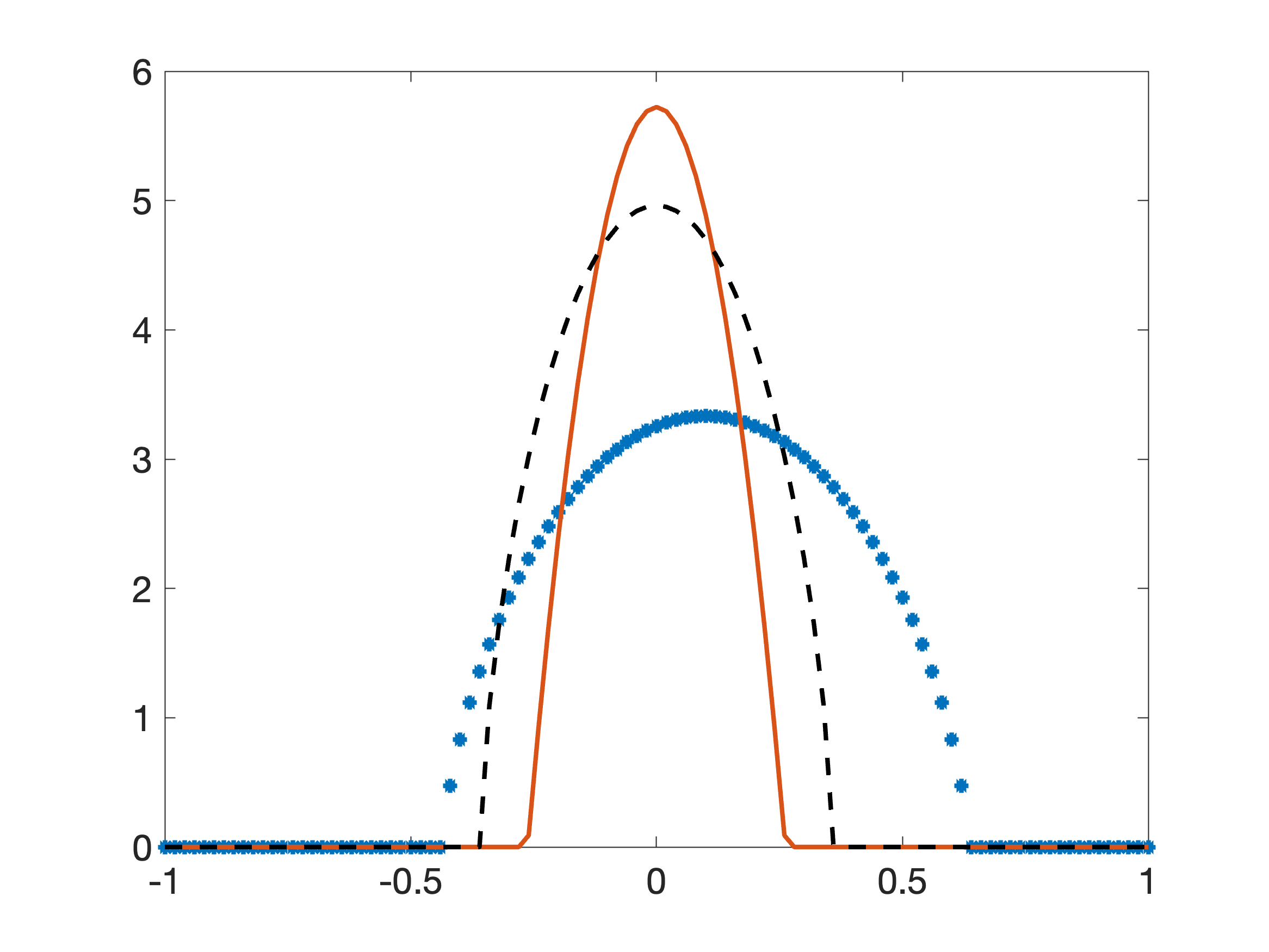}
    \includegraphics[width=0.32\textwidth]{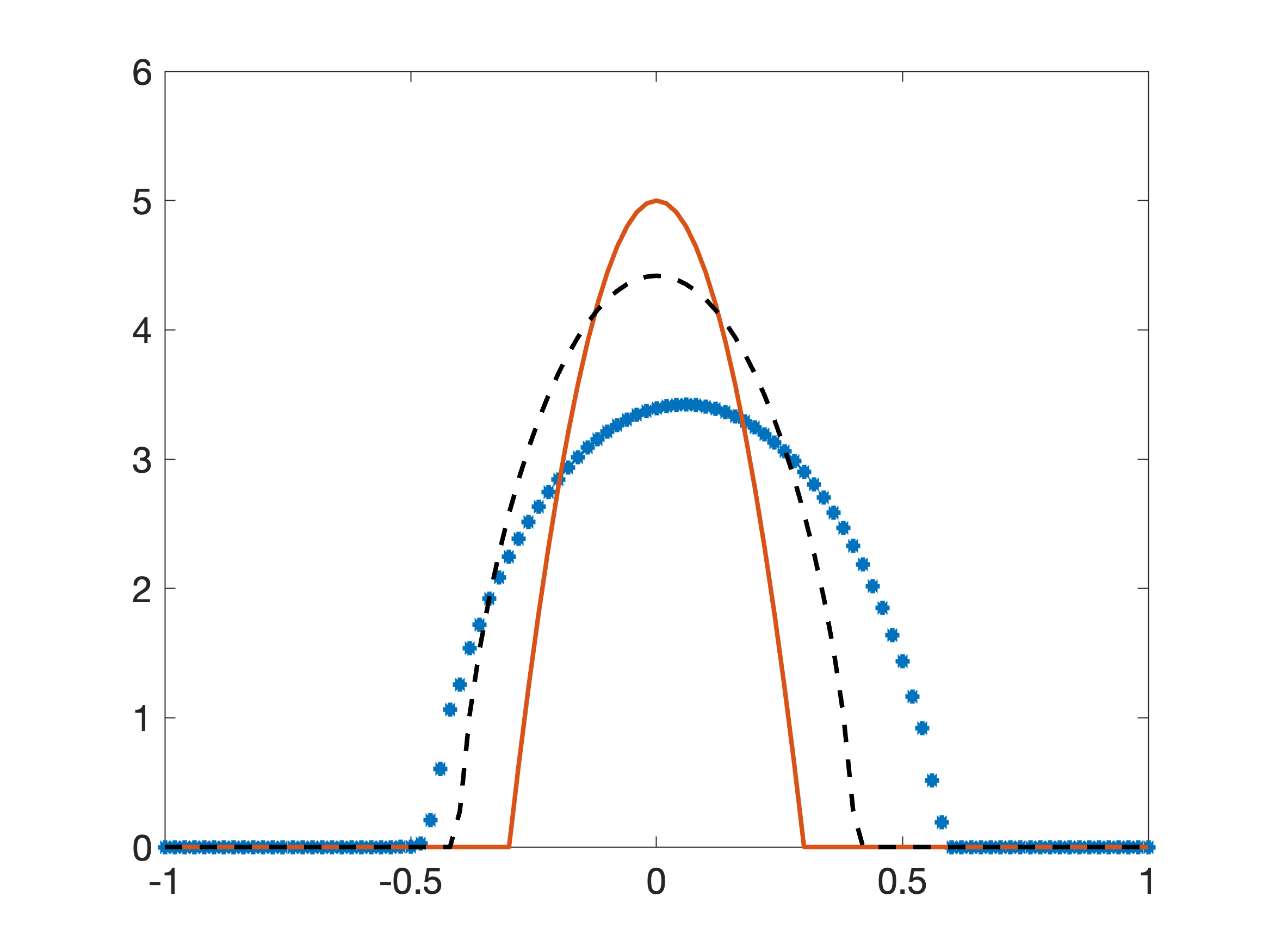}
    \includegraphics[width=0.32\textwidth]{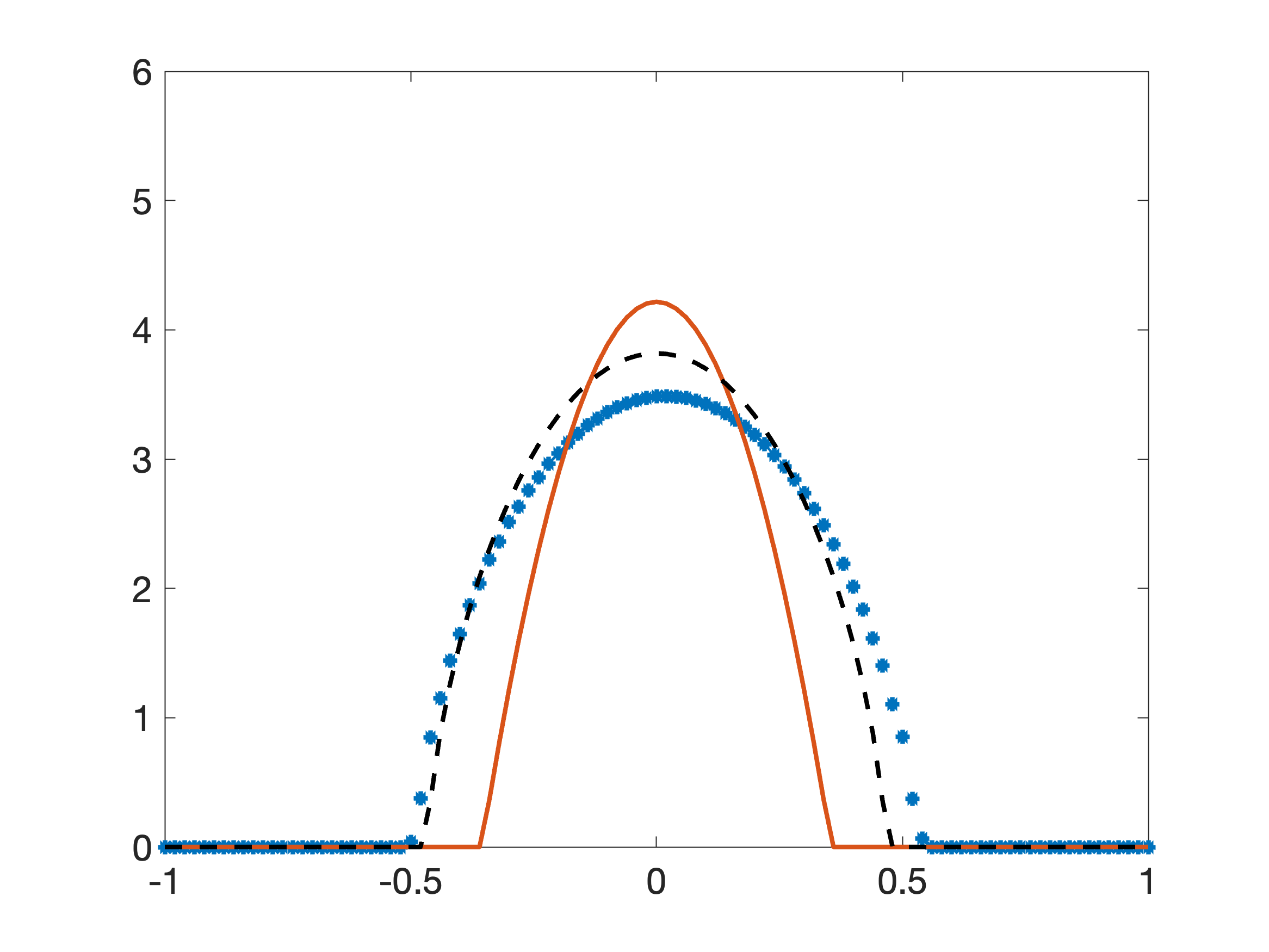}\\
    \includegraphics[width=0.32\textwidth]{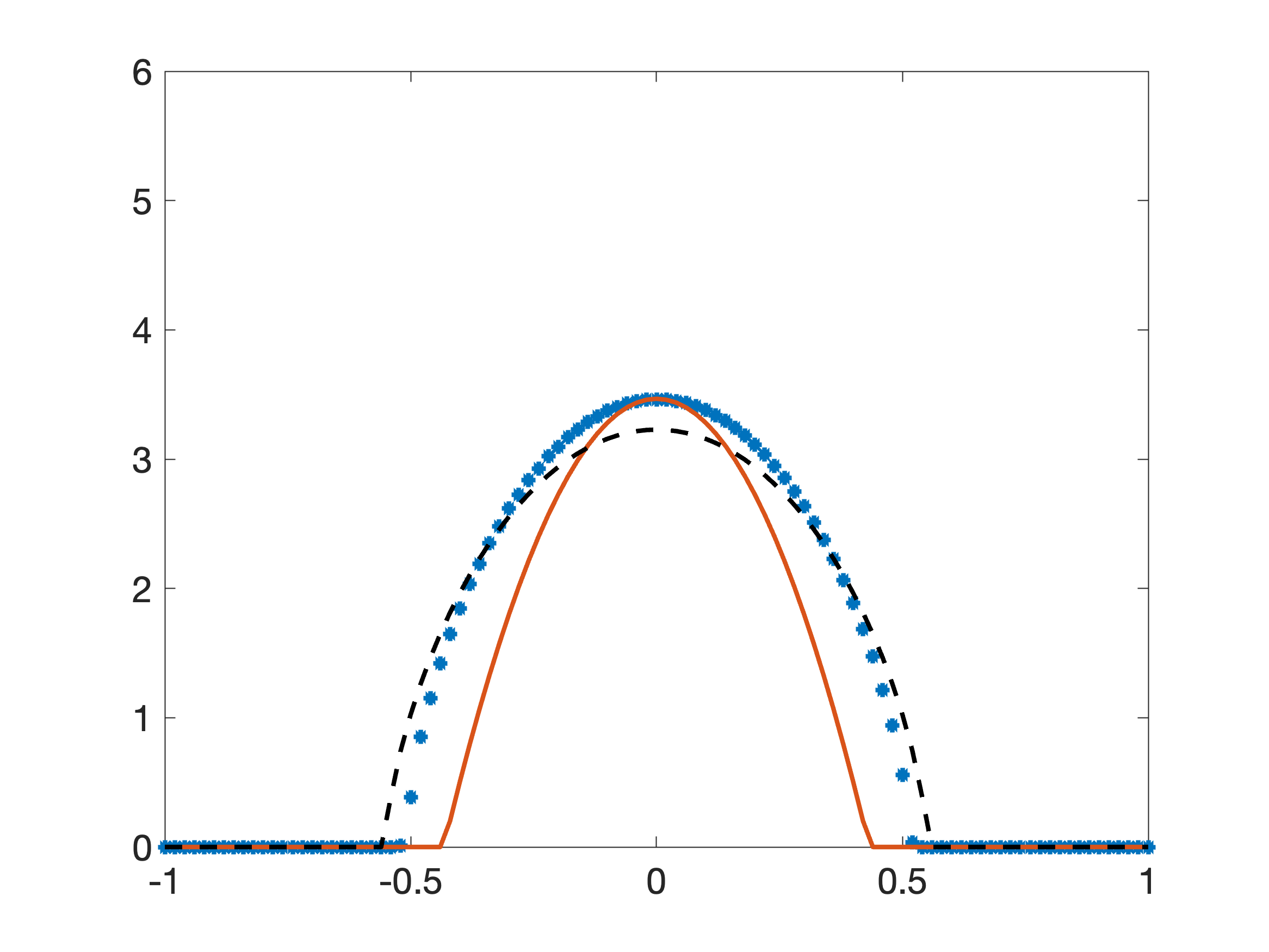}
    \includegraphics[width=0.32\textwidth]{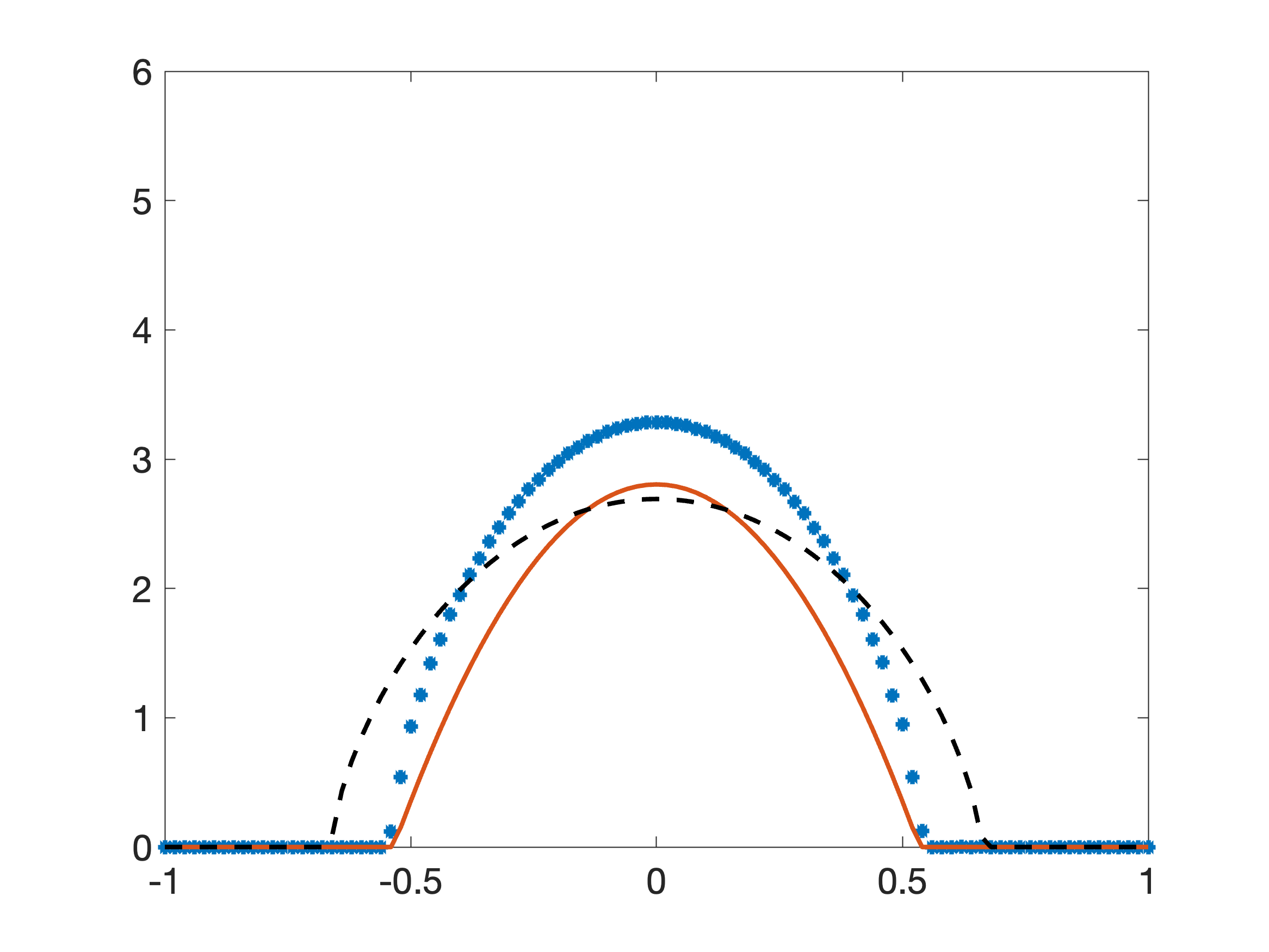}
    \includegraphics[width=0.32\textwidth]{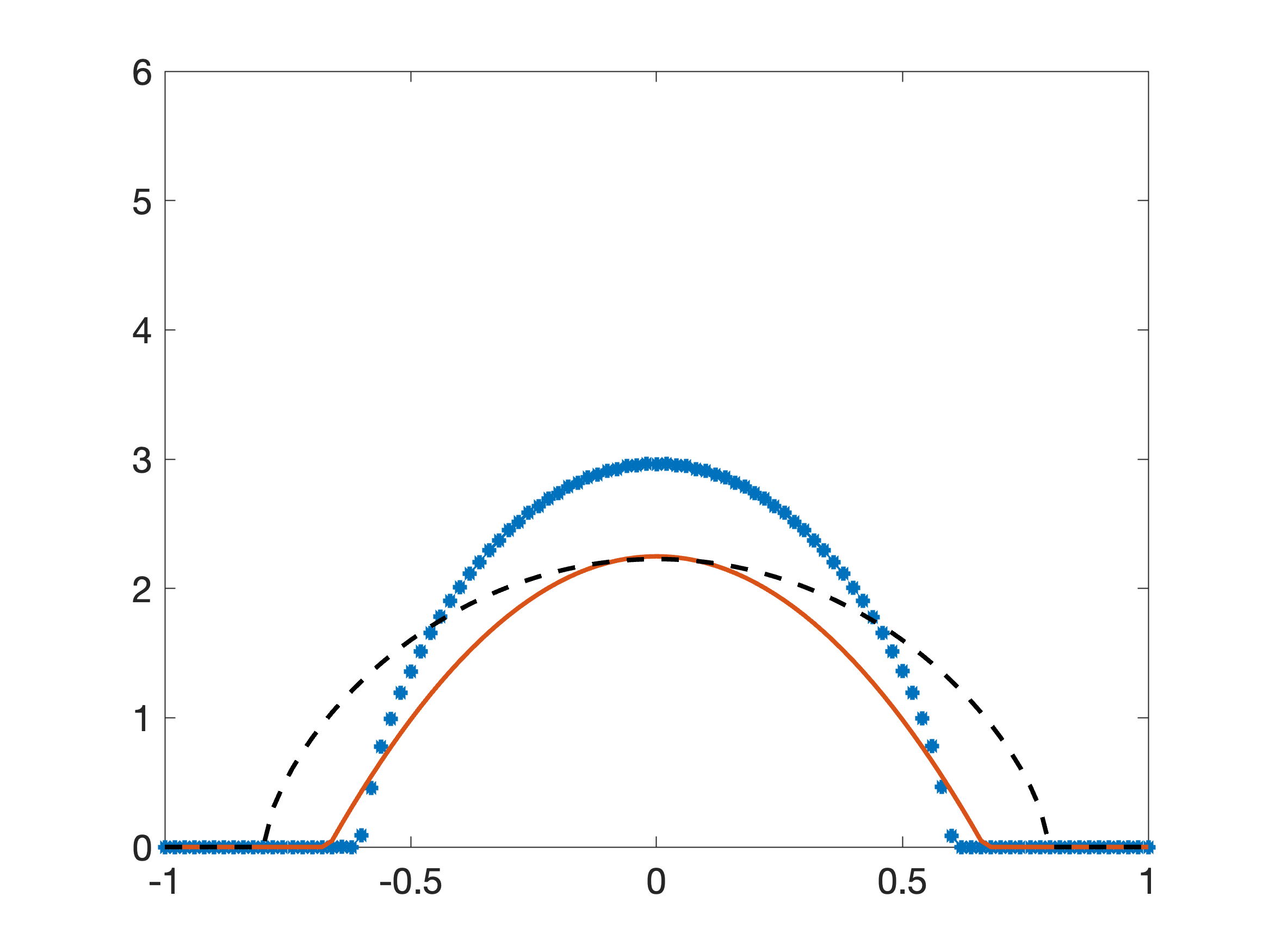}
    \caption{Barenblatt solution for $m^\dagger=2$ (solid red), perturbed $m=2.5$ (dashed black) and its \daJKO approximation (dotted blue) using the perturbed $m=2.5$, a shifted initial condition \eqref{eq:initial_shifted_space} with $\bar x=0.1$ and $\bar t=6\tau$ and expectation $B_h^1(u_h)$ and variance data $B_h^2(u_h)$ for the data term ($\vartheta=1$). The results are shown for $k=2^i$ for $i=0,\ldots,5$ from top left to bottom right. \label{fig:PME_data_no_noise_expected_var_m}}
\end{figure}

Therefore, let us assume that we have access to the mass of the ground truth Barenblatt profile. Incorporating the mass constraint via an observation operator similar to \eqref{eq:observation_disc_expectation} counteracts enforcing \eqref{eq:mass_const_LS}. 
Despite this obvious mismatch, a computational consequence is that the underlying optimization, Algorithm~\ref{alg:JKOStep}, does not converge anymore within $2\times 10^5$ iterations.
Therefore, we initialize the \daJKO scheme differently, by scaling the perturbed initial condition computed from \eqref{eq:Barenblatt}, \eqref{eq:initial_shifted_space} with $\bar x=0.1$, $\bar t=6\tau$ and $m=2.5$ such that its mass matches the mass of the ground truth. The results of this approach are displayed in Figure~\ref{fig:PME_data_no_noise_expected_var_mass_m}. In this case, the \daJKO scheme is able to match the ground truth much better.

\begin{figure}
    \centering
    \includegraphics[width=0.32\textwidth]{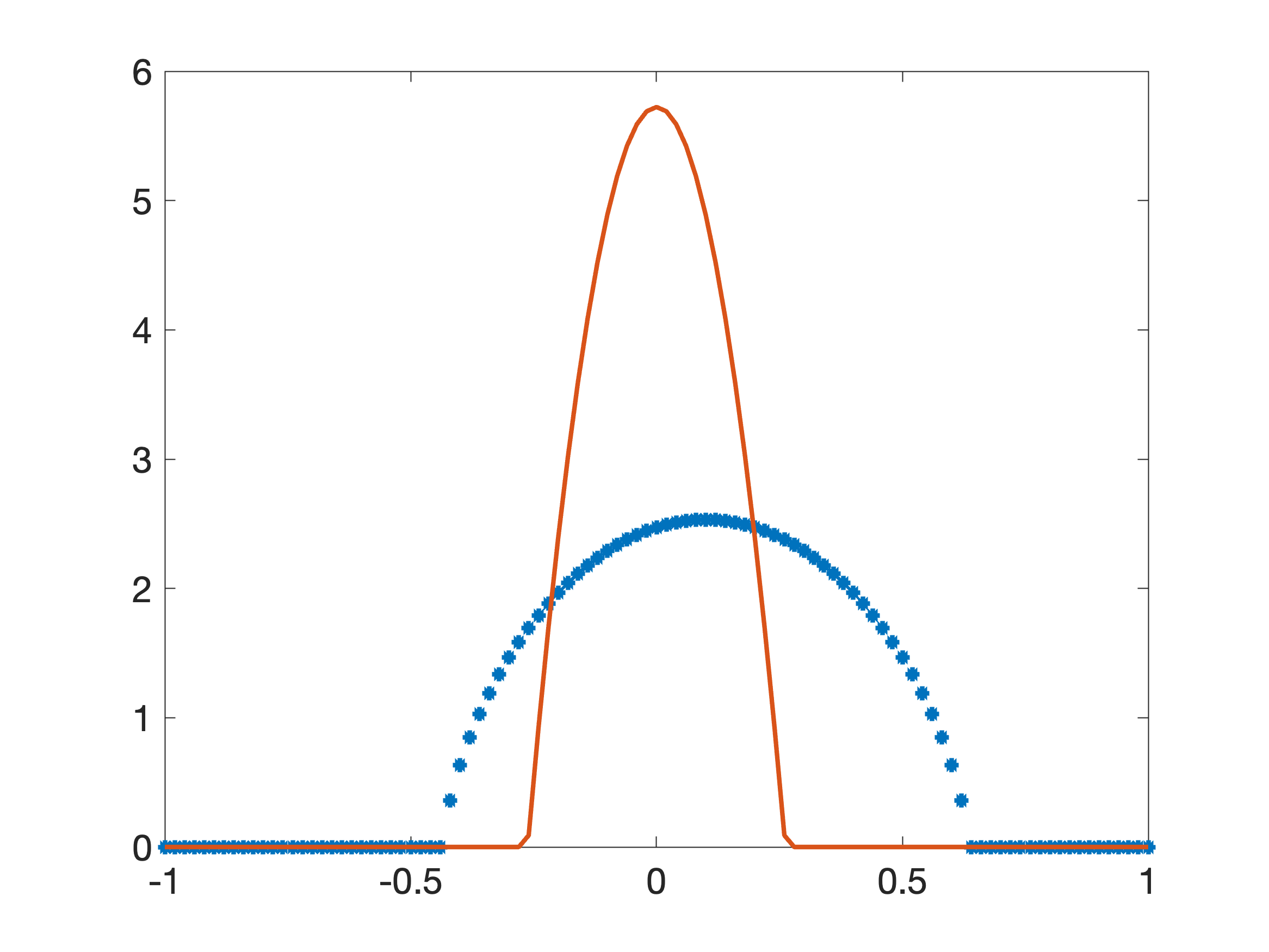}
    \includegraphics[width=0.32\textwidth]{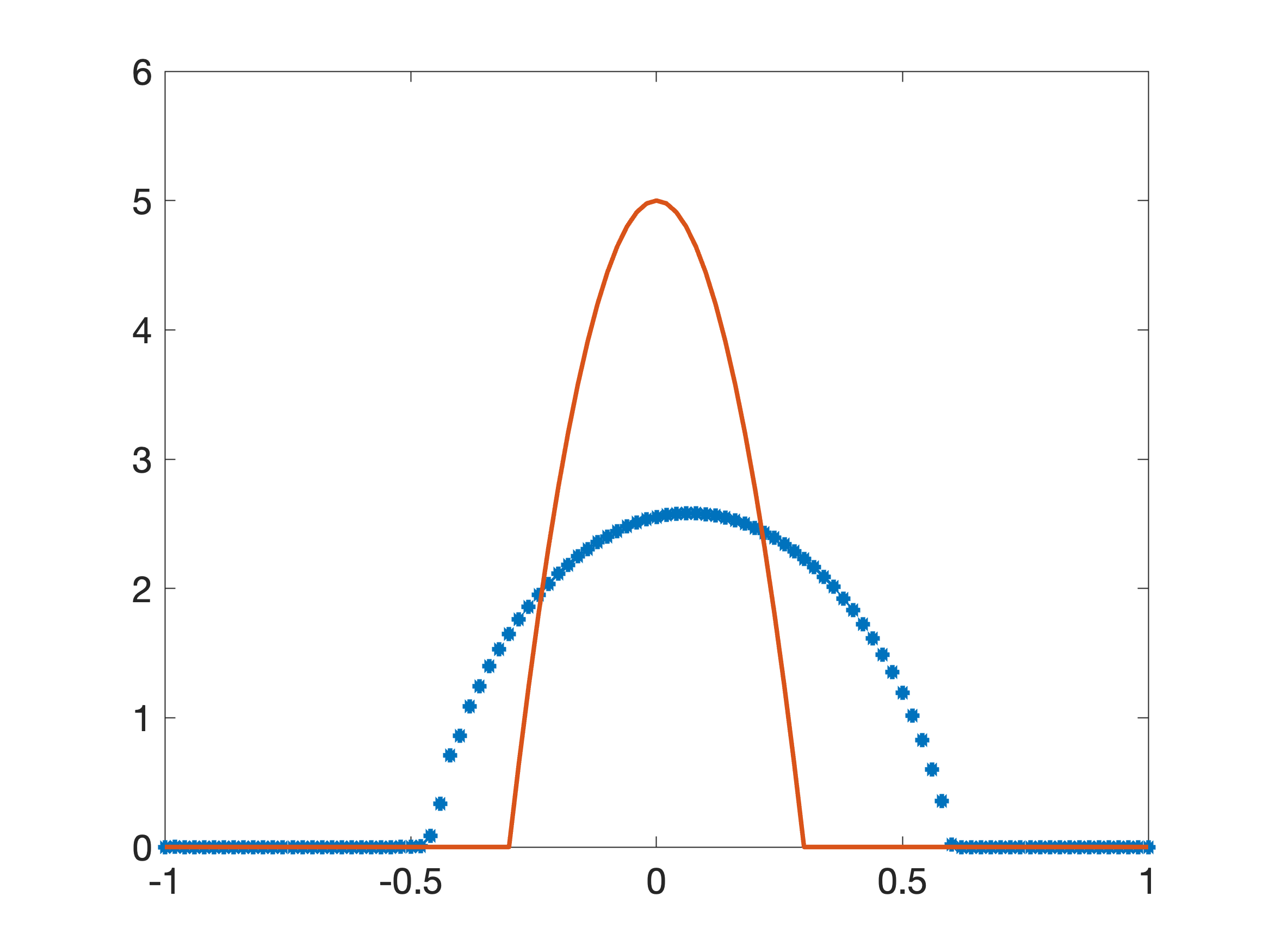}
    \includegraphics[width=0.32\textwidth]{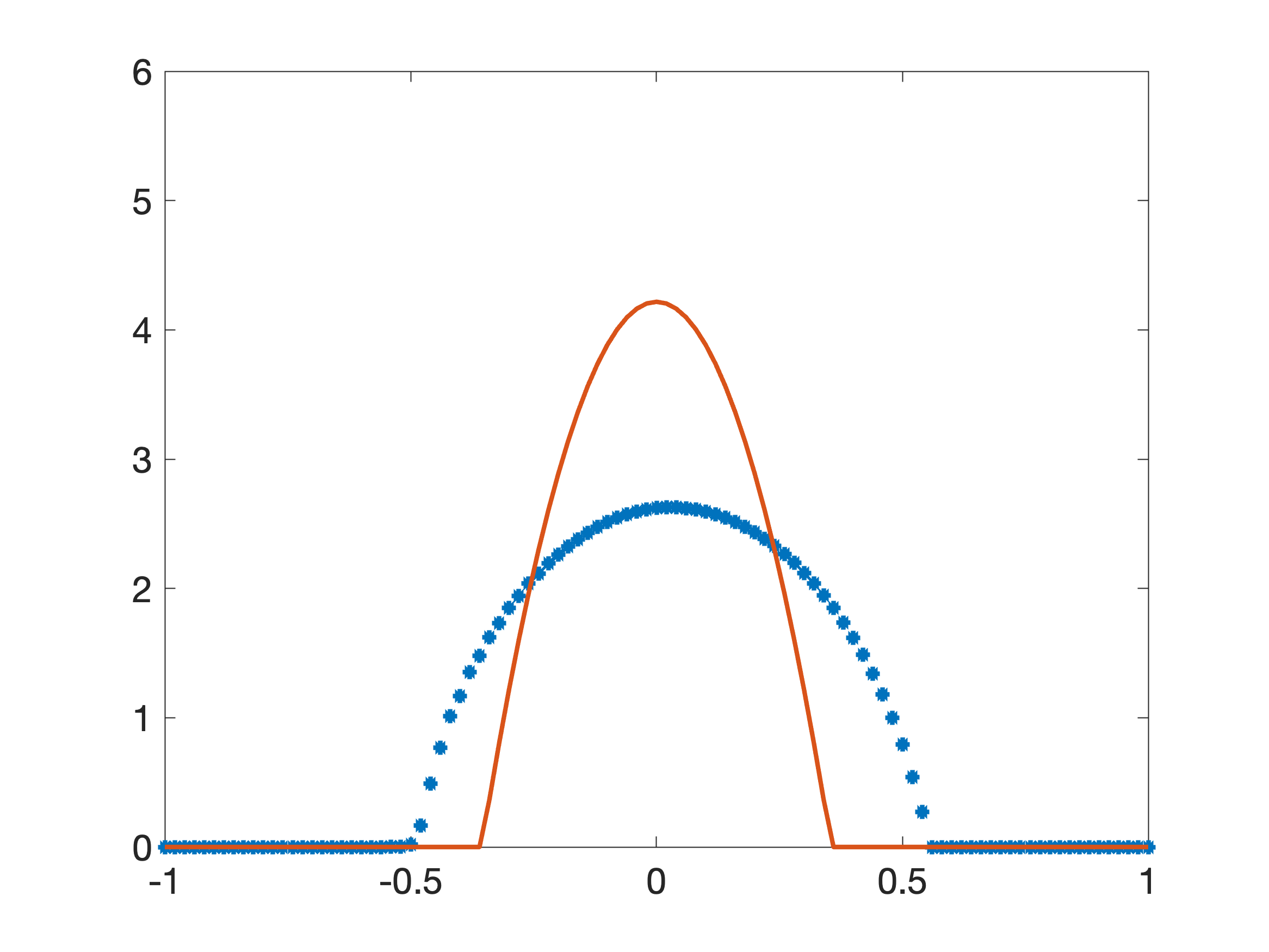}\\
    \includegraphics[width=0.32\textwidth]{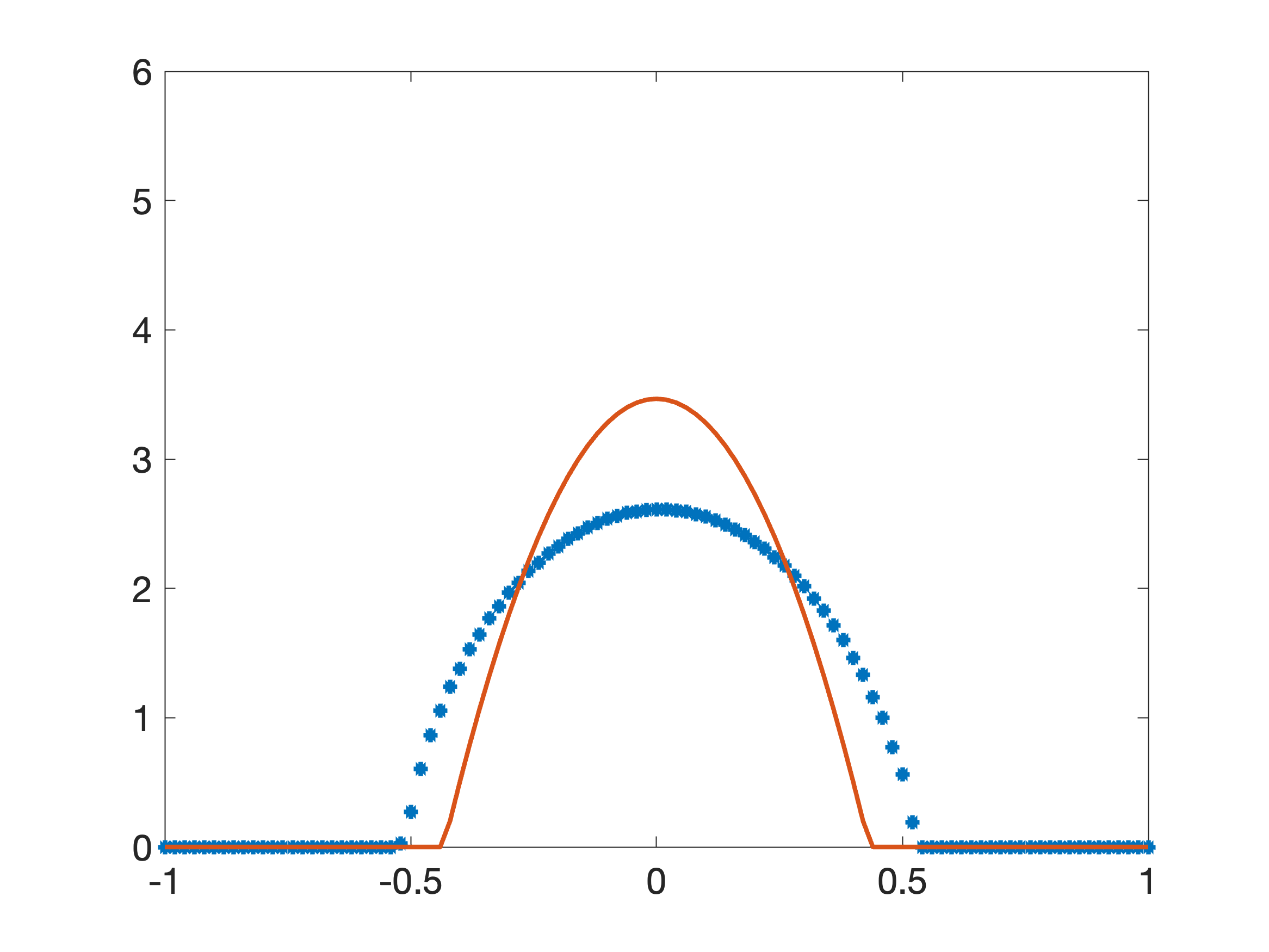}
    \includegraphics[width=0.32\textwidth]{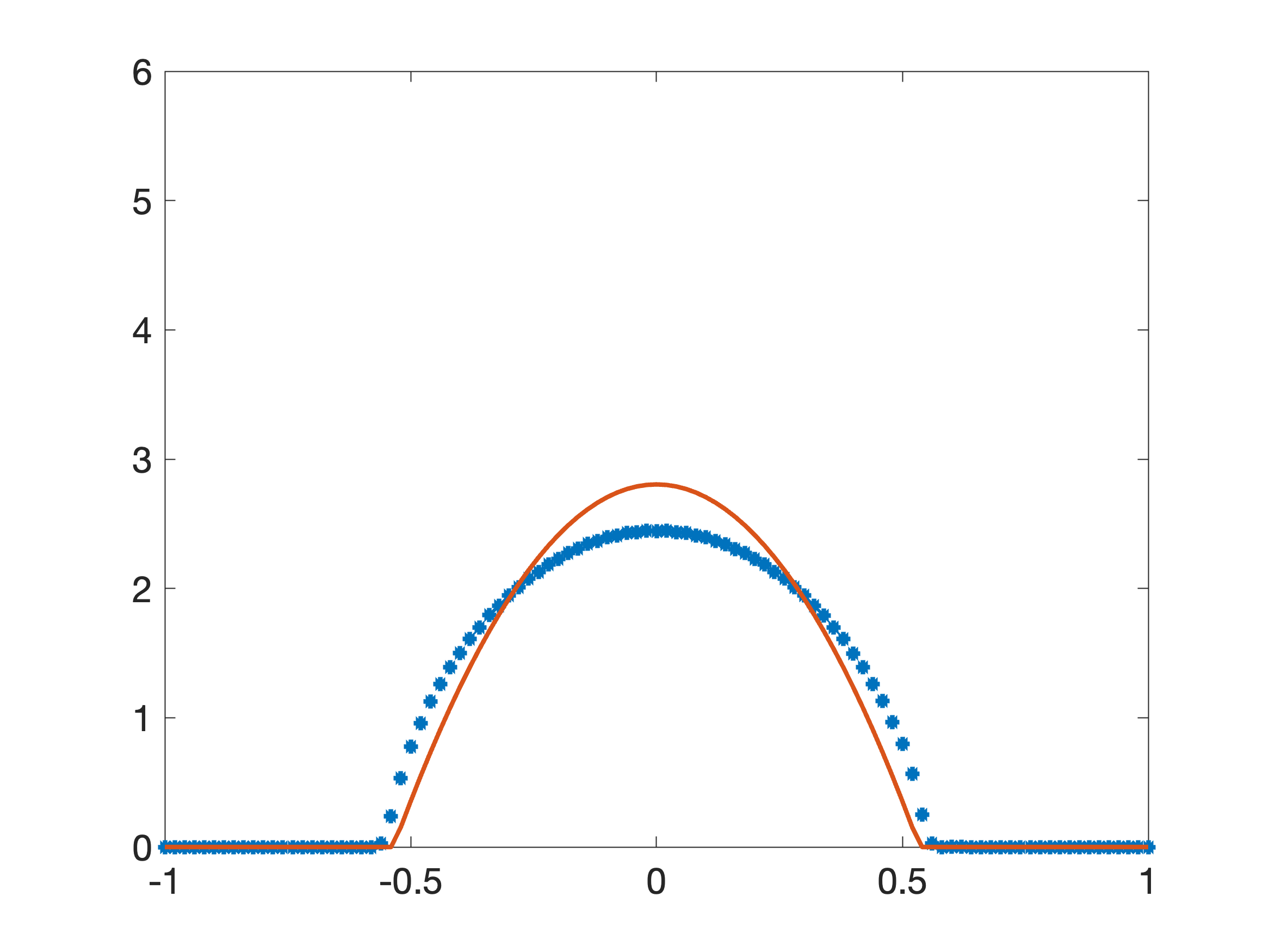}
    \includegraphics[width=0.32\textwidth]{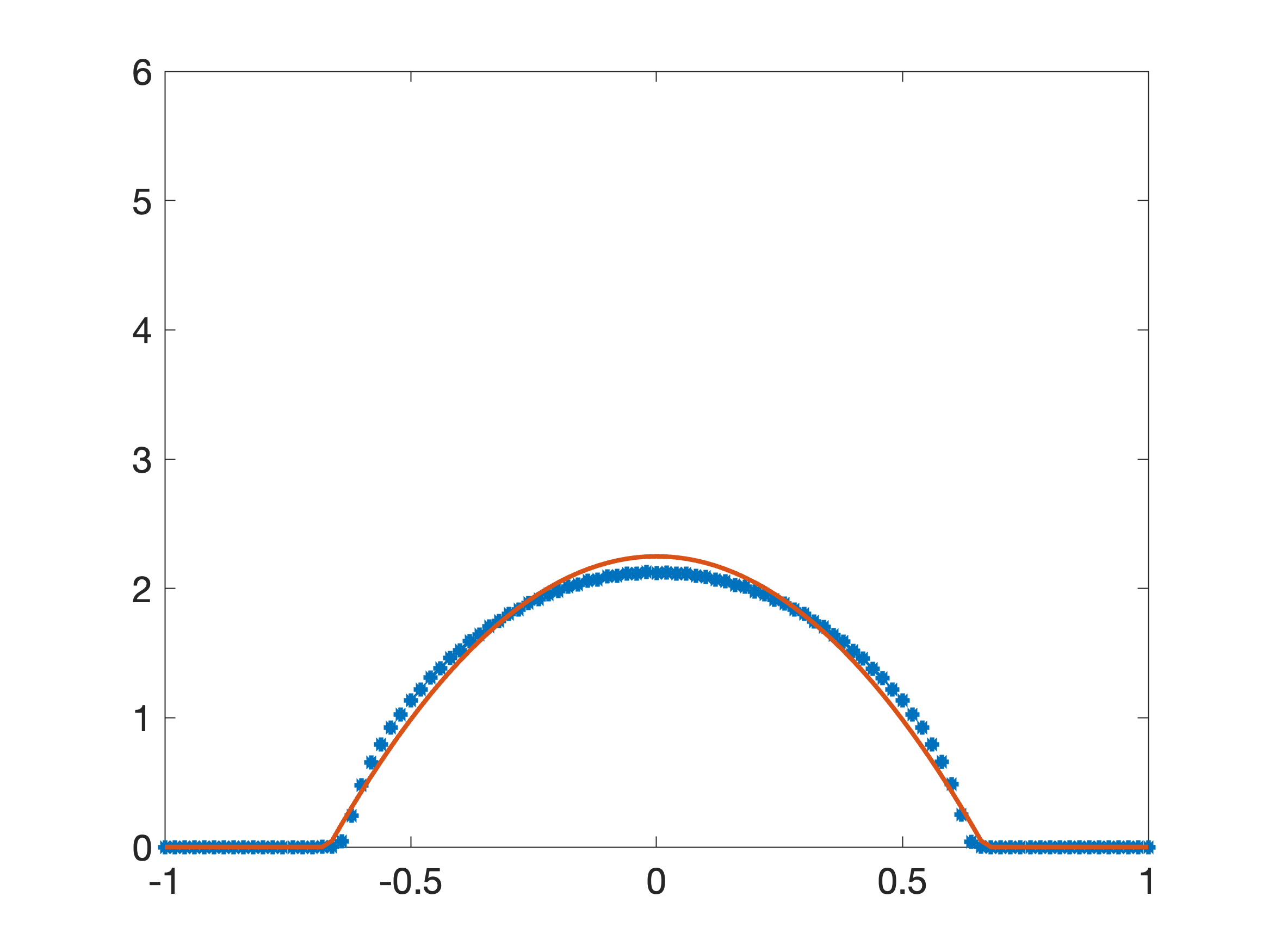}
    \caption{Barenblatt solution for $m^\dagger=2$ (solid red) and its \daJKO approximation (dotted blue) using the perturbed $m=2.5$, a shifted and scaled initial condition \eqref{eq:initial_shifted_space} with $\bar x=0.1$ and $\bar t=6\tau$ and expectation $B_h^1(u_h)$ and variance data $B_h^2(u_h)$ for the data term ($\vartheta=1$). The results are shown for $k=2^i$ for $i=0,\ldots,5$ from top left to bottom right. \label{fig:PME_data_no_noise_expected_var_mass_m}}
\end{figure}

\subsection{Chemotaxis with logarithmic kernel}
Our second example is the chemotaxis model presented in \cite{Blanchet2008_KellerSegelJKO}. Chemotaxis refers to the directed motion of bacteria in response to the gradient of a  chemical signal. Usually the bacteria move along the negative gradient of the chemical which, for example, drives them towards a food source. Denoting by $u=u(x,t)$ the density of bacteria and by $c=c(x,t)$ that of the chemical potential, the model reads as
\begin{equation}\label{eq:Chemotaxis}
\begin{aligned}
\partial_t u &=\Delta u-\chi \nabla \cdot[u \nabla c], & t>0, x \in \mathbb{R}^{d},\\
c(t, x)&=-\frac{1}{d \pi} \int_{\mathbb{R}^{d}} \log |x-y| u(t, y) \mathrm{d} y, & t>0, x \in \mathbb{R}^{d}, \\
u(0, x)&=u_0 \geq 0, & x \in \mathbb{R}^{d}.
\end{aligned}
\end{equation}
Again, this model constitutes a $2$-Wasserstein gradient flow with respect to the energy
\begin{align}
    F(u) = \int_\R^d u(x) \log (u(x))\;dx +\frac{\chi}{2d\pi}\int_{\R^d\times\R^d} \log |x-y|u(x)u(y)\;dxdy.
\end{align}
There has been a great interest in the mathematical analysis of chemotaxis models since many of them exhibit an interesting dichotomy: If their initial mass in under a given threshold, solutions exist for all time. If, on the other hand, the mass is large enough, a finite time blow-up occurs \cite{Blanchet2006}. For \eqref{eq:Chemotaxis}, the dichotomy also depends on the value of the sensitivity parameter $\chi$. Denoting by $M$ the mass on the initial datum, for $M\chi > 2d^2\pi$, finite time blow-up will occur while for $M\chi < 2d^2\pi$ solutions exist globally in time \cite{Calvez2010_log_blowup}.

In this example, we demonstrate that the \daJKO scheme can prevent blow-up. Therefore, we assume that the true value of $\chi=2$ and the initial condition is given by
\begin{align*}
    \rho_0(x) = G(x-1/3) + G(x+1/3)\quad\text{with } G(x)=\frac{1}{2 \pi\eta^2} e^{-|x|^2/\eta^2}
\end{align*}
and $\eta=1/5$. We have that $\int_\Omega\rho_0 \dx\approx 2.8$, and hence no blow-up occurs.
If, however, a wrong value of $\chi=10$ is used, the solution will blow-up, see Figure~\ref{fig:chemo}.
As in the previous section, we employ the expected value and the variance of the unperturbed solution ($\chi=2$) to prevent blow-up in the perturbed model ($\chi=10$); cf. \eqref{eq:observation_disc_expectation} and \eqref{eq:observation_disc_variance} for the definition of the observations. We change $\tau$ to $10^{-3}$, and keep the other parameters as in the previous section. Similar to \cite{Carrillo2022_PrimalDual}, we regularize the logarithmic convolution kernel by approximating it at $0$ by an integral average over two grid cells.
As can be seen from Figure~\ref{fig:chemo}, the \daJKO approximation does not blow-up in time. However, the \daJKO approximation does not fit the ground truth as good as in the previous section, which might be explained by the rather large perturbation in $\chi$ and the lack of sufficient data. This is in line with the analytical results in \cite{Calvez2010_log_blowup} which show that if the mass is above the critical mass, the second moment of the solution will eventually become negative which is a contradiction to the non-negativity of the solution. Thus, it makes sense that steering the second moment via measurements is able to prevent the blow-up.

\begin{figure}
    \centering
    \includegraphics[width=0.32\textwidth]{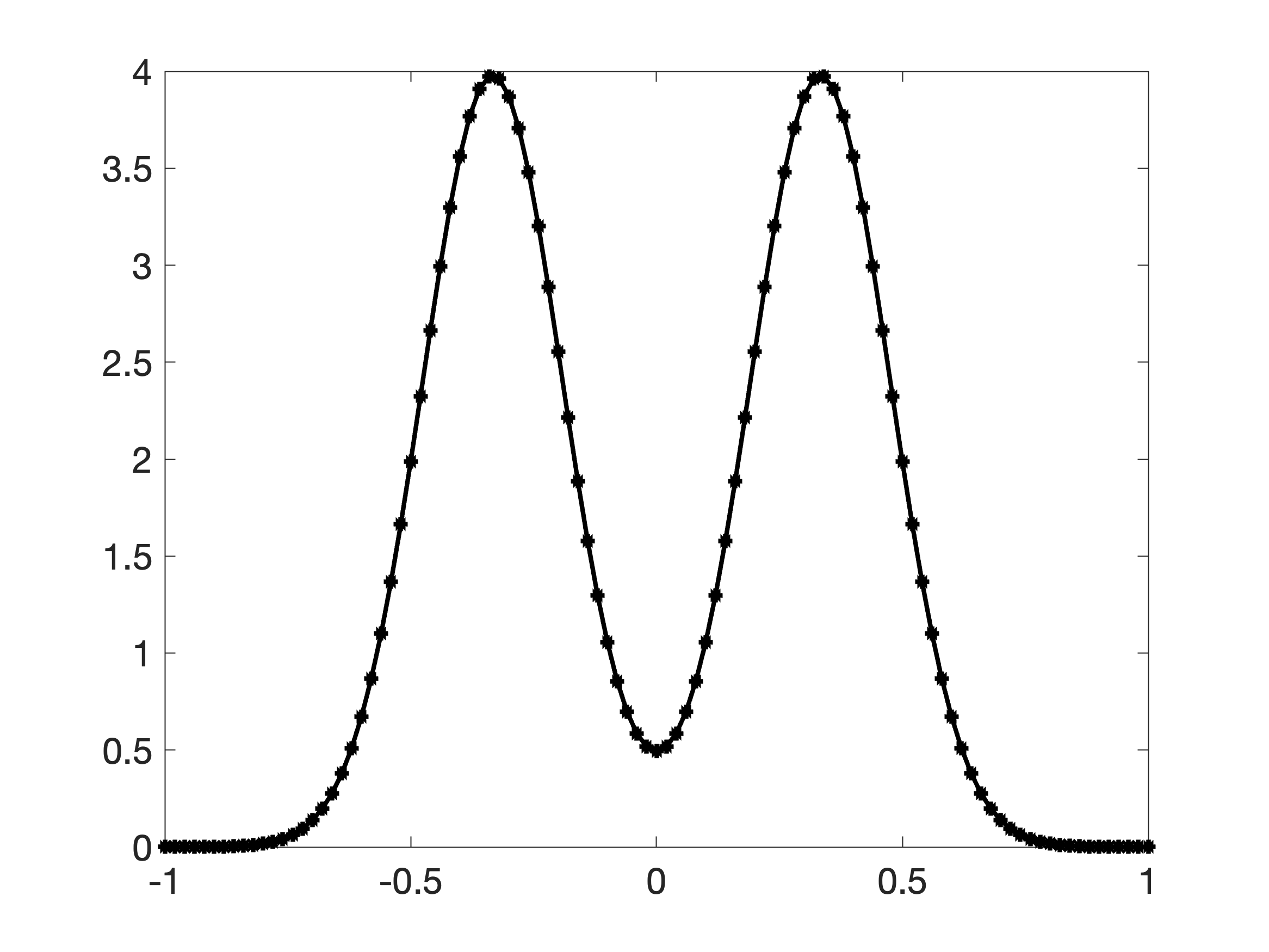}
    \includegraphics[width=0.32\textwidth]{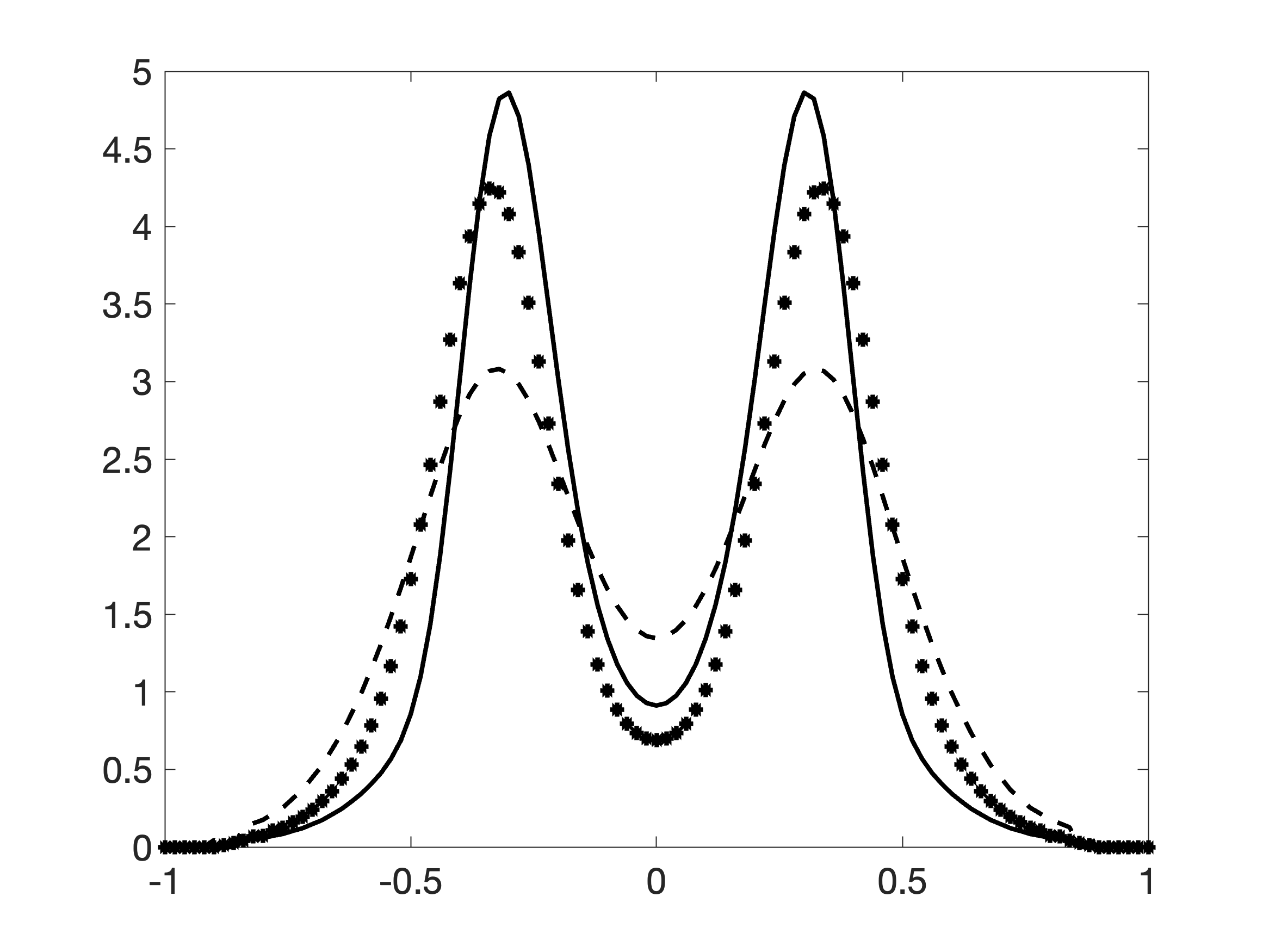}
    \includegraphics[width=0.32\textwidth]{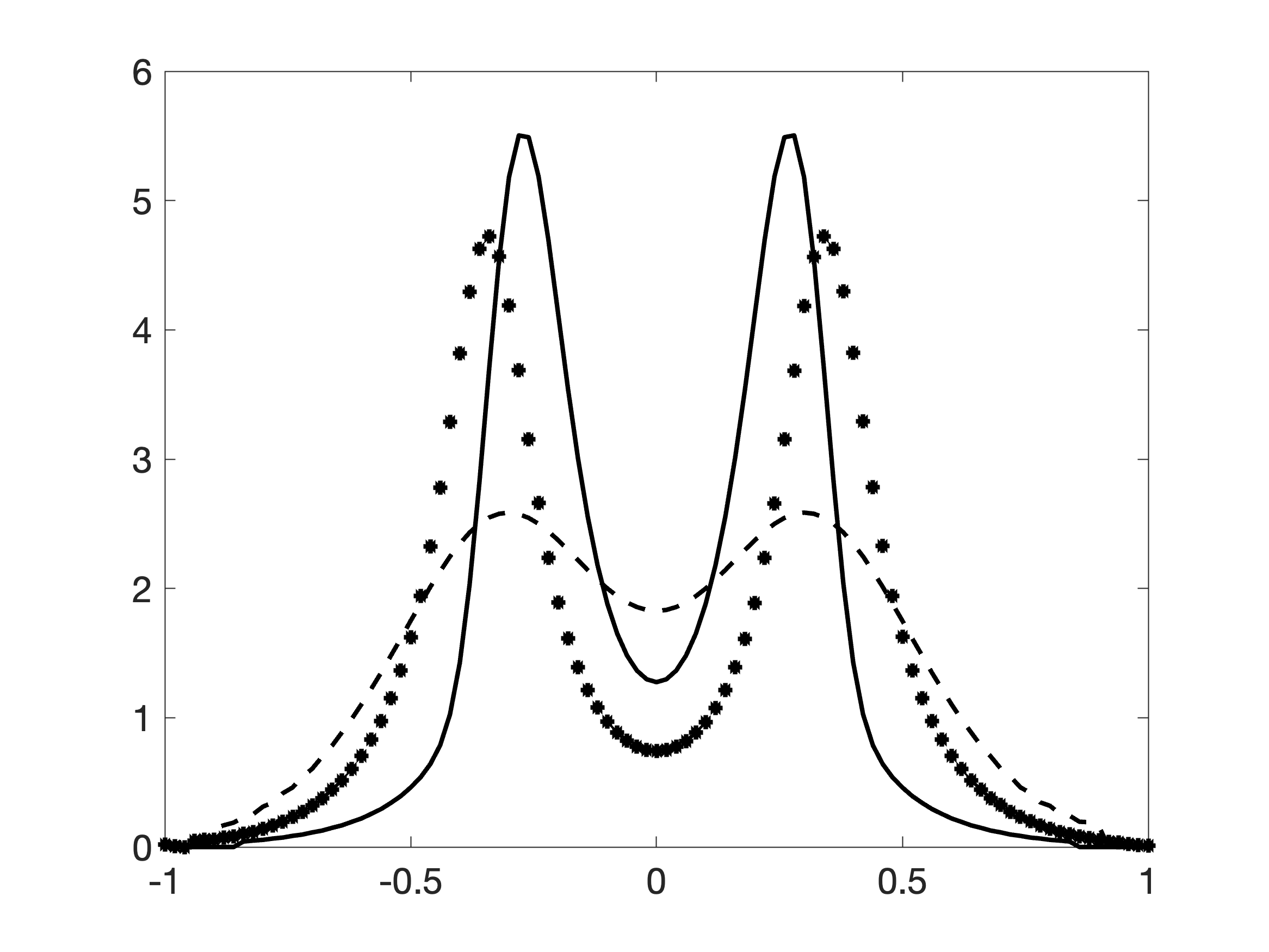}\\
    \includegraphics[width=0.32\textwidth]{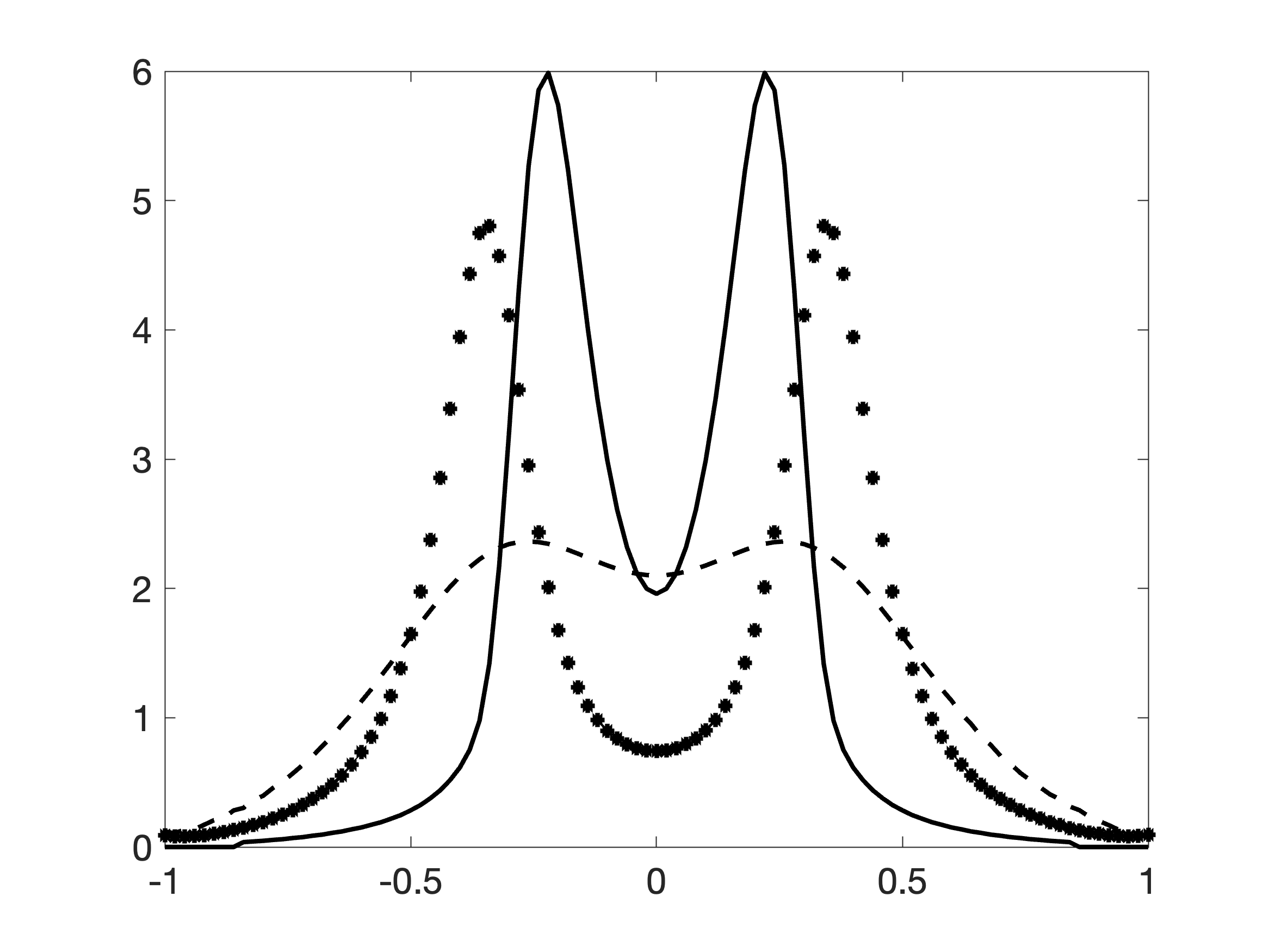}
    \includegraphics[width=0.32\textwidth]{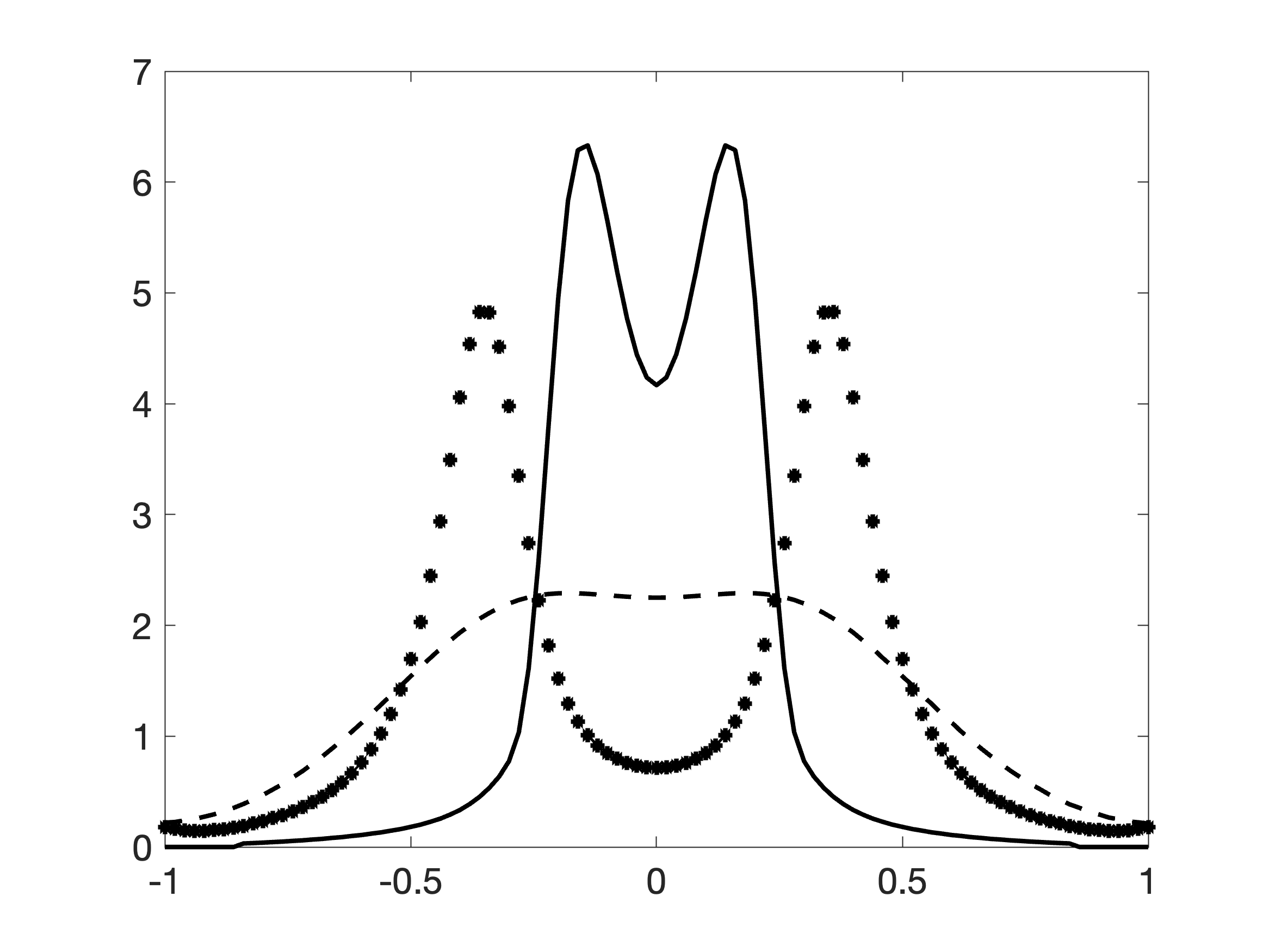}
    \includegraphics[width=0.32\textwidth]{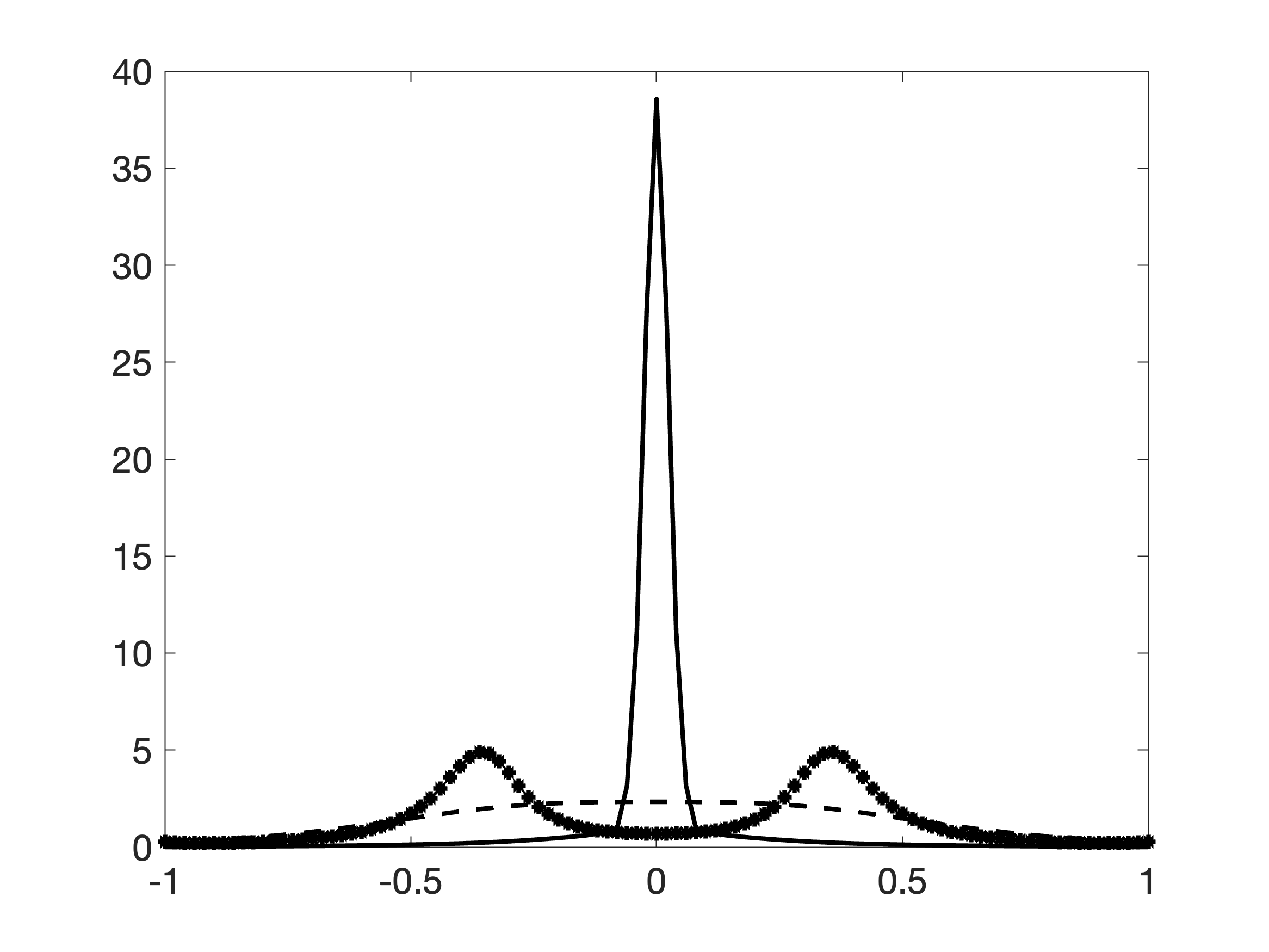}
    \caption{Solution to the chemotaxis equations \eqref{eq:Chemotaxis} for $\chi=2$ (dashed) and for $\chi=10$ (solid) for $t=k\tau$ with $k=10^{i-1}$, $i=0,\ldots,5$. We use different scalings of the axis. The solution of the daJKO scheme (dotted) remains bounded. \label{fig:chemo}}
\end{figure}

\begin{remark}
Model \eqref{eq:Chemotaxis} is a special case of the more general class of aggregation-diffusion equations of the form 
\begin{align*}
    \partial_t u = \nabla \cdot (\nabla u + u\nabla (K\ast u)),
\end{align*}
where $K: \R^d \to \R$ denotes a fixed convolution kernel. They are also Wasserstein GFs with respect to the energy
\begin{align}
    F(u) = \int_\Omega u\log(u)\;dx + \frac{1}{2}\iint_{\Omega \times \Omega}K(x-y)u(x)u(y)\;dxdy.
\end{align}
\end{remark}
\section{Outlook}
We presented a variational data assimilation approach based on a modified minimizing movement scheme together with a numerical scheme. We belief that the following extensions are interesting future research questions.
\begin{itemize}
    \item \emph{Non-linear mobilities and geometry identification}. Currently, we are able to treat equation of the form \eqref{eq:pde_grad_flow} which have a linear mobility $m(u)=u$. In many application, the mobility might also be non-linear which yields 
    \begin{align*}
        \partial_t u &= \nabla \cdot ( m(u) \nabla \frac{\delta F}{\delta u}),\quad \Omega \times (0,T).
    \end{align*}
    Here, $m$ is a concave function, e.g. $m(u)=u(1-u)$ which appears in models with volume filling \cite{Wrzosek2010_VolumeFilling}, see also \cite{EggPieSch2015a} for model discovery in this context. For such equations, a modified Wasserstein distance is needed which includes the mobility. They can be obtained by replacing the term $m^2/\rho$ in \eqref{eq:Benamou_Brenier} by 
    $$
    \frac{J^2}{m(\rho)},
    $$
    see \cite{Dolbeault_2008,Lisini2010} for details. This clearly requires a modification of the numerical scheme as the proximal operator changes. Independent of the context of this work, an  interesting question is whether it is possible to uniquely determine $m$ from measurements of geodesics with respect to the modified distance.
    \item \emph{Reaction terms and boundary conditions}. Another possible extension to a larger class of PDEs is to include reaction terms. This immediately results in the total mass being no longer conserved. Thus is is not feasible to work in the space of probability measures but one must use non-negative measures instead. In this case, the Wasserstein distance has to be replaced by the Fisher-Rao-type distance which allows for a change of mass, \cite{chizat2018_interpolating}. In the case of non-homogeneous boundary conditions, different modifications are necessary, see \cite{Humptert_PhD} for a formal discussion.
    \item \emph{Random effects}. So far, we remained in a purely deterministic setting, different to the classical approach in data assimilation. It would be interesting to included both random measurement as well as model errors.
    \item \emph{General geometries}. The present scheme is based on a finite difference approximation which makes it difficult to handle more complicated geometries when going to higher dimension. Here, finite volume, \cite{Cances2020_FVSchemeJKO} or finite element discretizations, \cite{Albi_2017}, for the divergence constraint could be used as long as the application of the proximal operators is computationally efficient.
\end{itemize}

\section*{Acknowledgements}
The authors thank G. Heinze (TU Chemnitz) for useful discussions.
\bibliographystyle{alpha}
\bibliography{sample}

\section*{Appendix}
%
In order to compute the larges real root $\rho^*$, we use a similar approach as in \cite{Carrillo2022_PrimalDual} and compute $\rho^*$ explicitly using Cardano's formula. For the convenience of the reader, we provide some detail. 
We start by introducing a new variables $z=x-\rho$, $p=\rho+\lambda$, and $c=\lambda|m|^2/2$. Hence, we look for the largest real root $z^*$ of the cubic equation
\begin{align*}
    z (z+p)^2=c.
\end{align*}
Clearly, $z^*\geq \max\{0,-p\}$, and $z^*$ is the only root in that range. In particular, $z^*+p\geq 0$, whence, $z^*$ is equivalently characterised by the solution of
\begin{align*}
    z+p=\sqrt{c/z},\qquad z>\max\{0,-p\}.
\end{align*}
Multiplying this equation by $\sqrt{z}$ and introducing $y=\sqrt{z}$ and $q=-\sqrt{c}$, we obtain the depressed cubic
\begin{align*}
    y^3 + pt + q=0.
\end{align*}
Defining the discriminant $\Delta=(p/3)^3+(q/2)^2$ and $C=(-q/2+\sqrt{D})^{1/3}$, the three roots are given by
\begin{align*}
    y_1= C-p/(3C),\quad y_2=\xi_1C-p/(3\xi_1C),\quad y_2=\xi_2 C -p/(3\xi_2 C),
\end{align*}
where $\xi_1=(-1+i\sqrt{3})/2$, and $\xi_2=(-1-i\sqrt{3})/2$ are two cube roots of $1$. Note that $C\neq 0$, since $\Re{\sqrt{D}}\geq 0$.

\end{document}